\documentclass[12pt]{amsart}
\usepackage{fullpage,booktabs, graphicx,psfrag,amsmath, amssymb,amsthm, hyperref, xcolor, tikz}

\newtheorem{theorem}{Theorem}[section]
\newtheorem{prop}{Proposition}[section]
\newtheorem{cor}{Corollary}[section]
\newtheorem{conj}{Conjecture}[section]
\newtheorem{lemma}[theorem]{Lemma}
\newtheorem{remark}{Remark}
\newtheorem{example}{Example}
\newtheorem{definition}{Definition}[section]

\newcommand{\R}{\mathbb{R}}

\newcommand{\Z}{\mathbb{Z}}
\newcommand{\N}{\mathbb{N}}
\newcommand{\Q}{\mathbb{Q}}
\newcommand{\inte}{{\mathop{\bf int}}}  
\newcommand{\USC}{{\mathrm{USC}}} 
\newcommand{\LSC}{{\mathrm{LSC}}} 

\newcommand{\conv}{\mathop{\bf conv}}
\newcommand{\cone}{\mathop{\bf cone}}

\newcommand{\Poisson}{\mbox{Poisson}}

\newcommand{\eg}{{\it e.g.}}
\newcommand{\ie}{{\it i.e.}}

\newcommand{\br}{\mathbf{br}}  

\title{Hamilton-Jacobi scaling limits of Pareto peeling in 2D}
\author{Ahmed Bou-Rabee}

\author{Peter S. Morfe}

\begin{document}

	\begin{abstract}
	Pareto hull peeling is a discrete algorithm, generalizing convex hull peeling, for sorting points in Euclidean space. We prove that Pareto peeling of a random point set in two dimensions has a scaling limit described by a first-order Hamilton-Jacobi equation and give an explicit formula for the limiting Hamiltonian, which is both non-coercive and non-convex.  This contrasts with convex peeling, which converges to curvature flow.  The proof involves direct geometric manipulations in the same spirit as Calder (2016).
	\end{abstract}
\maketitle


\section{Introduction}
\subsection{Overview}
	Consider $\R^2$ equipped with a norm $\varphi(\cdot)$ and let $A$ be a finite subset of $\R^2$. A point $x\in \R^2$ is in the {\it Pareto hull} of $A$
	if, for every $y \in \R^2 \setminus \{a\}$, there exists $a \in A$ such that $\varphi(a-x) < \varphi(a-y)$. The Pareto hull peeling process proceeds by repeatedly taking the Pareto hull, $\mathcal{P}(A)$, and removing points on its boundary:
	\begin{equation}
	E_1(A) = \mathcal{P}(A) \quad \mbox{and} \quad E_{k+1}(A) = \mathcal{P}(A \cap \inte(E_k(A))).
	\end{equation}
	When the unit ball of $\varphi(\cdot)$ has no flat spots, the Pareto hull coincides with the convex hull \cite{thisse1984some}. Calder-Smart showed that convex hull peeling of points drawn (independently) at random converges to curvature flow as the number of points goes to infinity \cite{calder2020limit}. Here we consider the more general case and find that whenever the unit ball of $\varphi(\cdot)$ has a facet,  the scaling limit of Pareto hull peeling solves a first-order Hamilton-Jacobi equation. Hence the facets lead to faster, `ballistic' motion contrasting with the strictly convex case which has a slower, `diffusive' limit. Higher dimensional analogues are discussed in Section \ref{sec:conclusion}.

	\begin{figure}[h!]
	\includegraphics[width=0.24\textwidth]{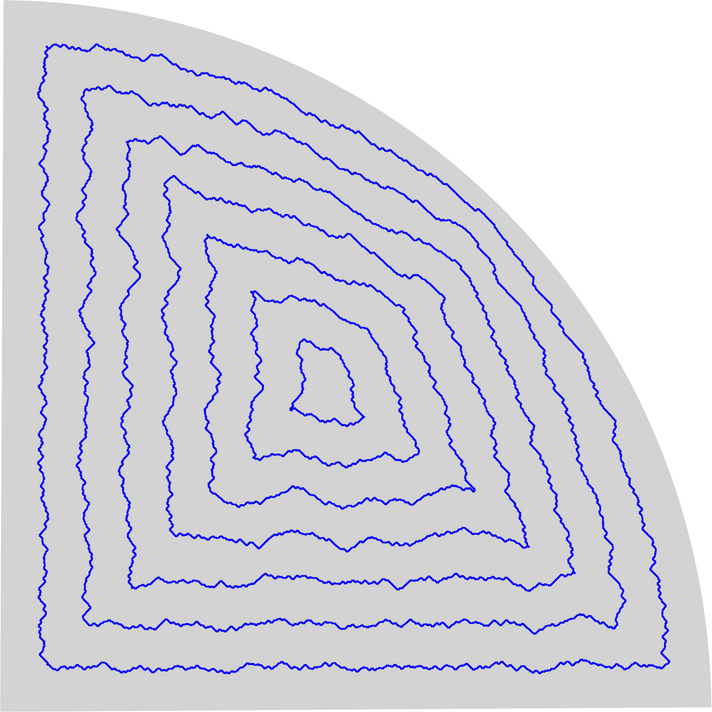} \quad
	\includegraphics[width=0.26\textwidth]{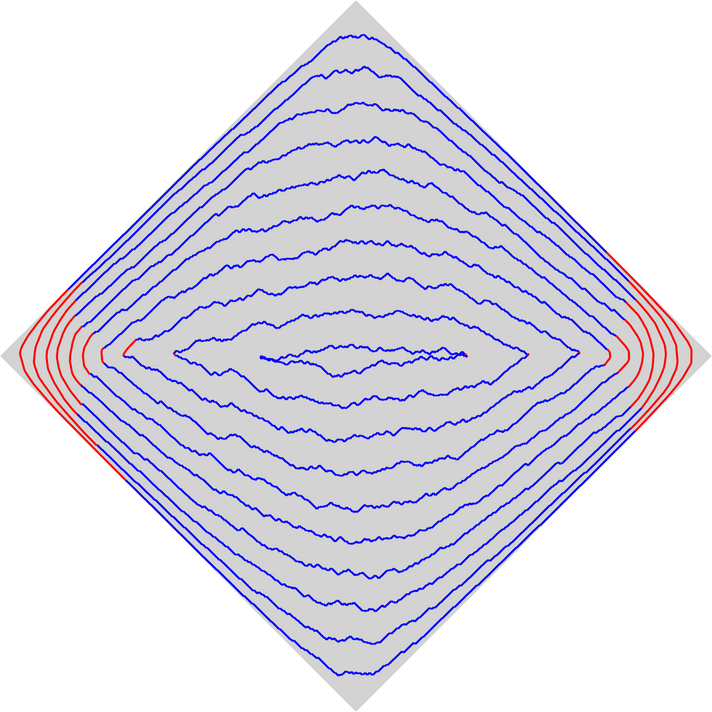}  \quad
	\includegraphics[width=0.24\textwidth]{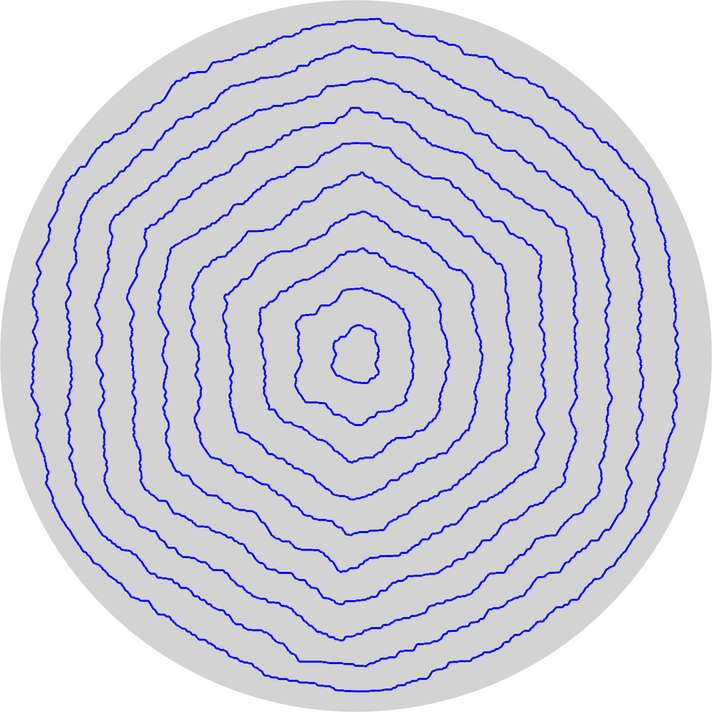}  
	\caption{Peels, $E_k$, for $k$ a multiple of $30$, of Pareto peeling of homogeneous Poisson clouds in the shaded domains, with respect to various $\varphi$ displayed in Figure \ref{fig:norm-balls}. Peels are colored blue if `constrained' by a facet and red otherwise --- this is made precise in Section \ref{sec:pareto-hulls}.} \label{fig:pareto-peeling}
	\end{figure}
	
	\begin{figure}[t]
	\includegraphics[width=0.15\textwidth]{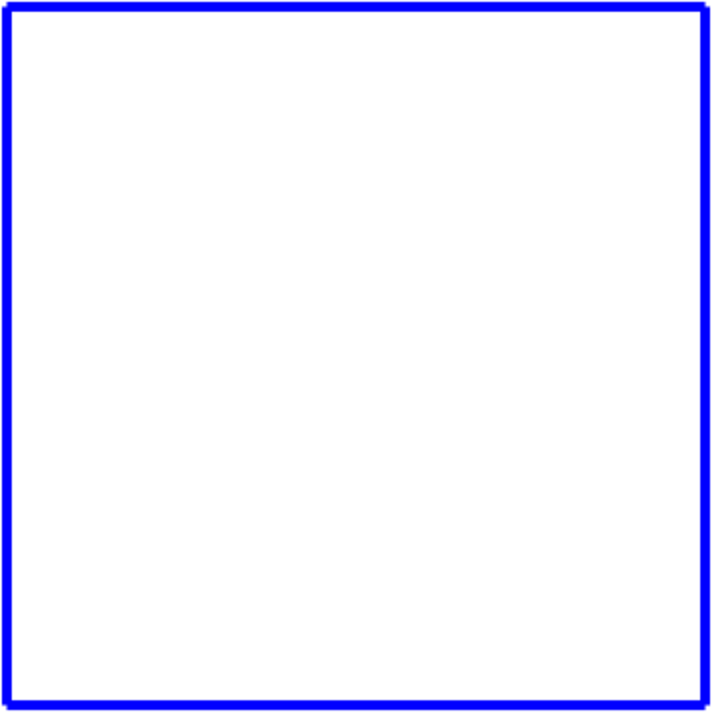}   \qquad
	\includegraphics[width=0.15\textwidth, height = 0.15\textwidth]{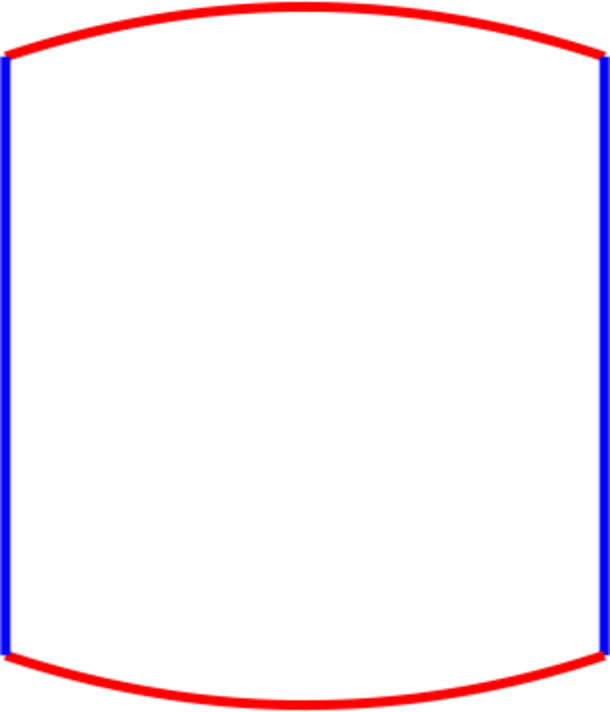}  \quad
	\includegraphics[width=0.18\textwidth]{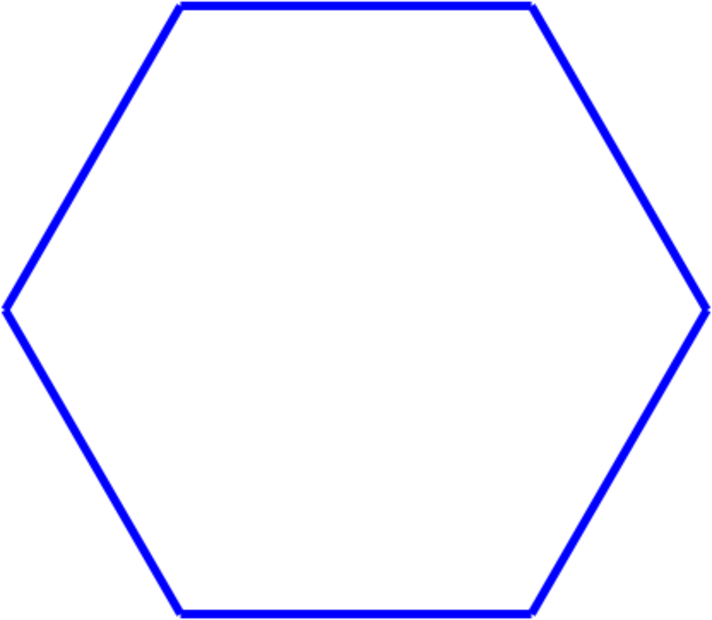}  \quad
	\caption{Unit balls, $\{ \varphi \leq 1\}$,  with flat edges outlined in blue and `round' edges in red.} \label{fig:norm-balls}
	\end{figure}

	\subsection{Background}
The term Pareto hull or Pareto envelope originates from computer science \cite{corne2000pareto, nouioua2005enveloppes}; 
however, these hulls were studied much earlier under the name {\it sets of strictly efficient points} in (what is now known as) the field of location analysis \cite{laporte2019introduction, smith2009locational}. Briefly, location analysis is a specialized branch of combinatorial optimization which studies the `best location' for a set of `facilities' under various constraints. 

The so-called {\it Fermat-Weber} or {\it point objective} problem aims to determine the location 
of a single facility which minimizes the distance to a finite number of demand points, \eg,  deciding where to build a factory
serving multiple customers. Since this is a multi-objective problem, there are several ways to define optimal --- 
one way to do so is with the Pareto hull. (Other definitions, in the context of this paper, are discussed in Section \ref{sec:conclusion}.)

Under this notion of optimality, Kuhn showed that when the chosen distance metric is Euclidean, the set of optimal solutions to the Fermat-Weber problem lie in the convex hull of the set of demand points \cite{kuhn1967pair, kuhn1973note}. In particular, the Pareto hull coincides with the convex hull in this case. Following earlier work of Ward-Wendell, Thisse-Ward-Wendell extended this characterization in two dimensions to any distance induced by a norm with strictly convex unit ball \cite{ward1985using, thisse1984some}. 
Geometric properties of the Pareto hull and general algorithms to compute them appear in papers by Ndiaye-Michelot \cite{ndiaye1997, ndiaye1998efficiency}, Durier \cite{durier1987sets, durier1990pareto},
Durier-Michelot \cite{durier1985geometrical, durier1994set},
and Pelegrin-Fernandez \cite{pelegrin1988determination, pelegrin1989determination}.

Notably, Durier-Michelot \cite{durier1986sets} present a beautiful and deep characterization of the Pareto hull 
in terms of supporting cones --- this generalizes the halfspace description of convex hulls. 
We use this to extend the dynamic programming principle for convex hull peeling
\cite{calder2020limit} to Pareto peeling. In fact, the scaling limits in this paper may be thought of as continuum versions
of this dynamic programming principle. 

The limiting equations we derive are closely related to the continuum limit of nondominated sorting, proving a conjecture of Calder \cite{calderminicourse}. 
Briefly, nondominated sorting is an algorithm for sorting points in Euclidean space according to the coordinatewise partial order.
Calder-Esedoglu-Hero showed that nondominated sorting of random points has a scaling limit described by an explicit Hamilton-Jacobi equation \cite{calder2014continuum,calder2014hamilton,calder2015pde,calder2017numerical}. Recently Calder-Cook established a rate of convergence to this continuum limit \cite{cook2022rates}
--- it would be interesting to adapt those ideas to Pareto peeling.

\subsection{Main result}

Our convergence result is captured via the {\it height function} of Pareto hull peeling, 
\begin{equation} \label{eq:height}
u_A = \sum_{k \geq 1} 1_{\inte(E_k(A))}.
\end{equation} 
For simplicity, we model our random data via a Poisson process \cite{kingman1992poisson}, $X_{n f}$, of intensity $n f$ in an open set $\mathrm{U}$.  Our only restriction on $f$ is that it is a bounded, strictly positive, continuous function in $\mathrm{U}$. We will write $u_n := u_{X_{nf}}$, and rescale by  $\bar{u}_n(x) := n^{-1/2} u_n(x)$. 

We require that the domain $\mathrm{U}$ be a bounded, open {\it Pareto efficient} subset of $\R^2$, that is, a set for which $\inte(\mathcal{P}(\bar{\mathrm{U}})) = \mathrm{U}$.
 We further assume that $\mathrm{U}$ is `compatible' with $\varphi$. The definition of compatibility is somewhat technical and will be given in Definition \ref{def:compatible} of Section \ref{subsec:domain-consistency} below. Importantly, we later indicate some necessity of this condition via an explicit counterexample.  For now, we note that when the unit ball of $\varphi$ is a polygon, any convex set is both Pareto efficient and compatible. 
In fact, it is useful to note that convex sets are always Pareto efficient (Lemma \ref{lemma:convexity-and-efficient}), but the two notions are not equivalent; see Figure \ref{fig:non-convex-example}.

In our main result, we assume that $\varphi$ is a norm in $\R^{2}$ for which the unit ball $\{\varphi \leq 1\}$ is not strictly convex, or, more precisely:
	\begin{align}
		&\text{The unit ball} \, \, \{\varphi \leq 1\} \, \, \text{has at least one boundary facet.}\label{eq:linear-assumption}
	\end{align}
An easy example is when $\{\varphi \leq 1\}$ is a polygon, but many more complicated shapes are also possible.

\begin{theorem}\label{theorem:pareto-convergence}
If $\varphi$ satisfies \eqref{eq:linear-assumption} and $\mathrm{U}$ is a bounded, open Pareto efficient set in $\R^{2}$ that is compatible with $\varphi$ (see Definition \ref{def:compatible} below), then, on an event of probability 1, the sequence of rescaled height functions $(\bar{u}_{n})_{n \in \mathbb{N}}$ converges uniformly in $\bar{\mathrm{U}}$
to the unique viscosity solution $\bar{u}$ of the PDE:
\begin{equation} \label{eq:eikonal}
\begin{cases}
\bar{H}_\varphi(D\bar{u}) = f & \mbox{ in $\mathrm{U}$}, \\
\bar{u} = 0 &\mbox{ on $\partial \mathrm{U}$}.
\end{cases}
\end{equation}
Here $\bar{H}_\varphi(\cdot)$ is a non-negative, continuous Hamiltonian that only depends on $\{\varphi \leq 1\}$ and is neither convex nor coercive (see \eqref{eq:effective-hamiltonian} below for the formula).

\end{theorem}

\begin{figure}
		\includegraphics[width=0.25\textwidth]{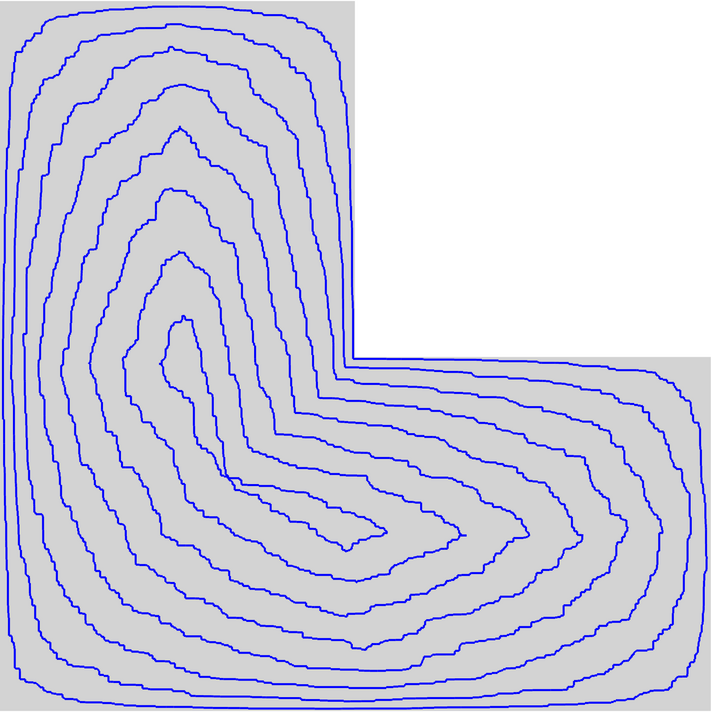} \qquad		\includegraphics[width=0.25\textwidth]{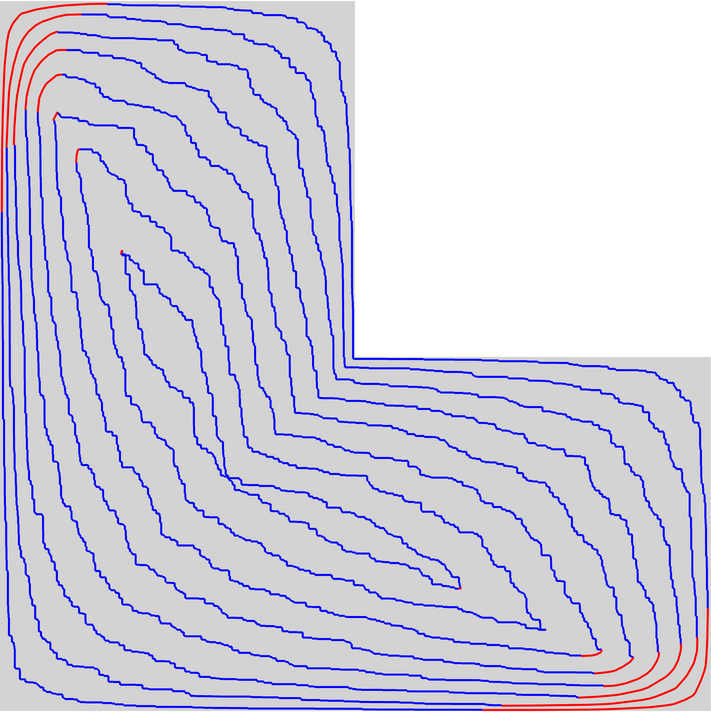}
		\caption{Pareto peeling of homogeneous Poisson clouds in the shaded domains with respect to the two norms in Example \ref{example:hamiltonian-l1-norm}. The domains are Pareto efficient with respect to these norms, but not convex.}
	\label{fig:non-convex-example}
\end{figure}

In the theorem, assumption \eqref{eq:linear-assumption} is a necessary condition.  When \eqref{eq:linear-assumption} fails, $\{\varphi \leq 1\}$ is strictly convex and classical results in location analysis imply that the Pareto hull is nothing but the convex hull.  In this case, Calder-Smart \cite{calder2020limit} have already shown that convex hull peeling converges, but with the
larger $n^{\frac{2}{3}}$ rescaling identified earlier by Dalal \cite{dalal2004counting}. As a natural byproduct of our arguments, we give a self-contained proof that the $n^{\frac{1}{2}}$ scaling of convex hull peeling is trivial.

\begin{cor} \label{cor:convex-peeling}  If $\{\varphi \leq 1\}$ is strictly convex, \ie, if \eqref{eq:linear-assumption} does not hold, then $\bar{u}_{n} \to \infty$ locally uniformly in $\mathrm{U}$.  \end{cor}

Finally, where the compatibility assumption is concerned, we prove that convergence may fail if it does not hold.

\begin{cor} \label{cor:compatible-counterexample} There is a norm $\varphi$ satisfying \eqref{eq:linear-assumption} and an open, bounded Pareto efficient set $\mathrm{U}$ in $\R^{2}$ that is \emph{not} compatible with $\varphi$ and for which the rescaled height functions do not converge uniformly to a continuous function in $\bar{\mathrm{U}}$.  \end{cor}

Our proof explicitly identifies the effective Hamiltonian $\bar{H}_{\varphi}(\cdot)$ in Theorem \ref{theorem:pareto-convergence}. The full description  of $\bar{H}_{\varphi}$ will be postponed till Section \ref{sec:pareto-hulls}.  For now, we give an example that already demonstrates the main qualitative features of these functions.  When $\varphi$ is the $\ell^{1}$ norm, this confirms a conjecture of Calder \cite{calderminicourse} and reflects the fact that nondominated sorting partly describes the local behavior of Pareto peeling. 
\begin{example}\label{example:hamiltonian-l1-norm}
For $p' \in \R^2$ when $\varphi(x) = |x_1| + |x_2|$,  the effective Hamiltonian is
\[
\bar{H}_{\varphi}(p') = |p'_1 p'_2|.
\]
If instead $\varphi(x) =  \max( |x_1 - x_2|, \|x\|_2)$, then we have
	\[
	\bar{H}_{\varphi}(p') = \max( p'_1 p'_2, 0).
	\]
	The unit balls for the prior two norms are displayed in Figure \ref{fig:norm-pics}. 
\end{example}

\subsection{Method of Proof}  There are three main steps in the proof of Theorem \ref{theorem:pareto-convergence}.  First we prove that the height function is determined by a dynamic programming principle (DPP).  The full description of the DPP is given in Section \ref{subsec:norm-partition}. Roughly, there is a family of cones $\{Q_{p}\}_{p \in \mathcal{N}^{*}}$ and a family of halfspaces $\{H_{q^{\perp}}\}_{q \in \mathcal{E}}$ 
such that the height function for Pareto peeling of a finite set of points $A$ is given by
	\begin{equation*}
		u_{A}(x) = \min \left\{ \inf_{p \in \mathcal{N}^{*}} \sup_{y \in x + \inte(Q_{p})} u_{n}(y) + 1_{A}(y), \inf_{q \in \mathcal{E}} \sup_{y \in x + \inte(H_{q^{\perp}})} u_{n}(y) + 1_{A}(y) \right\}.
	\end{equation*}
The proof of this fact uses a geometric characterization of Pareto hulls from location analysis, recalled as Theorem \ref{thm:cone_efficient} below.

This DPP is a generalization of the DPPs appearing in nondominated sorting and convex hull peeling and serves as a starting point for the analysis.   In fact, the structure of the limiting Hamiltonian  $\bar{H}_{\varphi}$ mirrors that of the DPP, 
	\begin{equation*}
		\bar{H}_{\varphi}(\xi) = \max \left\{ \sup_{p \in \mathcal{N}^{*}} \frac{\langle \xi,v_{p} \rangle \langle \xi, w_{p} \rangle}{|v_{p} \times w_{p}|}, 0 \right\},
	\end{equation*}
where $\{(v_{p},w_{p}) \, \mid \, p \in \mathcal{N}^{*}\}$ are certain vectors associated to the cones $\{Q_{p}\}_{p \in \mathcal{N}^{*}}$.  In particular, the cones of the DPP are in one-to-one correspondence with the quadratic terms, while the influence of the halfspaces $\{H_{q^{\perp}}\}_{q \in \mathcal{E}}$, represented by the zero appearing in the supremum, becomes trivial in the limit.  Due to the vanishing contribution of
the halfspaces, the Hamiltonian is always non-convex and non-coercive, and it becomes necessary to separately analyze the nondegenerate directions ($\xi$ with $\bar{H}_{\varphi}(\xi) > 0$) and degenerate directions ($\xi$ with $\bar{H}_{\varphi}(\xi) = 0$).  The presence of degenerate directions is a fundamental difference between Pareto hull peeling and convex hull peeling, as in the latter the limiting PDE is isotropic.

The analysis of nondegenerate directions (Sections \ref{sec:subsolution} and \ref{sec:supersolution}) is the second step of the proof.  Here we argue that, locally near a point where $\bar{H}_{\varphi}(D\bar{u}) > 0$ holds, an affine transformation of nondominated sorting describes the behavior of the height function. This is achieved via a geometric ``direct verification" argument that builds on the techniques of \cite{calder2016direct}.  

The third step (Section \ref{sec:degenerate-directions}) is to show that $D\bar{u}$ never points in a degenerate direction, or, more precisely, $\bar{H}_{\varphi}(D\bar{u}) > 0$ holds in the viscosity sense. The ideas involved in this step are completely new, as here, unlike in nondominated sorting, it is necessary to understand the behavior of the height function at points where it is influenced by multiple cones in the family $\{Q_{p}\}_{p \in \mathcal{N}^{*}}$.  The nondegeneracy of $D\bar{u}$ is proved using a series of growth lemmas. 
These lemmas demonstrate how the randomness of the point cloud influences the graph of the height function, forcing it to develop corners in the degenerate directions (see Remark \ref{remark:corners}). These corners are the reason that $\bar{H}_{\varphi}(Du) > 0$ holds.

These three steps show local convergence of the height function. To argue that the convergence occurs globally, we impose a compatibility condition on $\mathrm{U}$, the support of the point cloud,
which ensures that $\bar{u} = 0$ on $\partial \mathrm{U}$. We expect that this condition is necessary and provide a counterexample where the boundary condition fails in Section \ref{sec:counterexample}.

Lastly, this work constitutes a contribution to the location analysis literature. While the sets $\mathcal{E}$ and $\{Q_{p}\}_{p \in \mathcal{N}^{*}}$ are not new in location analysis, the fact that only the latter appears in the limiting Hamiltonian $\bar{H}_{\varphi}$ leads to the strange geometric properties of the limiting height function $\bar{u}$ already highlighted above.  At a technical level, this paper contributes some new notions, such as the compatibility condition, and also highlights the utility of convex duality (particularly dual cones), which may be useful elsewhere.  Finally, as is explained next, the limit $\bar{u}$ can be interpreted as the ``arrival time" function of a certain geometric flow.  To the best of our knowledge, this flow is new.  
\subsection{Level Set Formulation}\label{sec:level-set-pde}

The scaling limit of the height functions can be rephrased in terms of the Pareto hull peeling process itself.  Notice that if the limit function $u$ solves \eqref{eq:eikonal}, then the function $v$ defined in $\mathrm{U} \times (0,\infty)$ by 
	\begin{equation} \label{eq:arrival-time}
		v(x,t) = u(x) - t
	\end{equation}
is a solution of the parabolic PDE
	\begin{equation} \label{eq:level-set}
		\sqrt{f} v_{t} + \sqrt{\bar{H}_{\varphi}(Dv)} = 0 \quad \text{in} \, \, \mathrm{U} \times (0,\infty).
	\end{equation}
This can be understood as the level set formulation of a geometric flow.  

More precisely, if we define sets $(E_{t})_{t \geq 0}$ by 
	\begin{equation} \label{eq:geometric-flow}
		E_{t} = \{x \in \mathrm{U} \, \mid \, v(x,t) > 0\} = \{x \in \mathrm{U} \, \mid \, u(x) > t\},
	\end{equation}
then these sets form a generalized level set evolution with normal velocity
	\begin{equation} \label{eq:normal-velocity}
		V_{\partial E_{t}} = \sqrt{f^{-1} \bar{H}_{\varphi}(n_{\partial E_{t}})}.
	\end{equation}
The correspondence between level set PDE such as \eqref{eq:level-set} and generalized level set evolutions is explained in \cite{barles_souganidis_fronts}.
	
Stated in these terms, our result reads as follows:
	
	\begin{cor} \label{cor:level-set} Given $n \in \N$, let $\{E^{(n)}_{k}\}_{k \in \N}$ be the Pareto hull peeling process associated with $X_{nf}$.  If $\varphi$ satisfies \eqref{eq:linear-assumption} and $\mathrm{U}$ is a bounded, open Pareto efficient set compatible with $\varphi$, then, with probability one,
		\begin{equation*}
			\bar{E}^{(n)}_{\lfloor n^{\frac{1}{2}} t \rfloor} \to \bar{E}_{t} \quad \text{for each} \, \, t > 0,
		\end{equation*}
	where $(E_{t})_{t \geq 0}$ is the generalized level set evolution with velocity \eqref{eq:normal-velocity} and initial datum $E_{0} = \mathrm{U}$ and the convergence is in the Hausdorff metric. \end{cor}
	
We reiterate that when \eqref{eq:linear-assumption} fails, the norm ball $\{\varphi \leq 1\}$ is strictly convex, Pareto hull peeling coincides with convex hull peeling, and the scaling is different.  In \cite{calder2020limit}, it is shown that, in this case,
	\begin{equation*}
		E^{(n)}_{\lfloor n^{\frac{2}{3}} t \rfloor} \to E_{t}
	\end{equation*}
where $(E_{t})_{t \geq 0}$ shrinks according to affine curvature flow.
\begin{figure}
	\includegraphics[width=0.25\textwidth]{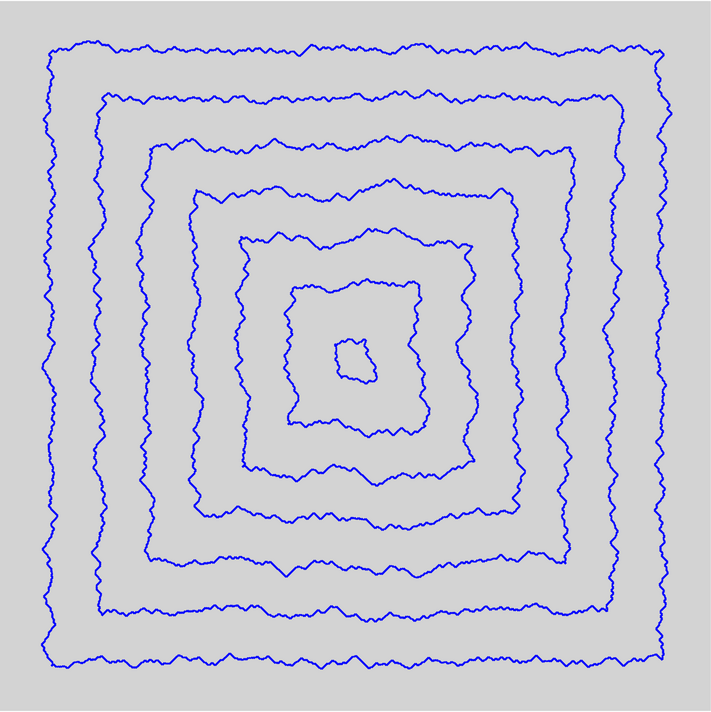} \qquad
	\includegraphics[width=0.25\textwidth]{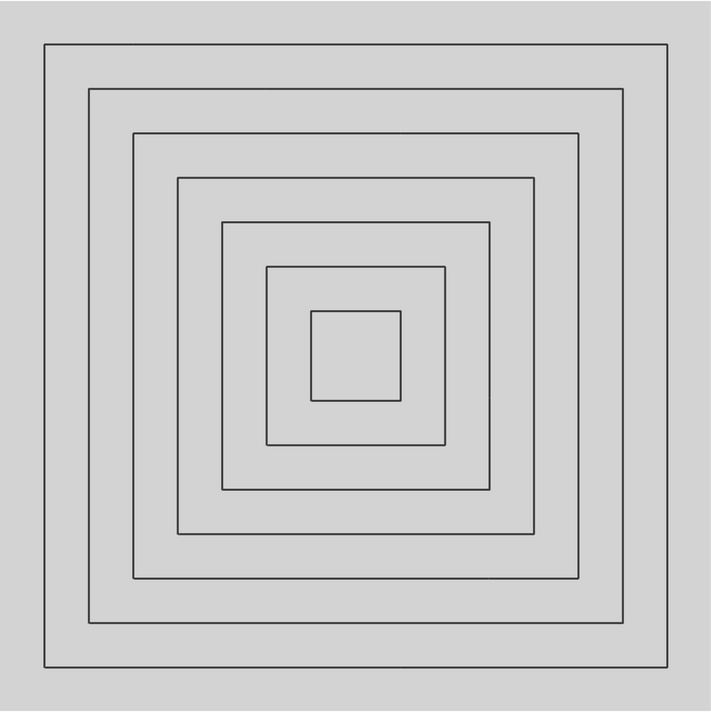}
	
	\caption{Pareto peeling of a homogeneous Poisson cloud with respect to the $\ell^{\infty}$ norm in a square domain $\mathrm{U} = (-1,1)^2$. On the left is a simulation and on the right 
		are level sets of the viscosity solution $u(x_1, x_2) =  1 - \max(|x_1|,|x_2|)$.} \label{fig:square_flow}
\end{figure}

\subsection{Outline of the Paper}  In Section \ref{sec:pareto-hulls}, the effective Hamiltonian is defined, its properties are discussed, and the necessary geometric preliminaries are reviewed.  This is also where the dynamic programming principle for the height functions is stated and proved.  In Section \ref{sec:q-nds}, we recall the basic scaling limit results related to nondominated sorting and explain how they can be generalized to the setting required here.  Additionally, at the end of Section \ref{sec:q-nds}, we describe how nondominated sorting can be regarded as an infinite volume limit of Pareto hull peeling.  The proof of Theorem \ref{theorem:pareto-convergence} is outlined in Section \ref{sec:viscosity-solutions}, which also includes preliminaries on viscosity solutions, estimates on the asymptotic behavior of the height functions, and proofs of the corollaries.  Sections \ref{sec:subsolution}, \ref{sec:supersolution}, and \ref{sec:degenerate-directions} comprise the main technical contributions of the paper and are devoted to proving that the limiting height functions solve \eqref{eq:eikonal}.  Finally, Section \ref{sec:conclusion} is a discussion of open questions for future work.

\subsection{Notation and conventions} \label{sec:notation}

\begin{itemize}
	\item Given $x \in \mathbb{R}^{2}$, we will sometimes write $x = (x_{1},x_{2})$ with $x_{1},x_{2} \in \mathbb{R}$ denoting the first and second components of $x$ with respect to the standard orthonormal basis of $\mathbb{R}^{2}$.
	\item Unless made explicit, $C, c$ are positive constants which may change from line to line.  Dependence of $C$ on other constants is indicated by a subscript (\eg, $C_{d}$ denotes a constant depending on the dimension $d$).
	\item For a subset $A$ of $\R^2$ write $|A|$ for its Lebesgue measure, $\bar{A}$ for closure, $\inte(A)$ for interior, and $\partial A$ for boundary. 
	\item For $x,y \in \R^2$, 
	\begin{equation} \label{eq:notation-for-squares}
	[x,y] = [x_1, y_1] \times [x_2, y_2]
	\end{equation}
	and for $a,b \in \bar{\R}$
	\[
	[a,b]^2 = [a,b] \times [a,b]
	\]
	and for $x \in \R^2$ and $b \in \bar{\R}$, 
	\[
	[x,b]^2 = [x_1,b] \times [x_2, b]
	\]
	and vice-versa. 
	
	\item $\langle x, y \rangle$ denotes the Euclidean inner product of $x,y \in \R^{2}$.
	
	\item $x \times y$ is the cross product of two vectors $x, y \in \R^{2}$.  Recall this can be computed via the determinant
	\begin{equation*}
	x \times y = \det \left( \begin{array}{c c}
	x_{1} & y_{1} \\
	x_{2} & y_{2}
	\end{array} \right).
	\end{equation*}
	Alternatively, using wedge products, $x \times y$ is the real number such that $x \wedge y = (x \times y) (1,0) \wedge (0,1)$.  
	
	\item Given a vector $q \in \R^2 \setminus \{0\}$, we denote by $H_{q}$ the halfspace determined by $q$ by	
		\begin{equation*}
			H_{q} = \{x \in \R^{2} \, \mid \, \langle q, x \rangle \geq 0\}.
		\end{equation*}
	
	\item $\|x\|_{\infty}$ denotes the $\ell^{\infty}$ norm, $\|x\|_{1}$ denotes the $\ell^{1}$ norm, and $\|x\| = \|x\|_2$ denotes the Euclidean or $\ell^{2}$ norm. 
	\item 
	
	$B(x_0, r) = \{ x \in \R^2 \, \mid \, \|x-x_0\|\leq r\}$ denotes the ball of radius $r$ centered around $x_0$.
	\item $S^{1}$ denotes the set of all unit vectors in $\mathbb{R}^{2}$, that is, $S^{1} = \{v \in \mathbb{R}^{2} \, \mid \, \|v\| = 1\}$.
	\item 
	Given $p = (p_{1},p_{2}) \in \R^{2}$, we denote by $p^{\perp}$ the vector defined by
	\begin{equation*}
	p^{\perp} = (-p_{2},p_{1}).
	\end{equation*}
	
	\item  $\cone(\mathrm{C}) = \{ a v \in \R^2 \, \mid \, a \geq 0 \mbox{ and } v \in \mathrm{C}\}$ and $\conv(C)$ is the convex hull of $C$. 
	
	\item Differential inequalities are interpreted in the viscosity sense.
	
	\item Given two random variables $X, Y$, we write $Y \overset{\mathcal{D}}{=} X$ if $X$ and $Y$ have the same distribution.
	
\end{itemize}

Lastly, we sometimes use the fact that the inner product and cross product can be computed in terms of lengths and angles.  More precisely, given $v, w \in \mathbb{R}^{2}$, we have
	\begin{equation*}
		\langle w, v \rangle = \|v\| \|w\| \cos(\theta) \quad \text{and} \quad w \times v = \|v\| \|w\| \sin(\theta),
	\end{equation*}
where $\theta$ is the angle traversed going from $w$ to $v$.  (See, \eg, \cite{marsden} or \cite{div-grad-curl}.)

\subsection{Code}
Programs used to generate the figures are included in the arXiv submission and also at \url{https://github.com/nitromannitol/2d_pareto_peeling}.

\subsection{Acknowledgments}
We thank Jeff Calder for helpful suggestions and encouragement. A.B. thanks Charles K. Smart for many inspiring discussions.  P.S.M. gratefully acknowledges his thesis advisor, P.E.\ Souganidis, for introducing him to viscosity solutions and homogenization and for unwavering support these past few years.

A.B. was partially supported by Charles K. Smart's NSF grant DMS-2137909 and NSF grant DMS-2202940. P.S.M. was partially supported by P.E. Souganidis's NSF grants DMS-1600129 and DMS-1900599 and NSF grant DMS-2202715.



\section{Pareto Hulls and the Effective Hamiltonian}\label{sec:pareto-hulls}

\begin{figure}
	\begin{tikzpicture}[scale=1.9]

\draw[thick,->] (0,0) -- (0,1.9);
\draw[thick,->] (0,0) -- (0,-1.9); 
\draw[thick,->] (0,0) -- (-1.9,0);
\draw[thick,->] (0,0) -- (1.9,0);

\draw[anchor = west] (2,0) node {$w_{p}$}; 
\draw[anchor = east] (0,2) node {$v_{p}$};
\draw[anchor = east] (-2,0) node {$-w_{p}$}; 
\draw[anchor = east] (0,-2) node {$-v_{p}$};

\draw[dashed] (-1,0) -- (0,1);
\draw[dashed] (0,-1) -- (1,0);

\draw[dashed] (0,1) arc[start angle=90, end angle=0, radius=1];
\draw[dashed] (0,-1) arc[start angle=-90, end angle=-180, radius=1];

\draw[anchor = east] (1.1,1.1) node {$Q_{p}$};
\draw[anchor = east] (-1.1, -1.1) node {$-Q_{p}$};
\draw[anchor = east] (1.1,-1.1) node {$-\mathcal{Q}_{p}$};
\draw[anchor = east] (-1.1,1.1) node {$\mathcal{Q}_{p}$};


%

%

\end{tikzpicture} 
	\caption{The cones associated to the norm $\varphi(x) = \max\{|x_{1} - x_{2}|,\|x\|_{2}\}$ from Example \ref{example:hamiltonian-l1-norm}.  In this case, $\mathcal{N}^{*} = \{(1,-1),(-1,1)\}$ so there are only two flat cones $\{Q_{p},-Q_{p}\}$.  Notice that $\mathcal{Q}_{p}$ and $-\mathcal{Q}_{p}$ are the cones generated by the facets (flat parts) of the curve $\{\varphi = 1\}$.  The circular arcs are $\mathcal{E}$.}\label{fig:cone_definition}
\end{figure}
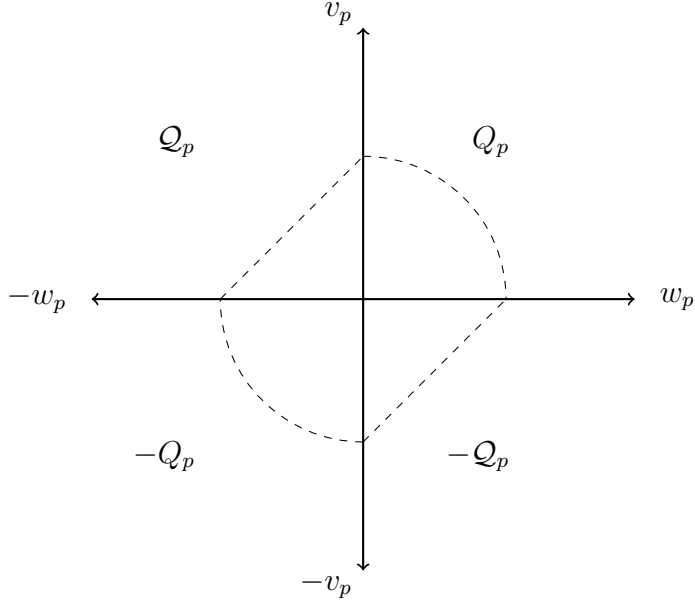

\begin{figure}
	\includegraphics[scale=0.125]{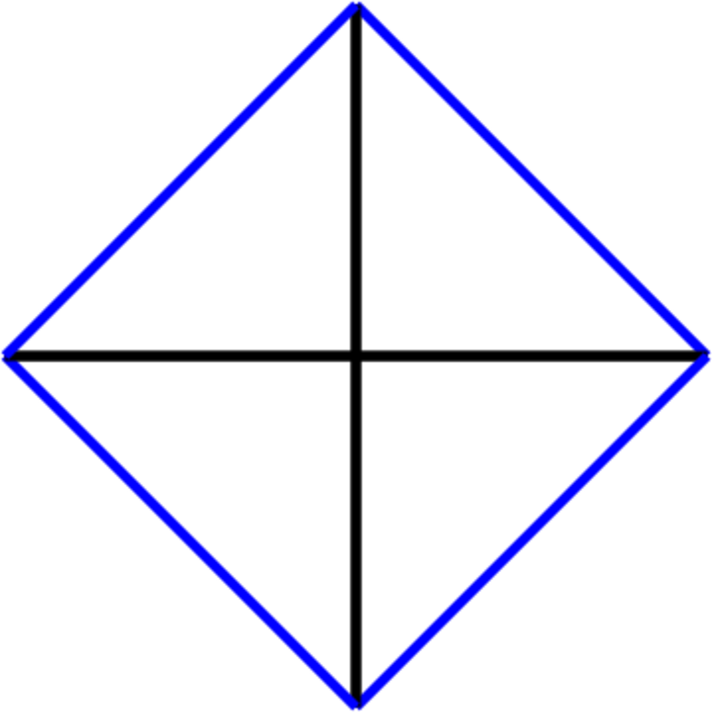} \qquad
	\includegraphics[scale=0.125]{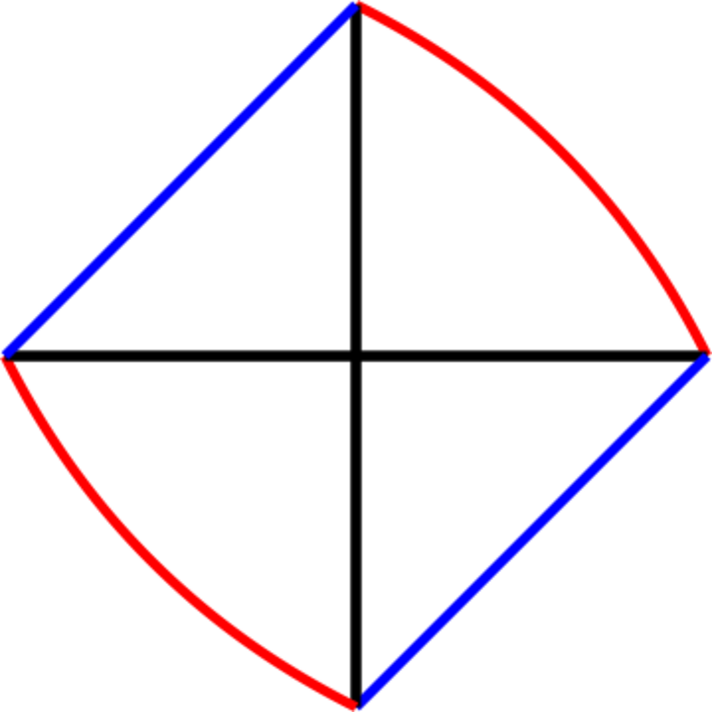}
	\caption{Unit balls of the two norms described in Example \ref{example:hamiltonian-l1-norm} partitioned into $\{\mathcal{Q}_p\}$ cones (in blue). }\label{fig:norm-pics}
\end{figure}

In this section, we give an explicit formula for the effective Hamiltonian $\bar{H}_{\varphi}$ in terms of certain geometric objects associated to the norm $\varphi$.   We then review the definition and main properties of the the Pareto hull and show that the height function associated with Pareto hull peeling satisfies a dynamic programming principle.  The section concludes with the derivation of some properties of $\bar{H}_{\varphi}$, including continuity and non-coercivity.

\subsection{Geometric Preliminaries} \label{subsec:norm-partition}
\begin{figure}
	\includegraphics[scale=0.125]{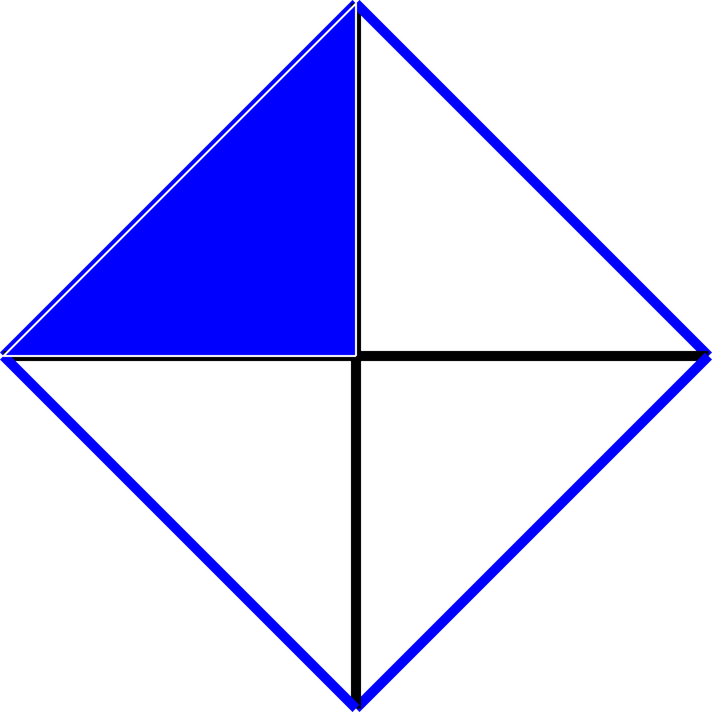} 
	\includegraphics[scale=0.125]{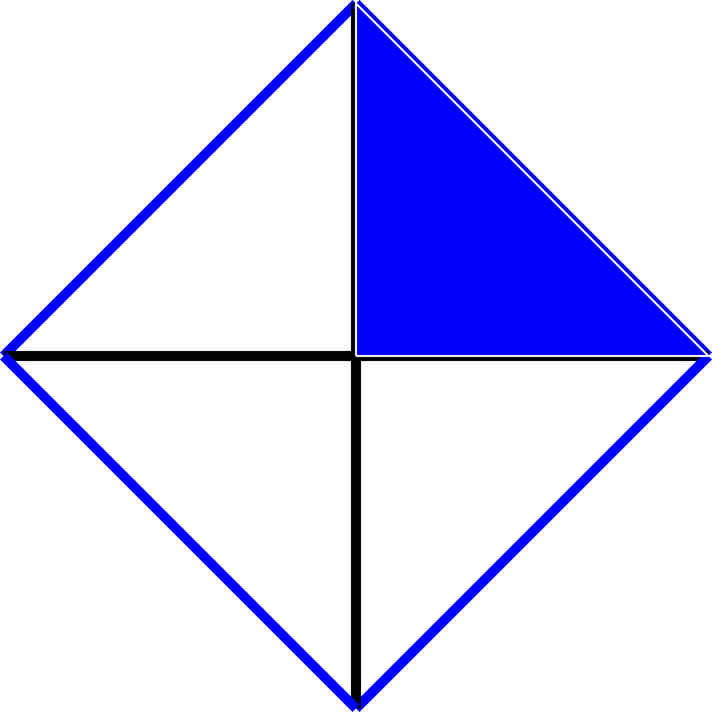} 
	\includegraphics[scale=0.125]{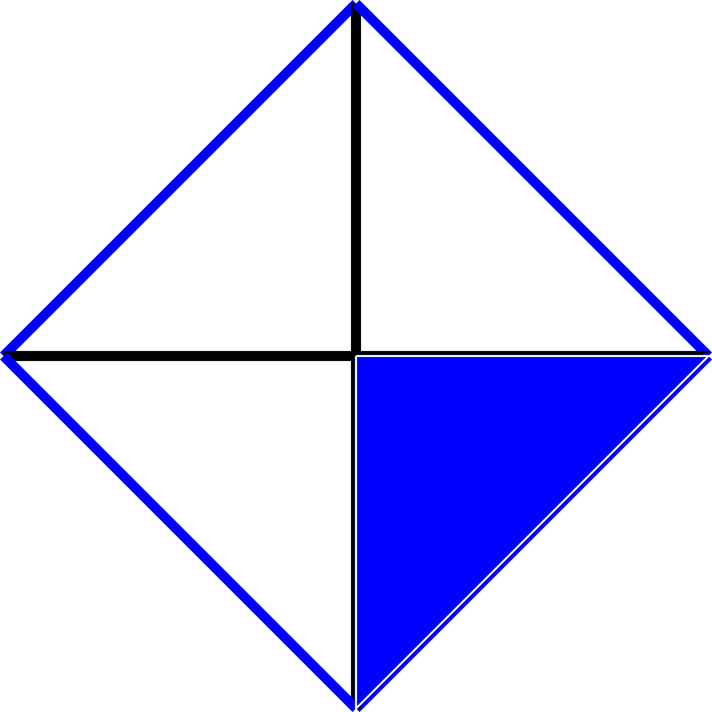} 
	\includegraphics[scale=0.125]{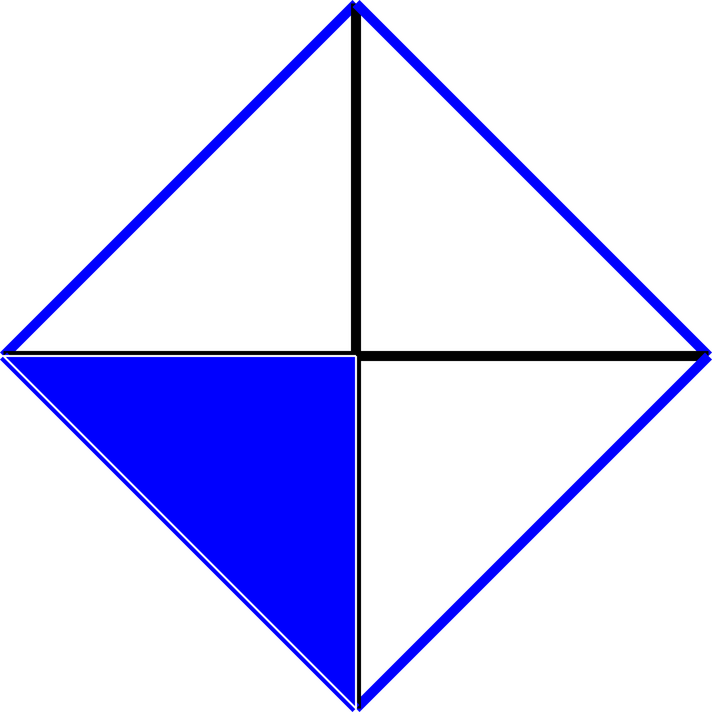} \\
		\includegraphics[scale=0.125]{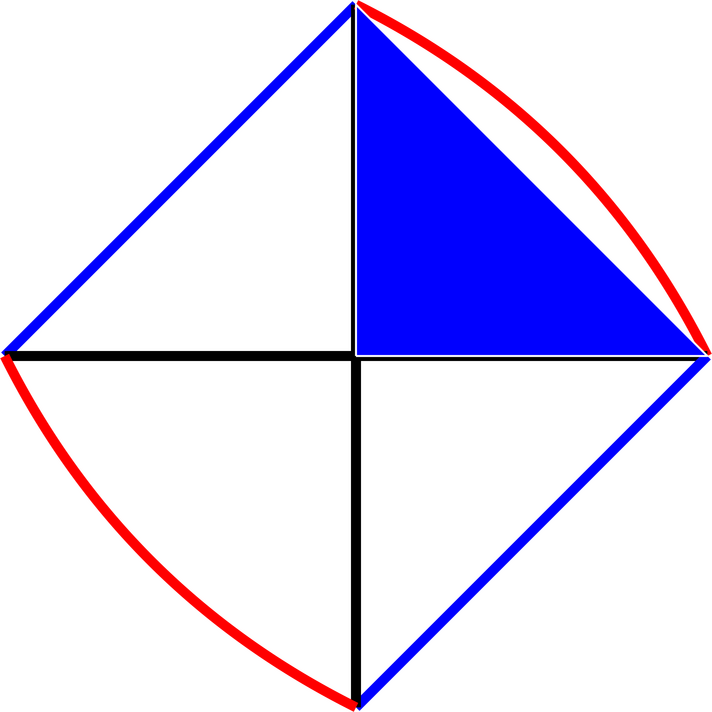} 
	\includegraphics[scale=0.125]{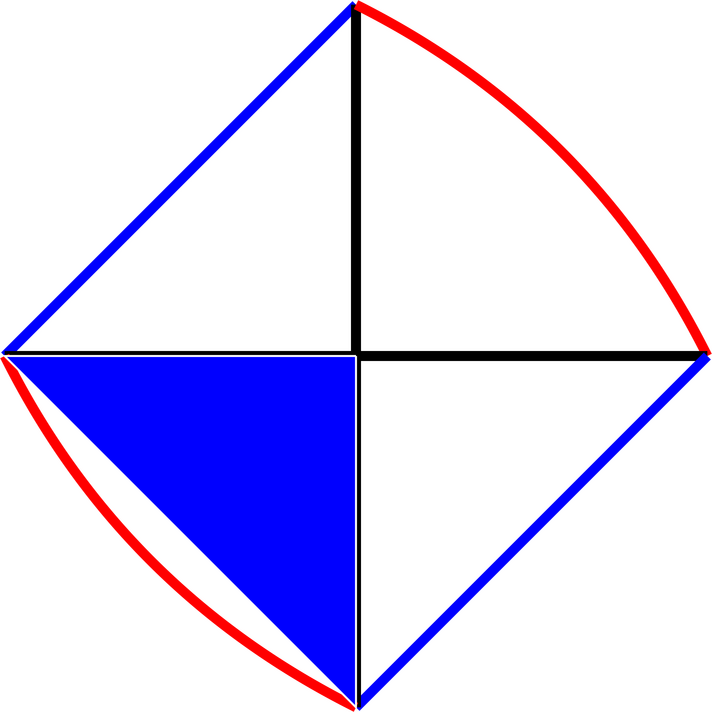} 
	\caption{$\{Q_p\}$ cones of the two norms described in Figure \ref{fig:norm-pics}.}
\end{figure}

Before describing the link between Pareto hull peeling and the limiting Hamilton-Jacobi PDE \eqref{eq:eikonal}, we need to fix notation and recall a number of concepts from convex analysis and location analysis.  

Recall that $\varphi$ denotes a norm in $\R^2$.  In this section, we impose no assumptions on $\varphi$ other than being a norm.  In the discussion that follows, we use basic facts and language from convex analysis; Rockafellar \cite{rockafellar} and Schneider \cite{schneider} contain the main definitions and results and the reader can consult, \eg, Bellettini \cite[Sections 2.1--2.2]{bellettini} and Morfe and Souganidis \cite[Section 2]{morfe_souganidis} for a discussion of the properties of norms specifically.  

It is convenient for us to note that $\varphi$ has a dual norm $\varphi^{*}$, determined by the formula
	\begin{equation*}
		\varphi^{*}(p) = \max\left\{ \frac{\langle p,q\rangle}{\varphi(q)} \, \mid \, q \in \R^{2} \setminus \{0\}\right\}.
	\end{equation*}  
Let $\mathcal{N}^{*}$ be the set of corner points of $\{\varphi^{*} = 1\}$, that is,
\begin{equation} \label{eq:normball-corners}
\mathcal{N}^{*} = \{p \in \{\varphi^{*} = 1\} \, \mid \, \# \partial \varphi^{*}(p) > 1\},
\end{equation} 
where $\partial \varphi^{*}$ denotes the subdifferential of $\varphi^{*}$ and $\# \partial \varphi^{*}(p)$ is the cardinality of $\partial \varphi^{*}(p)$.  We will be interested in the corresponding cones $\{\mathcal{Q}_{p}\}_{p \in \mathcal{N}^{*}}$ given by 
\begin{equation*}
\mathcal{Q}_{p} = \{q \in \R^{2} \, \mid \, \langle q, p \rangle = \varphi(q)\}.
\end{equation*}  
In view of the definition of $\mathcal{N}^{*}$, the cones $\{\inte(\mathcal{Q}_{p}) \, \mid \, p \in \mathcal{N}^{*}\}$ are nonempty and pairwise disjoint.  Therefore, by the separability of $\mathbb{R}^{2}$, it follows that $\mathcal{N}^{*}$ is a countable set.  At the same time, the sets $\{\partial \varphi^{*}(p)\}_{p \in \mathcal{N}^{*}}$ are precisely the boundary facets of $\{\varphi = 1\}$.  Thus, we may reformulate our main assumption \eqref{eq:linear-assumption} as:
	\begin{equation} \label{eq:facets_characterize}
		\mathcal{N}^{*} \, \, \text{is nonempty} \quad \iff \quad \eqref{eq:linear-assumption} \, \,  \text{holds.}
	\end{equation}

It is straightforward to verify that $\mathcal{Q}_{p} = \cone(\partial \varphi^{*}(p))$ for each $p \in \mathcal{N}^{*}$.  Therefore, geometrically, $\{\mathcal{Q}_{p}\}_{p \in \mathcal{N}^{*}}$ are the cones determined by the facets of $\{\varphi = 1\}$, as in Figures \ref{fig:cone_definition} and \ref{fig:norm-pics}.

Since we are working in dimension $d = 2$, for each $p \in \mathcal{N}^{*}$, we can fix a basis $\{w_{p},v_{p}\} \subseteq \R^{2}$ such that 
\begin{equation} \label{eq:normball-facets}
\mathcal{Q}_{p} = \{c v_{p} - d w_{p} \, \mid \, c,d \geq 0\}.
\end{equation}
We can and will assume that $\|w_{p}\| = \|v_{p}\| = 1$ and $w_{p} \times v_{p} > 0$.

In what follows, we denote by $Q_{p}$ the convex cone obtained from $\mathcal{Q}_{p}$ by 
\begin{equation} \label{eq:normball-flat-cones}
Q_{p} = \{a w_{p} + b v_{p} \, \mid \, a, b \geq 0\}.
\end{equation}
We will refer to the sets $\{Q_p\}_{p \in \mathcal{N}^{*}}$ as {\it flat cones}.  The importance of these cones to the study of Pareto hulls was fully realized in the work of Durier and Michelot \cite{durier1986sets, durier1987sets}; see also the papers by Pelegrin and Fernandez \cite{pelegrin1988determination,pelegrin1989determination}.

It is important to note that $\mathcal{N}^{*}$ is invariant under negation, that is, $p \in \mathcal{N}^{*}$ if and only if $-p \in \mathcal{N}^{*}$, which is immediate since $\partial \varphi^{*}(-p) = - \partial \varphi^{*}(p)$.  Further, a direct computation shows that
	\begin{equation} \label{eq:negation-thing}
		\mathcal{Q}_{-p} = -\mathcal{Q}_{p} \quad \text{and} \quad Q_{-p} = -Q_{p}.
	\end{equation}
	
Finally, let 
\begin{equation} \label{eq:extreme-points}
\mathcal{E} = \mbox{extreme points of $\{\varphi \leq 1\}$}.
\end{equation}
As is well known, the curve $\{\varphi = 1\}$ equals the union of the facets $\{\partial \varphi^{*}(p)\}_{p \in \mathcal{N}^{*}}$ and the extreme points $\mathcal{E}$.  Thus, since every point in $\mathbb{R}^{2}$ is a multiple of some element of $\{\varphi = 1\}$,
	\begin{equation} \label{eq:decomposition_round_flat}
		\mathbb{R}^{2} =  \cone(\mathcal{E}) \cup \bigcup_{p \in \mathcal{N}^{*}} \cone(\partial \varphi^{*}(p)) = \cone(\mathcal{E}) \cup \bigcup_{p \in \mathcal{N}^{*}} \mathcal{Q}_{p}.
	\end{equation}

In this paper, we are primarily interested in norms $\varphi$ with at least one boundary facet.  Accordingly, the following standard terminology will be useful to keep in mind, so much so that we give a careful definition.

\begin{definition} \label{def:norms} If $\varphi$ is a norm in $\mathbb{R}^{2}$ such that (i) $\mathcal{N}^{*}$ is finite and (ii) $\mathbb{R}^{2} = \bigcup_{p \in \mathcal{N}^{*}} \mathcal{Q}_{p}$, then $\varphi$ is said to be \emph{polyhedral}.  Equivalently, $\varphi$ is polyhedral if and only if the unit ball $\{\varphi \leq 1\}$ is a polygon.

If $\varphi$ is a norm in $\mathbb{R}^{2}$ such that $\mathcal{N}^{*} = \emptyset$ (or, equivalently, $\{\varphi = 1\} = \mathcal{E}$), then the unit ball $\{\varphi \leq 1\}$ is said to be \emph{strictly convex.}   \end{definition}

At this stage, since it will be needed shortly, let us fix the notation $H_{q}$ for the half-space in $\mathbb{R}^{2}$ determined by the vector $q \in \mathbb{R}^{2}$, that is,
	\begin{equation} \label{eq:halfspace-formula}
		H_{q} = \{x \in \mathbb{R}^{2} \, \mid \, \langle q,x \rangle \geq 0\}.
	\end{equation}

\begin{figure}
		\includegraphics[scale=0.1]{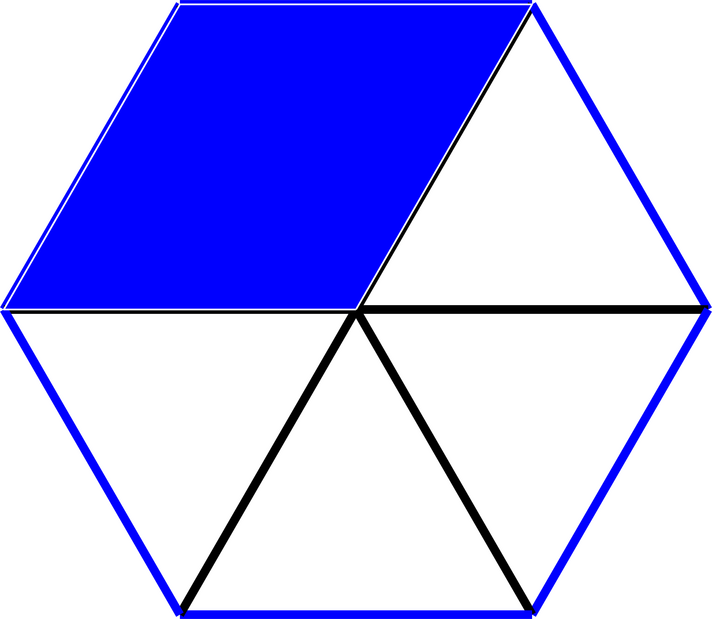} 
	\includegraphics[scale=0.1]{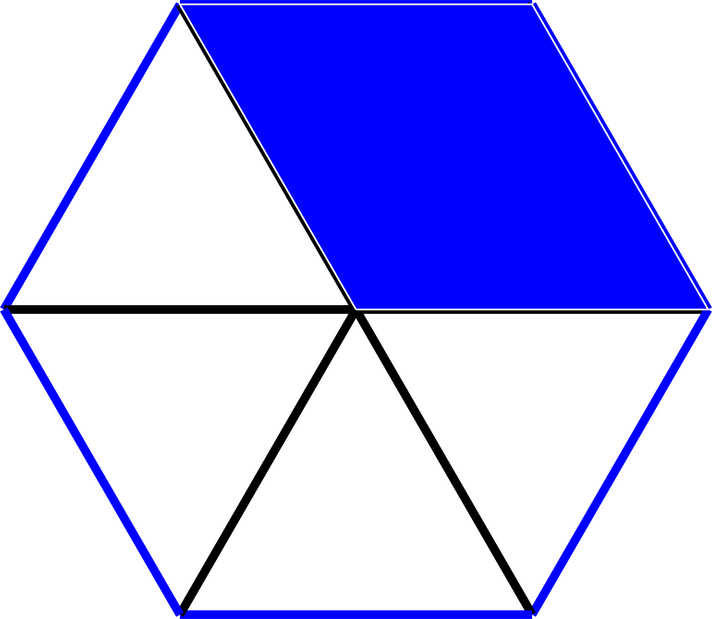} 
	\includegraphics[scale=0.1]{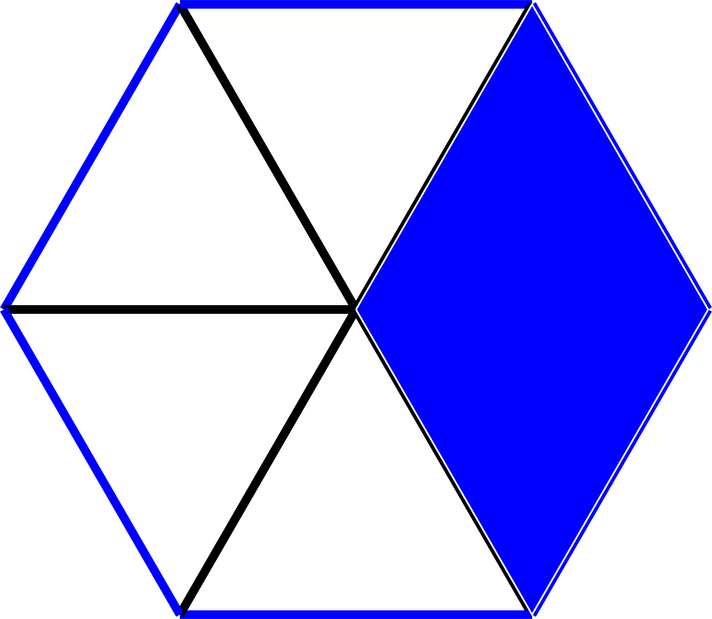} 
	\includegraphics[scale=0.1]{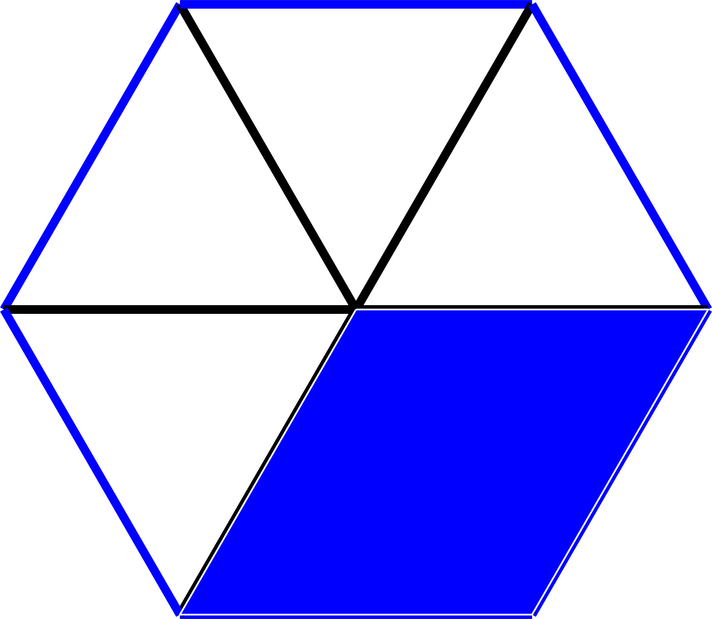} 
		\includegraphics[scale=0.1]{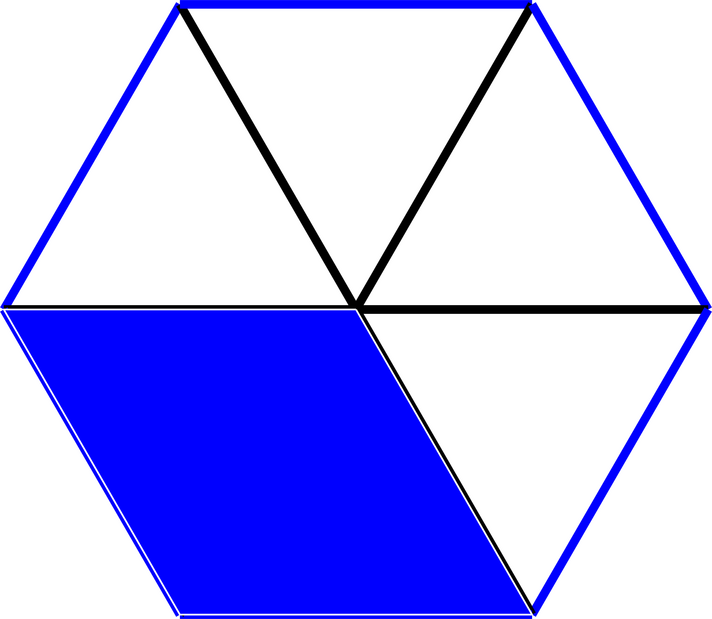} 
			\includegraphics[scale=0.1]{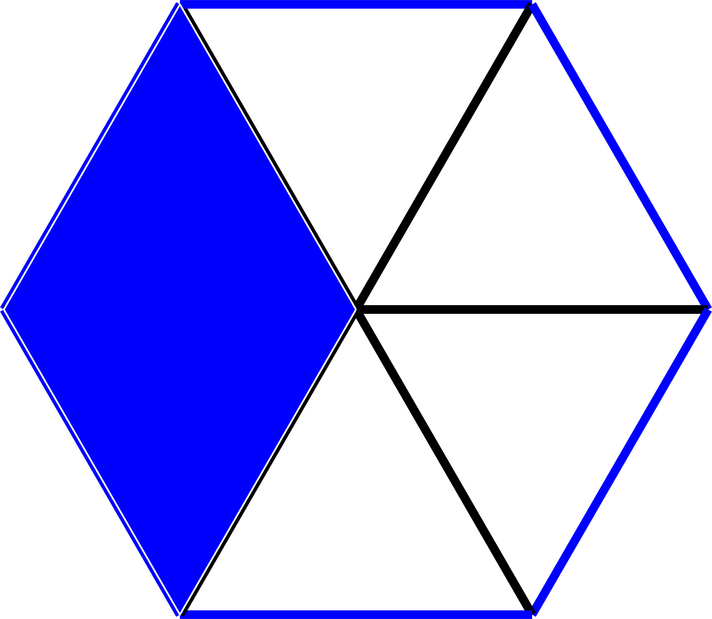}  \\
			
	\includegraphics[scale=0.1]{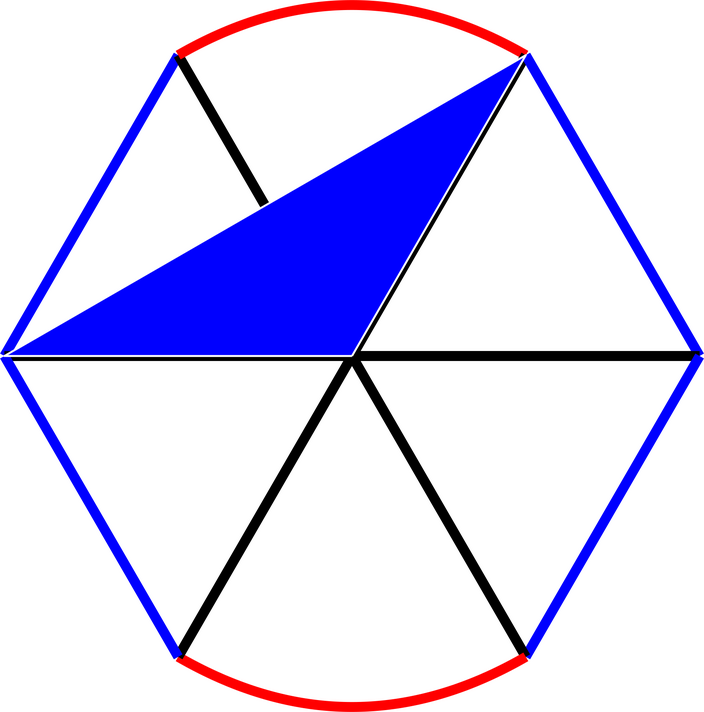} 
	\includegraphics[scale=0.1]{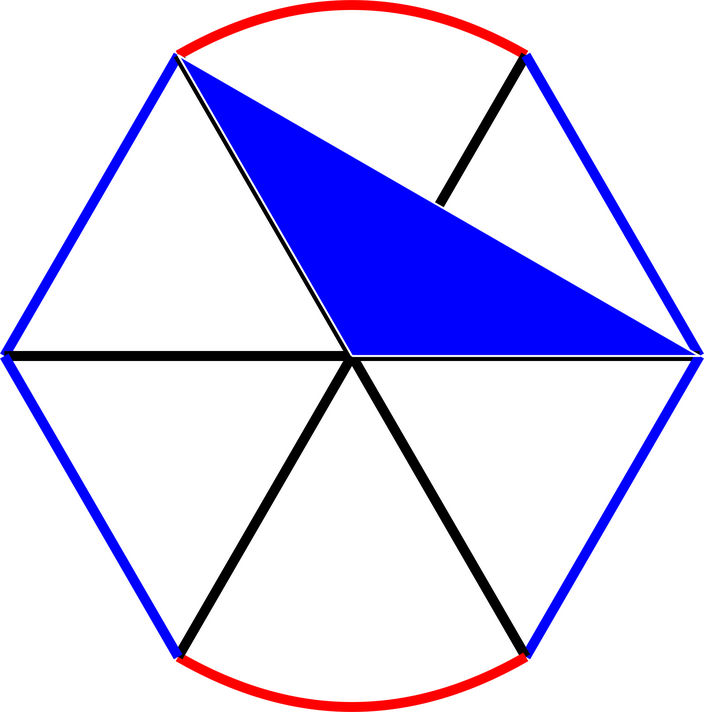} 
	\includegraphics[scale=0.1]{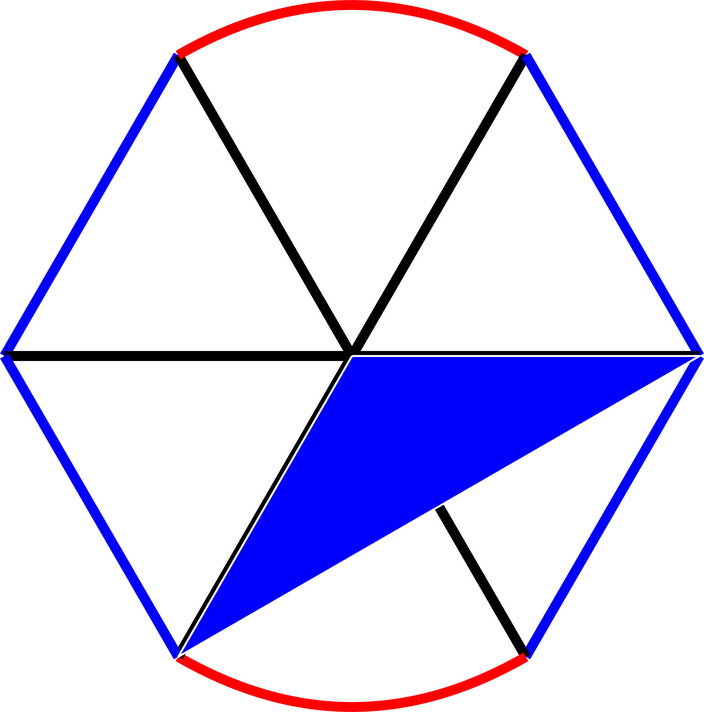} 
	\includegraphics[scale=0.1]{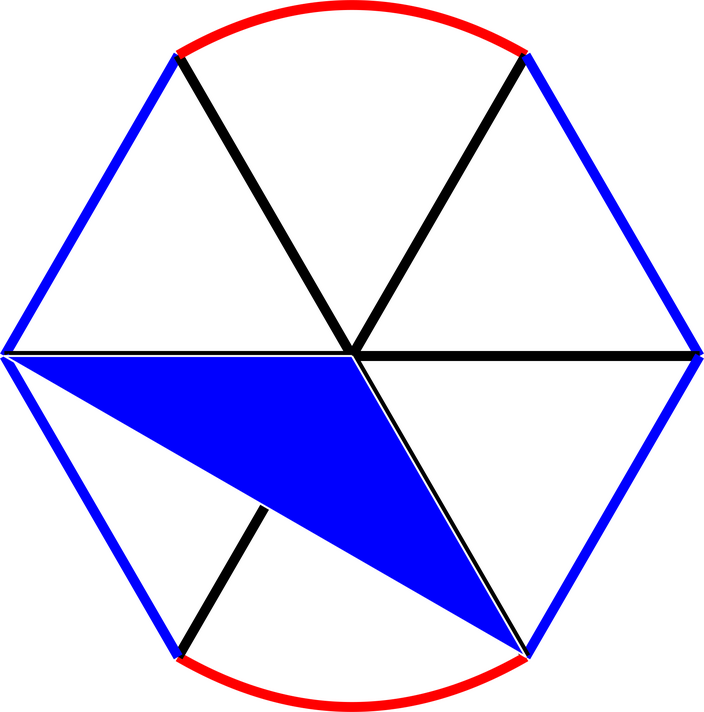} 
	\caption{$\{Q_p\}$ cones of the two indicated hexagonal norms.}
\end{figure}

\subsection{Pareto Hulls and Dynamic Programming}  With the notation of the previous section, we can now define the effective Hamiltonian:
\begin{equation} \label{eq:effective-hamiltonian}
\bar{H}_{\varphi}(\xi) = \max \left\{ \sup_{p \in \mathcal{N}^{*}} \frac{\langle \xi, v_{p} \rangle \langle \xi, w_{p} \rangle}{|v_{p} \times w_{p}|}, 0 \right\}.
\end{equation}
In addition to making $\bar{H}_{\varphi}$ non-negative, which is convenient for the level-set description \eqref{eq:normal-velocity}, the zero in the definition can be understood as the contribution from the `round parts' $\mathcal{E}$.

The derivation of the Hamilton-Jacobi equation \eqref{eq:eikonal} and the formula \eqref{eq:effective-hamiltonian} uses the fact that, in a certain sense, the height functions of Pareto peeling themselves satisfy a discrete PDE.  We make this precise via a dynamic programming formulation.

Before doing so, let us recall the definition of the Pareto hull.  Let $A \subseteq \R^2$ be a given compact set and recall that the {\it Pareto hull} of $A$ with respect to $\varphi$, $\mathcal{P}(A)$,  is
\begin{equation}
\mathcal{P}(A) := \{ x \in \R^2 : \forall y \not = x  \mbox{ there exists $a \in A$  with $\varphi(a-x) < \varphi(a-y)$} \}.
\end{equation}
 It is worth emphasizing at this point that the Pareto hull is monotone, that is,
\begin{equation} \label{eq:monotonicity}
\mathcal{P}(A) \subseteq \mathcal{P}(B) \quad \text{if} \quad A \subseteq B.
\end{equation}
For convenience, we abuse notation by defining $\mathcal{P}(\emptyset) = \emptyset$.

Recall from the introduction that \emph{Pareto hull peeling} of a finite set $A$ is a collection of sets, $\{E_{k}(A)\}_{k \in \N}$, defined recursively via
\begin{equation*}
E_1(A) = \mathcal{P}(A) \quad \mbox{and} \quad E_{k+1}(A) = \mathcal{P}(A \cap \inte(E_k(A))).
\end{equation*}
Associated to this process is a height function $u_{A}$ defined in $\R^{2}$ by 
\begin{equation*}
u_{A}(x) = \sum_{k = 1}^{\infty} 1_{\inte(E_{k}(A))}(x).
\end{equation*}
The next result, which identifies the dynamic programming principle satisfied by $u_{A}$, plays a fundamental role in what follows.

\begin{prop} \label{prop:dpp}  Given $A \subseteq \R^{2}$ finite, the height function $u_{A}$ satisfies a dynamic programming principle:
	\begin{equation} \label{eq:pareto-peeling-dpp}
	u_{A}(x) = \min \left\{\inf_{p \in \mathcal{N}^{*}} \sup_{y \in x + \inte(Q_{p})} u_{A}(y) + 1_{A}(y), \inf_{q \in \mathcal{E}} \sup_{y \in x + \inte(H_{q^{\perp}})} u(y) + 1_{A}(y) \right\} \quad \text{for} \, \, x \in \R^{2}.
	\end{equation}
Further, if $\varphi$ is polyhedral (see Definition \ref{def:norms}), then the dynamic programming principle simplifies as follows:
	\begin{equation} \label{eq:polyhedral-dpp}
		u_{A}(x) = \min_{p \in \mathcal{N}^{*}} \sup_{y \in x + \inte(Q_{p})} u_{A}(y) + 1_{A}(y) \quad \text{for} \, \, x \in \mathbb{R}^{2}.
	\end{equation}\end{prop}

\subsection{Cone characterization of Pareto hulls}  The dynamic programming principle follows from a characterization of Pareto hulls using the cones $\{Q_{p}\}$ and $\{H_{q^{\perp}}\}$ defined above.  
\begin{theorem}[\cite{durier1986sets, durier1987sets}]\label{thm:cone_efficient}
	Let $A$ be a compact set in $\R^2$, then
	\begin{align*}
	x \in \inte(\mathcal{P}(A)) \iff &\mbox{ for every $p \in \mathcal{N}^*$, }  A \cap (x + \inte(Q_p)) \not= \emptyset \\
	&\mbox{ and for every $q \in \mathcal{E}$, } A \cap (x + \inte(H_{q^{\perp}})) \not = \emptyset.
	\end{align*}
	If $\varphi$ is polyhedral, the halfspace constraint is unnecessary,  
	\[
	x \in \inte(\mathcal{P}(A)) \iff \mbox{ for every $p \in \mathcal{N}^*$, }  A \cap (x + \inte(Q_p)) \not= \emptyset.
	\]
\end{theorem}
\begin{proof}
	The proof combines Remark 2.1 in \cite{durier1986sets} and Theorem 4.1 in \cite{durier1986sets}
	together with Proposition 2.5 in \cite{durier1987sets}.
	For a different proof (and algorithms),
	see \cite{pelegrin1988determination,pelegrin1989determination}. 
\end{proof}

The previous representation is reminiscent of the halfspace separation characterization of convex hulls.  In fact, the following reformulation, which will be useful in what follows, shows that the Pareto hull can be thought of as a constrained convex hull.

\begin{cor} \label{cor:-constrained-convex-hull}  Given $A \subseteq \R^{2}$ compact and $x \in \R^{2}$, the inclusion $x \in \inte(\mathcal{P}(A))$ holds if and only if the following two conditions are satisfied:
	\begin{itemize}
		\item[(i)] $A \cap (x + \inte(Q_{p})) \neq \emptyset$ for each $p \in \mathcal{N}^{*}$,
		
		\item[(ii)] $x \in \inte(\conv(A))$.
	\end{itemize}
\end{cor}  	

For the reader's convenience, the proof of Corollary \ref{cor:-constrained-convex-hull} is provided at the end of Section \ref{prop:key-duality} below.

As an immediate consequence of the previous corollary, we recover the classical result that in dimension 2, the Pareto hull coincides with the convex hull whenever $\{\varphi \leq 1\}$ is strictly convex.

\begin{cor}[\cite{thisse1984some}]  If $\{\varphi \leq 1\}$ is strictly convex, then $\mathcal{P}(A) = \conv(A)$ for any compact $A \subseteq \R^{2}$. \end{cor}  

\begin{proof}  By our Definition \ref{def:norms}, the strict convexity of $\{\varphi \leq 1\}$ means that $\mathcal{N}^{*} = \emptyset$.  Hence condition (i) in the previous corollary is always vacuous in this setting.  Accordingly, that result reduces to the simple identity $\mathcal{P}(A) = \conv(A)$.  \end{proof}  

\subsection{Proof of the dynamic programming principle}  
We next use Theorem \ref{thm:cone_efficient} to prove the dynamic programming principle.  
\begin{proof}[Proof of Proposition \ref{prop:dpp}] In what follows, let $\mathcal{L} = \{Q_p : p \in \mathcal{N}^*\}$ if $\varphi$ is polyhedral.  Otherwise, if $\varphi$ is not polyhedral, we let $\mathcal{L} = \{Q_{p} : p \in \mathcal{N}^{*}\} \cup \{ H_{q^\perp} : q \in \mathcal{E}\}$.  Observe that all cones in $\mathcal{L}$ are convex.

	Let $A_k = A \cap \inte(E_k(A))$ and set $A_0 = A$. By monotonicity \eqref{eq:monotonicity}, $E_k(A) \supseteq E_{k+1}(A)$ so $A_k \supseteq A_{k+1}$ and hence $u_A(x) = k$ if $x \in A_{k} \backslash A_{k+1}$ for all $k \geq 0$. 
	
	Take $x \in \R^2$ and set $j = u_A(x)$. Thus, $x \not \in \inte(E_{j+1}(A))$
	and hence, by Theorem \ref{thm:cone_efficient}, there is 
	$Q \in \mathcal{L}$ so that 
	\begin{equation} \label{eq:strict1}
	A_j \cap (x + \inte(Q)) = \emptyset.
	\end{equation}
	This implies, together with $Q$ being a convex cone, 
	that  $u_A(z) + 1_A(z) \leq j$ for all $z \in (x + \inte(Q))$.
	Indeed, suppose for sake of contradiction that for some $z \in (x + \inte(Q))$  we have $u_A(z) +  1_A(z) \geq j+1$
	and consider the two possible cases, $u_A(z) \geq j+1$ or $u_A(z) = j$ and $z \in A$. 
	
	In the latter case, $z \in A_j \cap (x + \inte(Q))$, which contradicts \eqref{eq:strict1}. In the former case, $z \in \inte(E_{j+1}(A))$ and hence there is $z' \in A_j \cap (z + \inte(Q))$.  
	However, by convexity of $Q$, $\inte(Q) + \inte(Q) \subseteq \inte(Q)$ which implies $z' \in A_j \cap (x + \inte(Q))$, contradicting \eqref{eq:strict1}. Thus, 
	\[
	\inf_{Q \in \mathcal{L}} \sup_{z \in A \cap (x + \inte(Q))} (1_{A}(z) + u_A(z)) \leq j = u_{A}(x).
	\]

	For the other direction, assume $j \geq 1$ and let $Q \in \mathcal{L}$ be given. Since 
	$x \in \inte(E_{j}(A))$, by Theorem \ref{thm:cone_efficient}, there is $z \in A_{j-1} \cap (x + \inte(Q))$
	and so 
	\[
	\inf_{Q \in \mathcal{L}} \sup_{z \in A \cap (x + \inte(Q_{}))} (1_{A}(z) + u_A(z)) \geq 1 + (j-1) = u_{A}(x).
	\]
\end{proof}

\subsection{Duality} \label{sec:duality} In the sequel, convex duality will play a recurring role.  Thus, in this section, we describe some basic observations related to duality that will be useful in what follows.

First, the dual cones $\{Q^{*}_{p} \, \mid \, p \in \mathcal{N}^{*}\}$ determine the directions in which the Hamiltonian is nonzero.  These are defined by
\begin{equation}
Q^{*}_{p} = \bigcap_{v \in Q_{p}} \{\xi \in \R^{2} \, \mid \, \langle \xi, v \rangle \leq 0\}.
\end{equation}

Next, in the analysis of the Hamiltonian $\bar{H}_{\varphi}$, it will be convenient to define dual bases $\{(v_{p}^{*},w_{p}^{*}) \, \mid \, p \in \mathcal{N}^{*}\}$ by the rule
\begin{equation*}
\langle v_{p}^{*}, v_{p} \rangle = \langle w_{p}^{*}, w_{p} \rangle = 1, \quad \langle v_{p}^{*}, w_{p} \rangle = \langle w_{p}^{*}, v_{p} \rangle = 0.
\end{equation*}
These are well-defined since the pairs in $\{(w_{p},v_{p}) \, \mid \, p \in \mathcal{N}^{*}\}$ are themselves bases.  

Note that the dual bases provide coordinates for $\R^{2}$ in the sense that, given any $p \in \mathcal{N}^{*}$,
\begin{equation} \label{eq:coordinates}
\xi = \langle \xi, v_{p}^{*} \rangle v_{p} + \langle \xi, w_{p}^{*} \rangle w_{p} \quad \text{for each} \, \, \xi \in \R^{2}.
\end{equation}
Using these coordinates, we obtain an alternative formula for the expressions appearing in the definition of $\bar{H}_{\varphi}$.  

\begin{prop} \label{prop:formula_thing} For each $\xi \in \R^{2}$ and $p \in \mathcal{N}^{*}$,
	\begin{align} \label{eq:planar_thing}
		\langle \xi^{\perp} ,v_{p}^{*} \rangle = (w_{p} \times v_{p})^{-1} \langle \xi, w_{p} \rangle \quad \text{and} \quad \langle \xi^{\perp}, w_{p}^{*} \rangle = -(w_{p} \times v_{p})^{-1} \langle \xi, v_{p} \rangle.
	\end{align}
In particular,
	\begin{equation} \label{eq:duality}
	\frac{\langle \xi,v_{p} \rangle \langle \xi, w_{p} \rangle}{|v_{p} \times w_{p}|} = - |v_{p} \times w_{p}| \langle \xi^{\perp}, v_{p}^{*} \rangle \langle \xi^{\perp}, w_{p}^{*} \rangle.
	\end{equation}
\end{prop}  

\begin{proof}  Observe that we can write
	\begin{equation*}
	\langle \xi, v_{p} \rangle = \langle -(\xi^{\perp})^{\perp}, v_{p} \rangle = -\langle \langle \xi^{\perp}, v_{p}^{*} \rangle v_{p}^{\perp} + \langle \xi^{\perp}, w_{p}^{*} \rangle w_{p}^{\perp}, v_{p} \rangle = - \langle \xi^{\perp}, w_{p}^{*} \rangle \langle w_{p}^{\perp}, v_{p} \rangle.
	\end{equation*}
	A similar computation shows that $\langle \xi, w_{p} \rangle = - \langle \xi^{\perp}, v_{p}^{*} \rangle \langle v_{p}^{\perp}, w_{p} \rangle$.  At the same time, 
	\begin{equation*}
	\langle w_{p}^{\perp}, v_{p} \rangle = -\langle v_{p}^{\perp}, w_{p} \rangle = w_{p} \times v_{p}. 
	\end{equation*}
	Combining these formulas gives \eqref{eq:planar_thing}, from which \eqref{eq:duality} follows directly.\end{proof}  

Next, we show how the identities in \eqref{eq:planar_thing} imply an important bijective correspondence between $Q_{p}^{*}$ and $\mathcal{Q}_{p}$.

\begin{prop} \label{prop:key-duality} If $p \in \mathcal{N}^{*}$ and $\xi \in \R^{2}$, then $-\xi \in Q_{p}^{*}$ (resp.\ $-\xi \in \inte(Q_{p}^{*})$) if and only if $\xi^{\perp} \in \mathcal{Q}_{p}$ (resp.\ $\xi^{\perp} \in \inte(\mathcal{Q}_{p})$).  \end{prop}

\begin{proof}  In what follows, it will be important to recall that the cone $\mathcal{Q}_{p}$ and the basis $\{w_{p},v_{p}\}$ are related via the equation \eqref{eq:normball-facets}.

By definition of $Q_{p}$, $-\xi \in Q_{p}^{*}$ if and only if $\langle \xi, v_{p} \rangle \geq 0$ and $\langle \xi, w_{p} \rangle \geq 0$.  Since $w_{p} \times v_{p} > 0$ by the choice of $\{w_{p},v_{p}\}$, the formulas in \eqref{eq:planar_thing} imply that this occurs if and only if $\langle \xi^{\perp}, v_{p}^{*} \rangle \geq 0$ and $\langle \xi^{\perp}, w_{p}^{*} \rangle \leq 0$.  That is, by definition of $\mathcal{Q}_{p}$, $-\xi \in Q_{p}^{*}$ if and only if $\xi^{\perp} \in \mathcal{Q}_{p}$.  
	
	The previous argument works just as well if the interiors are considered instead.\end{proof}  

The last proposition helps us to unpack the formula \eqref{eq:effective-hamiltonian}.  Among the consequences, it shows that $\bar{H}_{\varphi}$ is never coercive.

\begin{prop} \label{prop:noncoercive}  (i)  Given $\xi \in \R^{2}$, there are at most two $p \in \mathcal{N}^{*}$ such that 
	\begin{equation*}
	\frac{\langle \xi, v_{p} \rangle \langle \xi, w_{p} \rangle}{|v_{p} \times w_{p}|} > 0.
	\end{equation*} 
	Furthermore, such $p$ necessarily satisfy $\xi^{\perp} \in \inte(\mathcal{Q}_{p}) \cup (-\inte(\mathcal{Q}_{p}))$.  
	
	(ii)  Given $\xi \in \R^{2}$,
	\begin{equation*}
	\bar{H}_{\varphi}(\xi) = 0 \quad \iff \quad \xi^{\perp} \in \cone(\mathcal{E}).
	\end{equation*}  
\end{prop}  

\begin{proof}  (i) If $\langle \xi, v_{p} \rangle \langle \xi, w_{p} \rangle > 0$, then $\xi \in \inte(Q_{p}^{*}) \cup (-\inte(Q_{p}^{*}))$.  Hence the previous result implies $\xi^{\perp} \in \inte(\mathcal{Q}_{p}) \cup (-\inte(\mathcal{Q}_{p}))$.  Since the sets $\{\inte(\mathcal{Q}_{p}) \, \mid \, p \in \mathcal{N}^{*}\}$ are disjoint and $\mathcal{Q}_{-p} = - \mathcal{Q}_{p}$, this determines $p$ up to negation.
	
	(ii)  Notice that if $\xi \in \inte(\mathcal{Q}_{p})$ for some $p \in \mathcal{N}^{*}$, then $\varphi(\xi)^{-1} \xi \notin \mathcal{E}$ since $\{\varphi = 1\}$ is flat in $\inte(\mathcal{Q}_{p})$.  Accordingly, $\cone(\mathcal{E}) \cap \inte(\mathcal{Q}_{p}) = \emptyset$ for each $p \in \mathcal{N}^{*}$.  We conclude by combining this last observation with (i).   \end{proof}  

Combining everything we have done in this section, we obtain the following alternative formula for $\bar{H}_{\varphi}$:
\begin{equation} \label{eq:alternative}
\bar{H}_{\varphi}(\xi) = \left\{ \begin{array}{r l}
-|v_{p} \times w_{p}| \langle \xi^{\perp}, v_{p}^{*} \rangle \langle \xi^{\perp}, w_{p}^{*} \rangle, & \text{if} \, \, \xi^{\perp} \in \mathcal{Q}_{p} \, \, \text{for some} \, \, p \in \mathcal{N}^{*}, \\
0, & \text{otherwise.}
\end{array} \right.
\end{equation}
This formula suggests that $\bar{H}_{\varphi}$ is more naturally interpreted as a function of the tangent vector $n_{\partial E_{t}}^{\perp}$ rather than the normal vector $n_{\partial E_{t}}$ in \eqref{eq:normal-velocity}.  Note that this explains the otherwise counter-intuitive $90^{\circ}$ discrepancy between the middle images in Figures \ref{fig:pareto-peeling} and \ref{fig:norm-balls}.

Before proceeding further to establish the continuity of $\bar{H}_{\varphi}$, let us return to Corollary \ref{cor:-constrained-convex-hull}, the proof of which is facilitated by the correspondence between $Q_{p}^{*}$ and $\mathcal{Q}_{p}$.

	\begin{proof}[Proof of Corollary \ref{cor:-constrained-convex-hull}] First, we prove the easier ``if" direction.  Suppose that $A \subseteq \mathbb{R}^{2}$ is compact and $x \in \mathbb{R}^{2}$ satisfies conditions (i) and (ii) of the corollary.  To see that $x \in \mathcal{P}(A)$, we invoke Theorem \ref{thm:cone_efficient}.  In view of (i), we only need to show that, given any $q \in \mathcal{E}$, there is an $a \in A$ such that $a \in x + \inte(H_{q^{\perp}})$.  Yet this follows directly from the fact that $x \in \inte(\conv(A))$ and classical separation theorems for convex sets (see \cite[Section 11]{rockafellar} or \cite[Section 1.3]{schneider}).
	
	Next, we prove the ``only if" direction.  Suppose that $A \subseteq \mathbb{R}^{2}$ is compact and $x \in \mathcal{P}(A)$.  By Theorem \ref{thm:cone_efficient}, $x$ satisfies (i) so it only remains to prove that (ii) also holds.  Toward this end, again by separation theorems, it suffices to show that if $q \in \mathbb{R}^{2} \setminus \{0\}$, then $A \cap (x + \inte(H_{q^{\perp}})) \neq \emptyset$.  The theorem implies that this is true if $q \in \mathcal{E}$, hence also if $q \in \cone(\mathcal{E})$ by homogeneity, so it only remains to consider the case when $q \notin \cone(\mathcal{E})$.  
	
	Suppose that $q \notin \cone(\mathcal{E})$.  By \eqref{eq:decomposition_round_flat}, there is a $p \in \mathcal{N}^{*}$ such that $q \in \mathcal{Q}_{p}$.  Thus, $q^{\perp} \in Q_{p}^{*}$ holds by Proposition \ref{prop:key-duality}.  By definition of the dual cone $Q_{p}^{*}$, this means that $-H_{q^{\perp}} \supseteq Q_{p}$, hence $H_{q^{\perp}} \supseteq -Q_{p} = Q_{-p}$ by \eqref{eq:negation-thing}.  Yet $x \in \mathcal{P}(A)$ so the theorem implies $A \cap (x + \inte(Q_{-p})) \neq \emptyset$.  Therefore, $A \cap (x + \inte(H_{q^{\perp}}) \neq \emptyset$, as claimed.   \end{proof}

\subsection{Continuity of the Hamiltonian}  In this section, we show that the effective Hamiltonian $\bar{H}_{\varphi}$ given by \eqref{eq:effective-hamiltonian} is a continuous function for an arbitrary norm $\varphi$. 


We start by proving that $\bar{H}_{\varphi}$ is locally bounded.  That is a consequence of the next lemma.

\begin{lemma}\label{lemma:local-boundedness}  For each $p \in \mathcal{N}^{*}$, if $\theta(p) = (\pi - \arcsin(w_{p} \times v_{p}))/2$, then
	\begin{equation*}
	\frac{\langle \xi, v_{p} \rangle \langle \xi, w_{p} \rangle}{|v_{p} \times w_{p}|} \leq \frac{\|\xi\|^{2} \tan(\theta(p))}{2} \quad \text{for each} \, \, \xi \in Q_{p}^{*}.
	\end{equation*}
\end{lemma}  

As we will see below, the angle $\theta(p)$ is small for all but finitely many $p \in \mathcal{N}^{*}$.  This will be used to prove that $\bar{H}_{\varphi}$ is continuous.

\begin{proof}  For convenience, we write $v = v_{p}$, $w = w_{p}$, and $\theta = \theta(p)$.  Since rotations don't change inner or cross products, we can rotate the plane so that
	\[
	w = (\cos(\theta),\sin(\theta)) \quad \mbox{and}  \quad v = (-\cos(\theta),\sin(\theta)).
	\]
	  Note that, after this rotation, we have $\xi \in Q_{p}^{*} \setminus \{0\}$ only if $-\xi = \|\xi\| (\cos(\psi),\sin(\psi))$ for some $\psi \in (0,\pi)$.  Henceforth, we will restrict attention to such angles $\psi$.
	
	Since $\|v\| = \|w\| =1$, our assumptions imply that
	\begin{equation*}
	\frac{\langle \xi, v \rangle \langle \xi, w \rangle}{|v \times w|} = \frac{\|\xi\|^{2} \cos(\pi - \psi - \theta) \cos(\psi - \theta)}{\sin(\pi - 2\theta)}.
	\end{equation*}
	Rewriting the numerator, we see that
	\[
	\cos(\pi - \psi - \theta) \cos(\psi - \theta) = (1/2)(-\cos(2 \theta)-\cos(2 \psi)), 
	\]
	which is maximized at $\psi = \frac{\pi}{2}$. 	Thus, 
	\begin{align*}
	\frac{\langle \xi, v \rangle \langle \xi, w \rangle}{|v \times w|} \leq \frac{\|\xi\|^{2} \cos\left(\frac{\pi}{2} - \theta \right) \cos \left(\frac{\pi}{2} - \theta \right)}{\sin(\pi - 2 \theta)} = \frac{\|\xi\|^{2} \tan(\theta)}{2}.
	\end{align*}
\end{proof}  

\begin{prop} \label{prop:hj-finite-set}
	  For each $R, M > 0$, let $\mathcal{N}^{*}_{R,M} \subseteq \mathcal{N}^{*}$ denote the subset 
	\begin{equation*}
	\mathcal{N}^{*}_{R,M} = \{p \in \mathcal{N}^{*} \, \mid \, |v_{p} \times w_{p}|^{-1} \langle \xi, v_{p} \rangle \langle \xi, w_{p} \rangle \geq M \quad \text{for some} \, \, \xi \in B(0,R)\}.
	\end{equation*}
	Then $\mathcal{N}^{*}_{R,M}$ is a finite set.  \end{prop}  

\begin{proof}  Since $\{\varphi^{*} \leq 1\}$ has finite perimeter, the disjoint line segments $\{\mathcal{Q}_{p} \cap \{\varphi^{*} = 1\} \, \mid \, p \in \mathcal{N}^{*}\}$ have summable lengths.  From this, it follows that, for each $\delta > 0$, the sets $\mathcal{N}^{*}(\delta)$ given by 
	\begin{equation} \label{eq:finite-set}
	\mathcal{N}^{*}(\delta) = \left \{p \in \mathcal{N}^{*} \, \mid \, |v \times w| > \delta \, \, \text{for some} \, \, v, w \in \mathcal{Q}_{p} \right\} 
	\end{equation}
	are finite.  
	
	In view of \eqref{eq:finite-set}, it only remains to show that, for each $R, M > 0$, there is a $\delta > 0$ such that $\mathcal{N}^{*}_{R,M} \subseteq \mathcal{N}^{*}(\delta)$.  However, this follows from Lemma \ref{lemma:local-boundedness}.  Indeed, if $\xi \in B(0,R)$ and $|v_{p} \times w_{p}|^{-1} \langle \xi, v_{p} \rangle \langle \xi, w_{p} \rangle \geq M$, then
	\begin{equation*}
	\frac{2M}{\tan(\theta(p))} \leq \|\xi\|^{2} \leq R^{2}
	\end{equation*}
	where $\tan(\theta(p)) = (\pi - \arcsin(w_{p} \times v_{p}))/2$.  From this, we deduce that $\tan(\theta(p)) \geq \frac{2 M}{R^{2}}$ and this readily implies $p \in \mathcal{N}^{*}(\delta)$ for some $\delta > 0$ depending on $M$ and $R$.         \end{proof}  

\begin{prop}  $\bar{H}_{\varphi}$ is continuous.  \end{prop}  

\begin{proof}  Note that $\bar{H}_{\varphi}$ is a lower semicontinuous function, being the supremum of a family of continuous functions.
	
	To show that it is continuous, it suffices to prove that $\bar{H}_{\varphi}$ restricts to a continuous function in $B(0,R)$ for each $R > 0$.  Given such an $R$ and a sequence $(\xi_{n})_{n \in \N} \subseteq B(0,R)$ converging to some $\xi \in \R^{2}$, there are two possibilities: either $\bar{H}_{\varphi}(\xi_{n}) \to 0$ or else $\bar{H}_{\varphi}(\xi_{n}) \geq M$ for some $M > 0$ and sufficiently large $n$.
	
	In the second case, Proposition \ref{prop:hj-finite-set} implies that $(\xi_{n})_{n \geq N} \subseteq \cup_{p \in \mathcal{N}^{*}_{R,M}} Q_{p}^{*}$.  At the same time, since $\mathcal{N}_{R,M}^{*}$ is finite, $\bar{H}_{\varphi}$ restricts to a continuous function in that set by Proposition \ref{prop:key-duality} and \eqref{eq:alternative}.  Therefore, $\bar{H}_{\varphi}(\xi) = \lim_{n \to \infty} \bar{H}_{\varphi}(\xi_{n})$.  
	
	In the other case, we know that $\lim_{n \to \infty} \bar{H}_{\varphi}(\xi_{n}) = 0$.  Therefore, by lower semicontinuity, $\bar{H}_{\varphi}(\xi) \leq 0$.  At the same time, $\bar{H}_{\varphi}$ is a non-negative function so this implies $0 = \bar{H}_{\varphi}(\xi)$, and hence $\bar{H}_{\varphi}(\xi) = \lim_{n \to \infty} \bar{H}_{\varphi}(\xi_{n})$.    \end{proof}

	\subsection{Affine invariance} \label{sec:affine-invariance}
We note that Pareto peeling has a certain invariance with respect
to linear transformations of the plane. 
\begin{lemma} \label{lemma:affine-invariance}
	Let $A$ be a finite set of points in $\R^2$, then for any bijective linear map $L:\R^2 \to \R^2$
	\[
	x \in \inte(\mathcal{P}(A)) \iff L(x) \in \inte(\mathcal{P}_L(L(A))), 
	\]
	where $\mathcal{P}_L$ denotes the Pareto hull with respect to the norm, $\varphi \circ L$.
\end{lemma}
\begin{proof}
	As $L$ is bijective, it suffices to prove one direction.  If $x \in \mathcal{P}(A)$, 
	then by Theorem \ref{thm:cone_efficient}, for every cone $Q \in \{ Q_p, H_{q^\perp}\}$, 
	\[
	A \cap (x + \inte(Q)) \not = \emptyset.
	\]
	Given such a $Q$, fix $y \in A$ and a $q_{Q} \in Q$ such that
	\[
	y = x + q_{Q}.
	\]
	Since $L$ is linear, 
	\[
	L(y) = L(x) + L(q_Q)
	\]
	and $L(y) \in L(A)$, $L(q_Q) \in \inte(L(Q))$, meaning
	\[
	L(A) \cap (L(x) + \inte(L(Q))) \not = \emptyset ,
	\]
	and we conclude by another application of Theorem \ref{thm:cone_efficient}. 
	\end{proof}

\begin{cor} \label{cor:affine_invariance}
	For any finite set $A \subseteq \R^2$ and any bijective linear map $L:\R^2 \to \R^2$, the height functions are related by $h_A(x) = h_{L(A)}(L(x))$.
\end{cor} 



\section{Preliminaries from Nondominated Sorting} \label{sec:q-nds}

In this section, for the sake of completeness, we recall the fundamental results on nondominated sorting that will be needed in the rest of the paper.  Due to the form of the dynamic programming principle satisfied by the height function, \ie, \eqref{eq:pareto-peeling-dpp}, it will be necessary to present the results in a more general setting in which the standard cone $[0,\infty)^{2}$ is replaced by a flat cone $Q_{p}$ for some $p \in \mathcal{N}^{*}$.  At the level of nondominated sorting, this change presents no new difficulties.

At the end of the section, we explain how nondominated sorting can be regarded as an infinite volume limit of Pareto hull peeling, 
and relate the two continuum PDE.

\begin{figure}
	\includegraphics[width=0.5\textwidth]{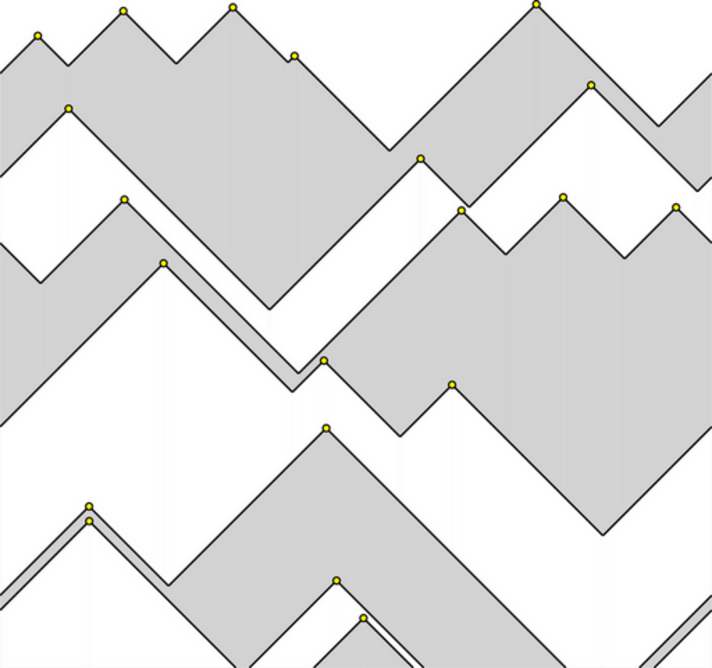}
	\caption{$Q$-nondominated sorting of a Poisson point cloud where $Q = \{x \in \R^2: \|x\|_{\infty} = x_2\}$. The shading indicates 
		alternating layers.} \label{fig:q-nds}
\end{figure}

\subsection{Nondominated Sorting}

We recall the definition of nondominated sorting.  As we will see below, nondominated sorting will be useful in characterizing the local behavior of the limiting height function $\bar{u}$ in the nondegenerate directions of the Hamiltonian.  In this sense, in the language of homogenization theory, it serves as a ``cell problem" for Pareto hull peeling.

Given a finite set of distinct points, $A \subseteq \R^2$, nondominated sorting arranges the set of points into layers by repeatedly removing or peeling the set of minimal elements. Specifically, let 
\begin{equation} \label{eq:component-wise-order}
x \leq y \iff x_1 \leq y_1 \quad \mbox{and} \quad x_2 \leq y_2
\end{equation}
denote the component-wise partial order.  Given $x, x' \in \mathbb{R}^{2}$, we say $x'$ {\it dominates} $x$
if $x' \leq x$.
Write $\mathcal{D}(A)$ for the set of all points in $\mathbb{R}^{2}$ that are dominated by some point in $A$.  If no such point $x' \in A$ exists, we say $x$ is {\it nondominated} relative to $A$.  If $A$ is empty, designate $\mathcal{D}(A) = \emptyset$. 

Following \cite{calder2014continuum,calder2014hamilton,calder2015pde,calder2017numerical}, we define {\it nondominated sorting} of a finite set $A \subseteq \mathbb{R}^{2}$ as follows: define $\{S_{j}(A)\}_{j \in \mathbb{N}}$ by
\begin{equation} \label{eq:nds-hulls}
S_1(A) = \mathcal{D}(A) \quad \mbox{ and } \quad S_{j+1}(A) = \mathcal{D}(A \cap \inte(S_j(A))).
\end{equation}
The nondominated sorting {\it depth function} of $A$ is defined to be 
\begin{equation} \label{eq:nds-depth}
s_{A}(x) = \sum_{n=1}^{\infty} 1_{\inte(S_n(A))}.
\end{equation}

\begin{remark}  Note that our definition of the depth function differs from that slightly from that in \cite{calder2014continuum,calder2014hamilton,calder2015pde,calder2017numerical}: the definition in those works involves the sum of $1_{S_{n}(A)}$ rather than $1_{\inte(S_{n}(A))}$.  Whereas the version in those works is upper semi-continuous, ours is lower semi-continuous.  The two definitions can be recovered from one another by taking upper and lower semi-continuous envelopes.  \end{remark}

Calder-Esedoglu-Hero showed that if $A$ consists of randomly scattered points, then, in the large sample limit, $s_A$ converges to the solution of a Hamilton-Jacobi equation \cite{calder2014continuum,calder2014hamilton,calder2015pde,calder2017numerical}.
A rate of convergence was recently established by Cook-Calder \cite{cook2022rates}.
These analyses relied on an equivalence between nondominated sorting of a set $A$ and the {\it longest chain} in $A$, a longest, totally ordered subset of $A$, $x_1 \leq \cdots \leq x_n$, $\{x_i\} \subseteq A$. Denote the length of the longest chain in $A$ by $\ell(A)$ and observe that 
\[
s_A(x) = \ell( (-\infty,x)^2 \cap A).
\]
Using this equivalence, one readily deduces that $s_{A}$ can be computed via a form of dynamic programming:
	\begin{equation} \label{eq:dpp-nds}
		s_{A}(x) = \sup_{y < x} \left( s_{A}(y) + 1_{A}(y) \right).
	\end{equation}
(Here we write $y < x$ if $y_{1} < x_{1}$ and $y_{2} < x_{2}$.)

\subsection{$Q$-nondominated Sorting}  In order to use the asymptotics of nondominated sorting in our analysis of Pareto hull peeling, it will be natural to generalize the former algorithm to include other partial orders on $\mathbb{R}^{2}$.

The definition of nondominated sorting is easily extended beyond \eqref{eq:component-wise-order} to any partial order on $\R^2$. Recall that any proper cone 
$Q \subseteq \R^2$ induces a partial order \cite{boyd2004convex}. A cone $Q$ is {\it proper} if it is convex, closed, 
has nonempty interior, and is pointed: $Q \cap (-Q) = \{0\}$. In particular, 
given any proper cone $Q$, let 
\begin{equation} \label{eq:cone-order}
x \leq_{Q} y \iff (y-x) \in Q
\end{equation}
denote the associated partial order. Observe that when $Q = [0,\infty)^2$, \eqref{eq:cone-order} coincides
with \eqref{eq:component-wise-order}.

By replacing \eqref{eq:component-wise-order} in nondominated sorting by \eqref{eq:cone-order}, 
we get {\it $Q$-nondominated sorting}, see Figure \ref{fig:q-nds}.  In particular, the $Q$-dominated points $\mathcal{D}^{Q}(A)$ associated with a finite set $A$ are defined by 
	\begin{equation*}
		\mathcal{D}^{Q}(A) = \bigcup_{x' \in A} \{x \in \mathbb{R}^{2} \, \mid \, x' \leq_{Q} x\}.
	\end{equation*}
Hence $Q$-nondominated sorting is described by the sets $\{S_{j}^{Q}(A)\}_{j \in \mathbb{N}}$ given by 
	\begin{equation*}
		S_{1}^{Q}(A) = \mathcal{D}^{Q}(A), \quad S_{j + 1}^{Q}(A) = \mathcal{D}^{Q}(A \cap \inte(S_{j}^{Q}(A))).
	\end{equation*}
The depth function $s_A^{Q}$ and longest chain operator $\ell^{Q}(A)$ for $Q$-nondominated sorting are defined entirely analogously.  This leads to the following {\it dynamic programming principle} for the depth function,
\begin{equation} \label{eq:nds-dpp}
s^{Q}_A(x) = \sup_{y \in (x - \inte(Q))} \left( s_{A}^{Q}(y) + 1_{A}(y) \right)
\end{equation}
and a representation in terms of the $Q$-longest chain
	\begin{equation} \label{eq:q-longest-chain}
		s^{Q}_{A}(x) = \ell^{Q}((x - Q) \cap A).
	\end{equation}

\begin{remark}
	Topological properties of $Q$-nondominated point sets have been analyzed previously in an optimization context \cite{luc1985structure, dinh2005generalized}.
\end{remark}

\subsection{Convergence of the longest chain}
In this section we record known results on the longest chain of Poisson points in rectangles and simplicial domains.
Let $\mathcal{B}$ denote the set of bounded coordinate rectangles in $\R^2$ and let $X_{n f}$ denote a Poisson process of intensity $n f$ in $\R^2$.

\begin{prop}[\cite{hammersley1972few, calder2015pde}] \label{prop:longest-chain-convergence}
	Fix a non-negative $f \in L^{\infty}_{\text{loc}}(\mathbb{R}^{2})$ and, for each $n \in \mathbb{N}$, let $X_{nf}$ be a Poisson process in $\mathbb{R}^{2}$ of intensity $nf$. On an event of probability 1, 
	for all $B \in \mathcal{B}$,
	\begin{equation}
	\limsup_{n \to \infty} n^{-1/2} \ell(X_{n f} \cap B) \leq 2 \Big(\sup_{B} f\Big)^{\frac{1}{2}} |B|^{1/2} 
	\end{equation}
	and
	\begin{equation}
	\liminf_{n \to \infty} n^{-1/2} \ell(X_{n f} \cap B) \geq 2 \left(\inf_{B} f \right)^{\frac{1}{2}} |B|^{1/2}.
	\end{equation}
	
\end{prop}
\begin{proof}[Sketch of proof].
	We start with the case $f \equiv 1$. By the subadditive ergodic theorem \cite{hammersley1972few}, there is a constant $c_{2} \geq 0$ such that, for each $x \in \Q^2$, on an event of probability 1, $\lim n^{-1/2} \ell(X_n \cap (x + [0,1]^2)) = c_2$.  It is known that $c_{2} = 2$; see, for instance, \cite{aldous_diaconis}. In particular, by scaling and taking the intersection over a countable number of probability 1 events, 
	for each $x, x' \in \Q^2$
	\[
	\lim_{n \to \infty} n^{-1/2} \ell(X_n \cap (x + R_{x'})) = 2 |R_{x'}|^{1/2},
	\]
	where $R_{x'} = [0,x']$ (recall from \eqref{eq:notation-for-squares} that this denotes an axis-aligned square), on an event of probability 1. By approximation, this implies 
	\[
	\lim_{n \to \infty} n^{-1/2} \ell(X_n \cap B) = 2 |B|^{1/2},
	\]
	for any $B \in \mathcal{B}$ on an event of probability 1. The extension to arbitrary $f$ uses the standard coupling of Poisson processes \cite{kingman1992poisson}. 
\end{proof}
In the next section, we will establish the subsolution property of the limiting height function by employing an elegant argument of Calder \cite{calder2016direct}.  In order to do so, it is necessary to observe that the asymptotics of the longest chain from the previous result are unchanged if a cube is replaced by a suitable simplex.  Specifically, for $v \in (0, \infty)^2$, denote the simplex 
\begin{equation}
S_{v} := \{ x \in (-\infty,0]^2 \, \mid \, 1 + \langle x , v \rangle \geq 0\}.
\end{equation}
The next result shows that the asymptotics of the longest chain in $S_{v}$ are consistent with the result for rectangles.

\begin{prop}[\cite{calder2015pde, calderminicourse}] \label{prop:nds-consistency}
	Fix a non-negative $f \in L^{\infty}_{\text{loc}}(\mathbb{R}^{2})$ and, given $n \in \mathbb{N}$, let $X_{nf}$ be a Poisson process in $\mathbb{R}^{2}$ of intensity $nf$.  With probability one, for any $v \in (0,\infty)^2$ and $x \in \R^2$,
	\begin{equation}
	\limsup_{n \to \infty} \frac{\ell((x + S_v) \cap X_{nf})}{n^{1/2}}  \leq \left( \frac{\sup_{x + S_v} f}{v_1 v_2}\right)^{1/2}.
	\end{equation}
\end{prop}

\subsection{Isomorphism between versions of nondominated sorting} \label{sec:isomorphism}
We note a useful change of variables which will allow us to translate between 
$Q$-nondominated sorting and standard nondominated sorting. This change of variables immediately leads to 
versions of Propositions \ref{prop:longest-chain-convergence} and \ref{prop:nds-consistency} for $Q$-nondominated sorting.

Let $Q \subseteq \R^2$ be a proper cone
and let $v,w \in \mathbb{R}^{2}$ denote its extremal directions, \ie, 
\[
Q = \{ aw + bv \, \mid \, a,b \geq 0 \}.
\]
As $Q$ is proper, $\{v,w\}$ form a basis of $\R^2$ and we may define a linear bijection $L_Q:\R^2 \to \R^2$ by 
\begin{equation} \label{eq:q-linear-map}
L_Q(aw + bv) = (a,b).
\end{equation}
Note that $L_Q$ preserves the order:
\[
L_Q(x) \leq L_Q(y) \iff x \leq_{Q} y
\]
for $x, y \in \R^2$. Therefore, given a finite set $A \subseteq \R^2$,  
\begin{equation}
x \in \mathcal{D}^{Q}(A) \iff L_Q(x) \in \mathcal{D}(L(A))
\end{equation}
These observations imply the following. 
\begin{lemma} \label{lemma:nds-change-of-variables}
	If $X_{g}$ is a Poisson point process of intensity $g \in C(\R^2)$ then
	\[
	\ell^{Q}(X_g \cap A) = \ell(L_Q(X_{g} \cap A)) \overset{\mathcal{D}}{=} \ell(X_{g'} \cap L_Q(A))
	\] 
	for all finite subsets $A$ of $\R^2$ where $g' = |v \times w| g$.
\end{lemma} 

It will be useful to know later that $L_{Q}$ has operator norm given by 
	\begin{equation}
		\|L_{Q}\| = (1 - |\langle v,w \rangle|)^{-1/2}. \label{eq:operator-norm}
	\end{equation}  
Indeed, given $a,b \in \mathbb{R}$, Young's inequality implies that
	\begin{equation*}
		\|aw + bv\|^{2} = a^{2} + b^{2} + 2 ab \langle v, w \rangle \geq (1 - |\langle v, w \rangle|) \|L_{Q}(aw + bv)\|^{2}
	\end{equation*}
with equality for well-chosen $a,b$.  

\subsection{Interpretation as Pareto Hull Peeling}
\begin{figure}
	\includegraphics[width=0.35\textwidth]{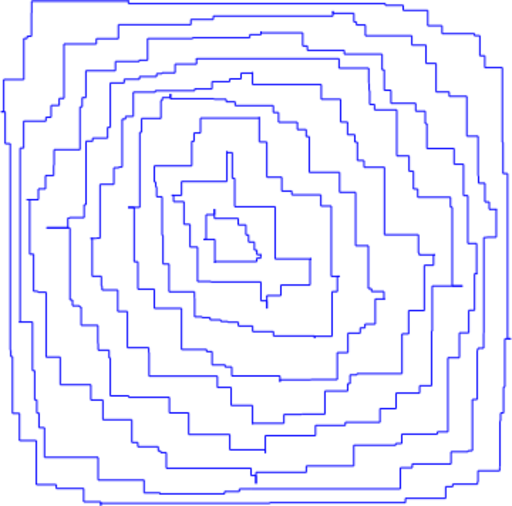} \qquad
	\includegraphics[width=0.35\textwidth]{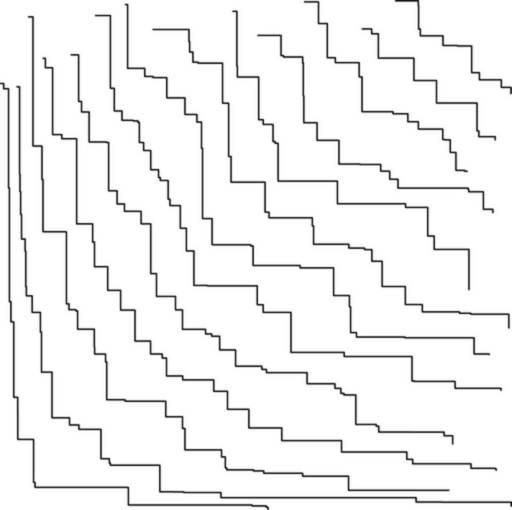} \\
	\includegraphics[width=0.3\textwidth]{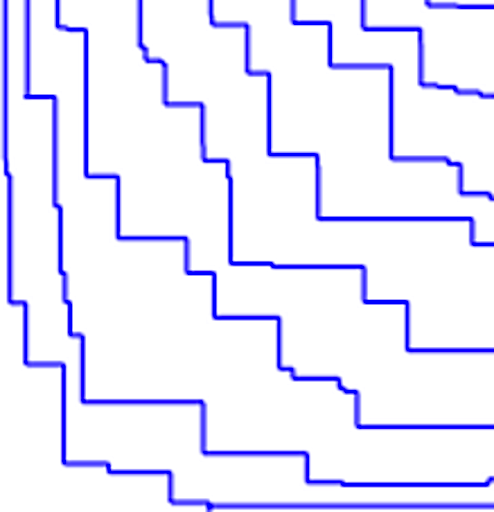} \qquad
	\includegraphics[width=0.3\textwidth]{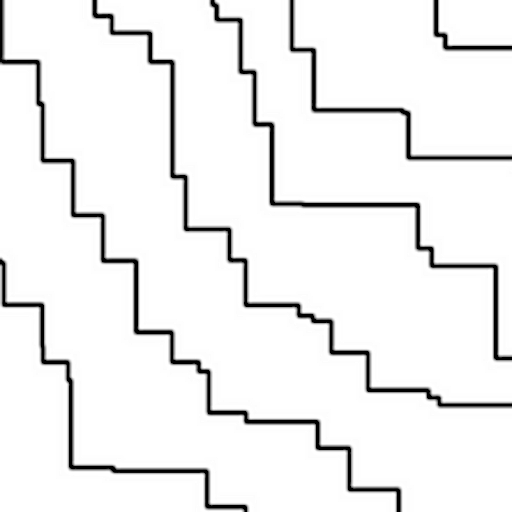} 
	\caption{On the left is $\ell_1$-Pareto peeling and on the right non-dominated sorting of 1000 Poisson random points. 
		The bottom row is a zoomed in piece of the top row, meant to emphasize the similarity of the local structure. }
\end{figure}

 Nondominated sorting and its scaling limit can be interpreted as special solutions of Pareto hull peeling and the PDE \eqref{eq:eikonal}.  This connection plays a fundamental role in the results that follow so it will be sketched here.  

First, let $X$ be a unit intensity Poisson point process in $[0,\infty)^{2}$ and let $\varphi$ be the $\ell^{1}$ norm.  Let $u^{(N)}$ be the height function of Pareto hull peeling applied to the restriction $X^{(N)} = X \cap [0,N]^{2}$.  The dynamic programming principle (Proposition \ref{prop:dpp}) implies that $u_{1}^{(N)}$ satisfies 
	\begin{equation*}
		u_{1}^{(N)}(x) = \min_{p \in \mathcal{N}^{*}} \sup_{y \in x + \inte(Q_{p})} u_{1}^{(N)}(y) + 1_{X}(y),
	\end{equation*}
where, in this case, $\mathcal{N}^{*} = \{(1,1),(1,-1),(-1,1),(-1,-1)\}$ and $Q_{p} = \{x \in \mathbb{R}^{2} \, \mid \, p_{1} x_{1} \geq 0, \, \, p_{2} x_{2} \geq 0\}$. 

As was observed by Calder \cite{calderminicourse}, when $N \to \infty$, $u_{1}^{(N)}$ converges to a function $u_{1} : [0,\infty)^{2} \to [0,\infty)$, and, due to the geometry of the domain and the properties of the Poisson process $X$, the dynamic programming principle simplifies
	\begin{equation*}
		u_{1}(x) = \sup_{y \in x - (0,\infty)^{2}} u_{1}(y) + 1_{X}(y).
	\end{equation*}
Thus, since this is nothing but the dynamic programming principle \eqref{eq:dpp-nds}, a straightforward comparison argument implies that the limit $u_{1}$ is precisely the depth function $s_{X}$.  In this way, nondominated sorting can be understood as the infinite volume limit of Pareto hull peeling in the quadrant $[0,\infty)^{2}$.

Next, upon rescaling, $u_{1}$ converges to the solution of a PDE.  Precisely, let $X_{n} = n^{-\frac{1}{2}} X$ be the rescaled point process, which is Poisson of intensity $n$, and let $u_{n}(x) = u_{1}(\sqrt{n} x)$, which is readily shown to be the depth function $s_{X_{n}}$ of $X_{n}$.  The results of \cite{calder2014continuum,calder2014hamilton,calder2015pde,calder2017numerical} therefore imply that, after normalizing by $\sqrt{n}$, 
	\begin{equation*}
		n^{-\frac{1}{2}} u_{n} \to \bar{u} \quad \text{with probability one,}
	\end{equation*}
where $\bar{u}$ is the unique solution of the Hamilton-Jacobi equation
	\begin{equation*}
		\left\{ \begin{array}{r l}
			\bar{u}_{x_{1}} \bar{u}_{x_{2}} = 1 & \text{in} \, \, (0,\infty)^{2}, \\
			\bar{u} = 0 & \text{on} \, \, \partial [0,\infty)^{2}, \\
			\bar{u}_{x_{1}} \geq 0, \, \, \bar{u}_{x_{2}} \geq 0 & \text{in} \, \, (0,\infty)^{2},
		\end{array} \right.
	\end{equation*}
where the last condition is interpreted to mean that $\bar{u}$ is nondecreasing in each argument.  In view of the central equation of interest in this work \eqref{eq:eikonal} and the identity $\bar{H}_{\varphi}(p') = |p_{1}'p_{2}'|$ that holds when $\varphi$ is the $\ell^{1}$ norm, it is natural to recast this equation for $\bar{u}$ in the form
	\begin{equation} \label{eq:special_corner_solution}
		\left\{ \begin{array}{r l}
			\bar{H}_{\varphi}(D\bar{u}) = 1 & \text{in} \, \, (0,\infty)^{2}, \\
			\bar{u} = 0 & \text{on} \, \, \partial [0,\infty)^{2}.
		\end{array} \right.
	\end{equation}
Indeed, it is straightforward to show that \eqref{eq:special_corner_solution} also holds; and $\bar{u}$ is the unique viscosity solution that is bounded from below. This means $\bar{u}$ is effectively a special solution of \eqref{eq:eikonal}.  Further, since $\bar{u}$ is given explicitly by $\bar{u}(x) = 2|x_{1} x_{2}|^{\frac{1}{2}}$, this function gives rise to a self-similar solution of the corresponding geometric flow \eqref{eq:normal-velocity}.

More generally, if $\varphi$ is any norm in $\mathbb{R}^{2}$, $p \in \mathcal{N}^{*}$, and $f$ is some bounded, positive continuous function in $Q_{p}$, then $Q_{p}$-nondominated sorting as defined above can be interpreted as Pareto hull peeling in the cone $Q_{p}$ and its continuum limit $\bar{u}$ solves the equation
	\begin{equation}
		\left\{\begin{array}{r l}
			\frac{\langle D\bar{u}, v_{p} \rangle \langle D\bar{u},w_{p} \rangle}{|v_{p} \times w_{p}|} = \bar{H}_{\varphi}(D\bar{u}) = f(x) & \text{in} \, \, Q_{p},\\
			\bar{u} = 0 & \text{on} \, \, \partial Q_{p}.
		\end{array} \right.
	\end{equation}


\section{Viscosity solutions and basic estimates} \label{sec:viscosity-solutions}

In this section, we set up the proof of our main result, Theorem \ref{theorem:pareto-convergence}.  We begin by stating the intermediate results that will be used in the proof and showing how they imply the theorem. The remainder of the section establishes $L^{\infty}$ and boundary H\"{o}lder estimates for the rescaled height functions.

The key compatibility assumption is defined in Section \ref{subsec:domain-consistency}.

\subsection{Preliminaries}  Before delving into the proof of Theorem \ref{theorem:pareto-convergence}, we start with some preliminary results that are fundamental in what follows.  

The next result is the basic link between Pareto hull peeling and $Q$-nondominated sorting as defined in Section \ref{sec:q-nds}.  

	\begin{prop} \label{prop:compare_to_nds} Given a finite set $A \subseteq \mathbb{R}^{2}$ and a $p \in \mathcal{N}^{*}$, if $u_{A}$ is the height function of Pareto hull peeling (see \eqref{eq:height}) and $s_{A}^{Q_{p}}$ is the depth function of $Q_{p}$-nondominated sorting, then
		\begin{equation*}
			u_{A} \leq s_{A}^{Q_{p}} \quad \text{in} \, \, \mathbb{R}^{2}.
		\end{equation*}
	\end{prop}  
	
		\begin{proof}  Note that it suffices to prove $\{s_{A} \leq k\} \subseteq \{u_{A} \leq k\}$ for each $k \in \mathbb{N} \cup \{0\}$.  We argue by induction.  To start with, if $x \in \{s_{A} \leq 0\} = \{s_{A} = 0\}$, then \eqref{eq:nds-dpp} implies that
			\begin{equation*}
				0 = s_{A}(x) = \sup \left\{ s_{A}(y) + 1_{A}(y) \, \mid \, y \in x - \inte(Q_{p}) \right\}.
			\end{equation*}
		It follows that $1_{A} = 0$ in $x - \inte(Q_{p})$.  Hence \eqref{eq:negation-thing} implies that $A \cap (x + \inte(Q_{-p})) = A \cap (x - \inte(Q_{p})) = \emptyset$.  Therefore, by Theorem \ref{thm:cone_efficient} and the definition of $u_{A}$, we must have $u_{A}(x) = 0$.  
		
		Next, suppose that $\{s_{A} \leq k\} \subseteq \{u_{A} \leq k\}$ for some given $k \in \mathbb{N} \cup \{0\}$.  We claim that the inclusion $\{s_{A} \leq k + 1\} \subseteq \{u_{A} \leq k + 1\}$ also holds.  Suppose that $x \in \{s_{A} \leq k + 1\}$.  We need to show that $u_{A}(x) \leq k + 1$.  If $s_{A}(x) \leq k$, then $u_{A}(x) \leq k$ by hypothesis; thus, let us assume $s_{A}(x) = k + 1$.  If $u_{A}(x) = 0$, then there is nothing to prove.  Otherwise, by the definition of $u_{A}$ and Theorem \ref{thm:cone_efficient},  there is a $y \in A \cap (x + \inte(Q_{-p}))$ such that
			\begin{equation*}
				u_{A}(x) = u_{A}(y) + 1.
			\end{equation*}
		By \eqref{eq:nds-dpp}, $s_{A}(x) \geq s_{A}(y) + 1$.  Thus, $s_{A}(y) \leq k$ and the inductive hypothesis implies that $u_{A}(y) \leq k$.  We conclude that $u_{A}(x) \leq k + 1$ by the choice of $y$.\end{proof}  

The next result explains the Pareto efficiency assumption on the domain $\mathrm{U}$.  Put simply, when $\mathrm{U}$ is Pareto efficient, the height function $u_{n}$ vanishes outside of $\mathrm{U}$, and hence it is reasonable to expect that the limit will be described by a PDE like \eqref{eq:eikonal}.

	\begin{prop} \label{prop:basic_boundary_condition}  Suppose that $\mathrm{U}$ is a bounded, open, Pareto efficient set in $\mathbb{R}^{2}$.  Given any finite $A \subseteq \bar{\mathrm{U}}$, the height function $u_{A}$ satisfies
		\begin{equation*}
			u_{A} = 0 \quad \text{in} \, \, \mathbb{R}^{2} \setminus \mathrm{U}.
		\end{equation*} \end{prop}
		
			\begin{proof}  Since $A \subseteq \bar{\mathrm{U}}$, the monotonicity of the Pareto hull operation implies that $\inte(E_{1}(A)) = \inte(\mathcal{P}(A)) \subseteq \inte(\mathcal{P}(\bar{U}))$.  Thus, by the Pareto efficiency of $\mathrm{U}$, $E_{1}(A) \subseteq \inte(\mathcal{P}(\bar{\mathrm{U}})) = \mathrm{U}$.  We conclude upon observing that $E_{1}(A) = \{u_{A} \geq 1\} = \{u_{A} > 0\}$.  \end{proof}    
			
The previous proposition shows that Pareto efficiency provides a natural assumption under which the Pareto peeling process of a set $A \subseteq \mathrm{U}$ is confined to $\mathrm{U}$.  

\begin{figure}
	\includegraphics[width=0.25\textwidth]{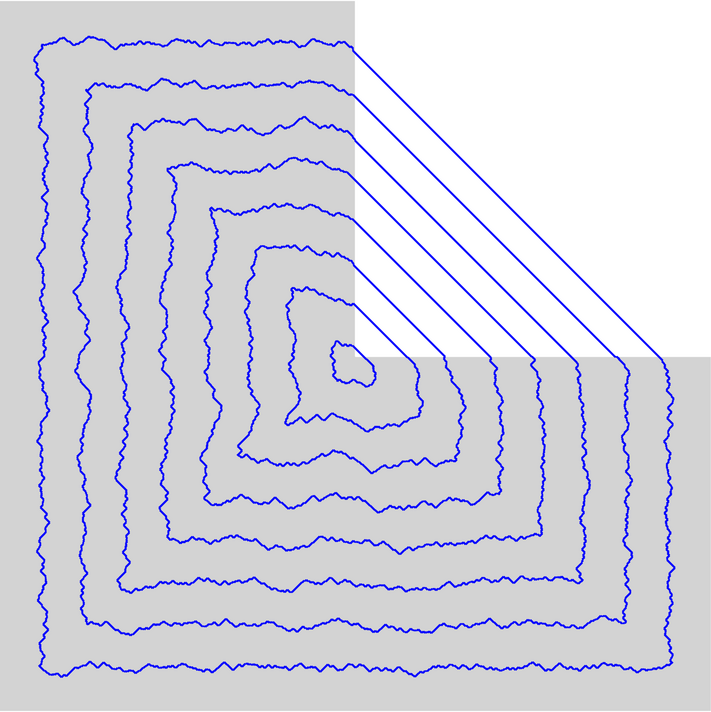} \qquad
	\includegraphics[width=0.25\textwidth]{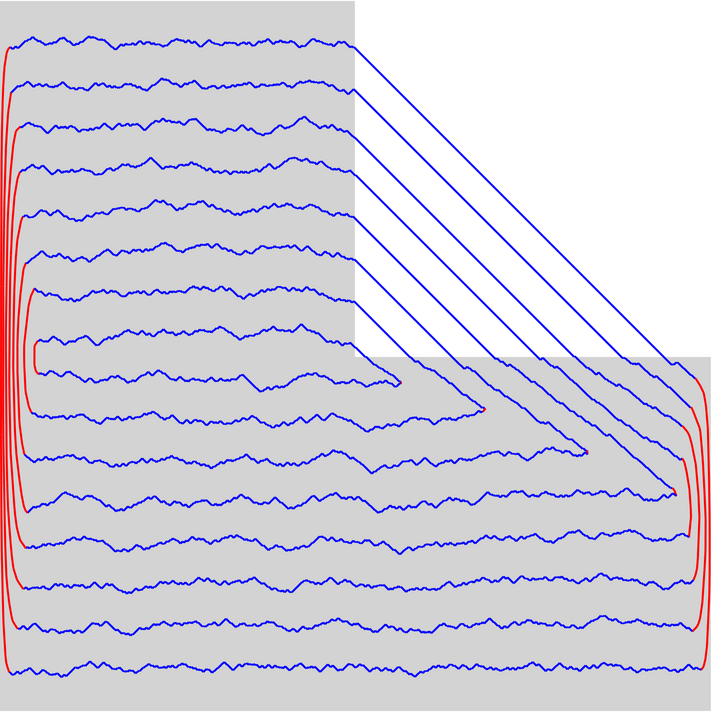} \qquad
	\includegraphics[width=0.25\textwidth]{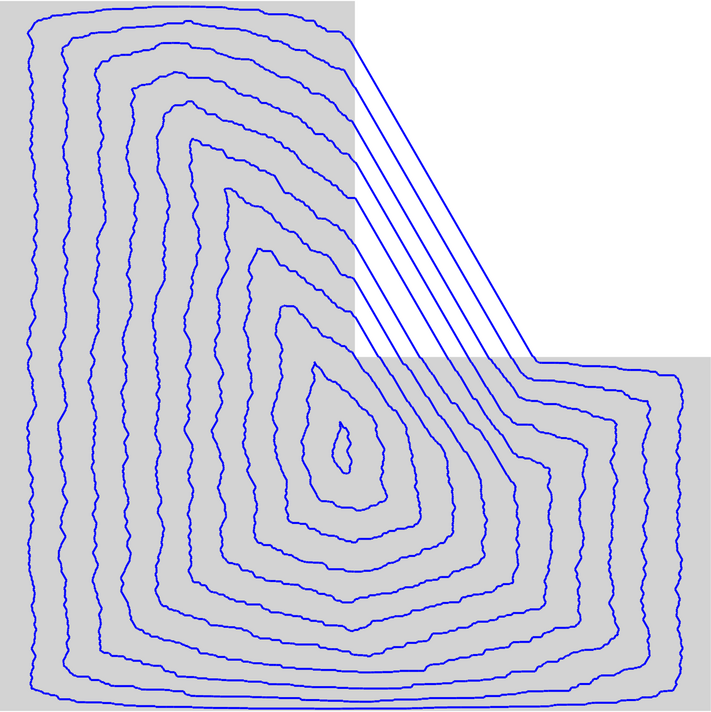}
	
	\caption{Pareto peeling when the domain $\mathrm{U}$ (shaded) is not Pareto efficient.  The pictures correspond to peeling with respect to the norms in Figure \ref{fig:norm-balls}.  Notice that Proposition \ref{prop:basic_boundary_condition} fails to hold: some peels exit the domain.} \label{fig:not-efficient}
\end{figure}

\begin{remark} \label{remark:non-efficient} When $\mathrm{U}$ is \emph{not} Pareto efficient, simulations show that the peels may exit $\mathrm{U}$; see Figure \ref{fig:not-efficient}.  Hence, in this case, we expect it will be more fruitful to treat the data $X_{nf}$ as a Poisson process in $\inte(\mathcal{P}(\bar{U}))$ with intensity $n \tilde{f}$, where $\tilde{f} = f$ in $U$ and $\tilde{f} = 0$, elsewhere.  By analogy with Theorem \ref{theorem:pareto-convergence}, we expect that the rescaled height functions converge in this setting to a solution $\bar{u}$ of the PDE
	\begin{equation*}
		\left\{ \begin{array}{r l}
			\bar{H}_{\varphi}(D\bar{u}) = f & \text{in} \, \, \mathrm{U}, \\
			\bar{H}_{\varphi}(D\bar{u}) = 0 & \text{in} \, \, \inte(\mathcal{P}(\bar{U})) \setminus \mathrm{U}, \\
			\bar{u} = 0 & \text{on} \, \, \partial \mathcal{P}(\bar{U}).
		\end{array} \right.
	\end{equation*}
This does not fit into the framework of Theorem \ref{theorem:pareto-convergence} since the new intensity $\tilde{f}$ is not positive everywhere, an assumption we do not yet know how to relax.  Further, while the enlarged domain $\inte(\mathcal{P}(\bar{U}))$ is Pareto efficient, it need not satisfy the compatibility assumption. \end{remark}

In fact, convex domains are always Pareto efficient. 
\begin{lemma} \label{lemma:convexity-and-efficient}
	If $\mathrm{U}$ is an open, bounded convex set in $\mathbb{R}^{2}$, then it is Pareto efficient. 
\end{lemma}

\begin{proof}  By the definition of Pareto hull, $\bar{\mathrm{U}} \subseteq \mathcal{P}(\bar{\mathrm{U}})$ always holds.  Since $\bar{\mathrm{U}}$ is convex and compact, we can apply \cite[Corollary 4.1]{durier1986sets} to deduce that $\mathcal{P}(\bar{\mathrm{U}}) \subseteq \bar{\mathrm{U}}$.  This proves that $\bar{\mathrm{U}} = \mathcal{P}(\bar{\mathrm{U}})$.  Finally, since $\mathrm{U}$ is open and convex, $\inte(\bar{\mathrm{U}}) = \mathrm{U}$, hence $\mathrm{U} = \inte(\mathcal{P}(\bar{\mathrm{U}}))$ follows. \end{proof}

If $\{\varphi \leq 1\}$ is not strictly convex, then there are many non-convex Pareto efficient sets.  See Figure \ref{fig:non-convex-example} for an illustration of such a set.

\subsection{Proof of Theorem \ref{theorem:pareto-convergence}}  We follow what is now a classical approach in the viscosity solutions literature.  Recall that our interest is in the limit of the rescaled height functions $(\bar{u}_{n})_{n \in \mathbb{N}}$ defined through \eqref{eq:height} with $\bar{u}_{n} := n^{-\frac{1}{2}} u_{n}$ and $u_{n} := u_{X_{nf}}$ for some Poisson process $X_{nf}$ in $\mathrm{U}$ of intensity $nf$.  We begin the proof by defining so-called upper and lower half-relaxed limits $u^{*}$ and $u_{*}$ in $\bar{\mathrm{U}}$ by 
\begin{align}
u^{*}(x) &= \lim_{\delta \to 0^{+}} \sup \left\{ \bar{u}_{n}(y) \, \mid \, \|y - x\| + n^{-1} \leq \delta \right\},  \label{eq:upper-half-relaxed} \\
u_{*}(x) &= \lim_{\delta \to 0^{+}} \inf \left\{ \bar{u}_n(y) \, \mid \, \|y - x\| + n^{-1} \leq \delta \right\}. \label{eq:lower-half-relaxed}
\end{align}

To prove our main result, we will argue that $u^{*}$ and $u_{*}$ are, respectively, viscosity sub- and supersolutions of the Hamilton-Jacobi equation \eqref{eq:eikonal} and apply the comparison principle to conclude that $u^{*} = u_{*}$.

First, let us recall the relevant definitions from the theory of viscosity solutions.  In the next definition, $\USC(\mathcal{O})$ (resp.\ $\LSC(\mathcal{O})$) denotes the set of functions that are 
upper (resp.\ lower) semicontinuous at all points in $\mathcal{O}$. 

\begin{definition} \label{def:interior} (i) We say that $w \in\USC(\mathrm{U})$ is a \emph{viscosity subsolution} of the equation $\bar{H}_{\varphi}(D\bar{u}) = f$ in $\mathrm{U}$ if for each $x_{0} \in \mathrm{U}$ and each smooth function $\psi$ defined in a neighborhood of $x_{0}$, the following statement holds: if there is an $r > 0$ such that $\psi \geq w$ in $B(x_{0},r)$ and $\psi(x_{0}) = w(x_{0})$, then
	\begin{equation*}
		\bar{H}_{\varphi}(D\psi(x_{0})) \leq f(x_{0}).
	\end{equation*}
We abbreviate this by writing $\bar{H}_{\varphi}(Dw) \leq f$ in $\mathrm{U}$.
	
(ii) We say that $v \in\LSC(\mathrm{U})$ is a \emph{viscosity supersolution} of the equation $\bar{H}_{\varphi}(D\bar{u}) = f$ in $\mathrm{U}$ if for each $x_{0} \in \mathrm{U}$ and each smooth function $\psi$ defined in a neighborhood of $x_{0}$, the following statement holds: if there is an $r > 0$ such that $\psi \leq v$ in $B(x_{0},r)$ and $\psi(x_{0}) = v(x_{0})$, then
	\begin{equation*}
		\bar{H}_{\varphi}(D\psi(x_{0})) \geq f(x_{0}).
	\end{equation*}
We abbreviate this by writing $\bar{H}_{\varphi}(Dv) \geq f$ in $\mathrm{U}$.

(iii) A function $u \in C(\mathrm{U})$ is a \emph{viscosity solution} of $\bar{H}_{\varphi}(D\bar{u}) = f$ in $\mathrm{U}$ if it is both a viscosity sub- and supersolution. \end{definition}

As is customary in the viscosity solutions literature, we will abbreviate the condition in (i) and (ii) by saying ``$\psi$ touches $w$ from above at $x_{0}$" and ``$\psi$ touches $v$ from below at $x_{0}$," respectively.  

To prove that $u^{*}$ and $u_{*}$ are viscosity sub- and supersolutions, we start with an $L^{\infty}$ estimate.
\begin{lemma} \label{lemma:l-inf-estimate}
	If $\varphi$ satisfies \eqref{eq:linear-assumption} (that is, $\mathcal{N}^{*}$ is nonempty) and $\mathrm{U}$ is bounded, then there is a constant $0 < C := C_{\varphi, f, \mathrm{U}} < \infty$ so that on an event of probability 1
	\[
	\limsup_{n \to \infty} \sup_{\R^2}  \bar{u}_n \leq C.
	\]
\end{lemma}
This will be proved below by comparing to nondominated sorting.

Next, we show that the boundary behavior is controlled provided $\mathrm{U}$ is compatible with $\varphi$.

\begin{lemma} \label{lemma:boundary-continuity}
	If $\mathrm{U}$ is a bounded, open, and Pareto efficient set compatible with $\varphi$ (see Definition \ref{def:compatible} below), then, on an event of probability 1, for each $x_{0} \in \partial \mathrm{U}$, there is a constant $A := A_{\varphi,\mathrm{U},x_{0}} > 0$ such that, for all $n$ sufficiently large and all $\delta \in (0,1)$,
	\begin{equation} \label{eq:holder_estimate}
	\sup_{B(x_0, \delta)} u_{n}\leq A\left(\sup_{\mathrm{U}} f \right)^{\frac{1}{2}} \sqrt{\delta n}.
	\end{equation}
In particular, $u_{*} = u^{*} = 0$ on $\partial \mathrm{U}$ almost surely.
\end{lemma}

The lemma enables us to prove that the boundary condition in \eqref{eq:eikonal} is satisfied.  Its proof is the only place in the paper where the compatibility assumption is used.

Appealing again to nondominated sorting, we show that $u^{*}$ is a subsolution.

	\begin{prop} \label{prop:subsolution} On an event of probability 1, the upper half-relaxed limit $u^{*}$ defined by \eqref{eq:upper-half-relaxed} satisfies $\bar{H}_{\varphi}(Du^*) \leq f$ in $\mathrm{U}$.
	\end{prop}  
	
Finally, in the core of the paper, we show that $u_{*}$ is a supersolution.

\begin{prop}\label{prop:supersolution}
		If $\varphi$ satisfies \eqref{eq:linear-assumption} then, on an event of probability 1, the lower half-relaxed limit $u_{*}$ defined by \eqref{eq:upper-half-relaxed} satisfies $\bar{H}_{\varphi}(Du_{*}) \geq f$ in $\mathrm{U}$.
	
	
\end{prop}

The remainder of the paper is devoted to the proof of these results.  For the sake of completeness, we show how they imply Theorem \ref{theorem:pareto-convergence}.  The main step involves invoking the comparison principle.  Let us recall that result and briefly sketch its proof.

\begin{lemma} \label{lemma:comparison-principle}
Suppose $f > 0$ in $\mathrm{U}$, an open, bounded set in $\R^2$.
If $w \in \USC(\bar{\mathrm{U}})$ satisfies $\bar{H}_{\varphi}(Dw) \leq f$ in $\mathrm{U}$; $v \in \LSC(\bar{\mathrm{U}})$ satisfies
$\bar{H}_{\varphi}(Dv) \geq f$ in $\mathrm{U}$; and $v \geq w$ on $\partial \mathrm{U}$, then 
$v \geq w$ in $\mathrm{U}$.
\end{lemma}

\begin{proof}
For each $\epsilon \in (0,1)$, if we define $w^{\epsilon} = (1-\epsilon) w$, then 
$\bar{H}_{\varphi}(Dw^{\epsilon}) = (1-\epsilon)^2 \bar{H}_{\varphi}(D w) \leq (1-\epsilon)^2 f$.  Hence $w^{\epsilon}$ is a strict subsolution
and, thus, by strict comparison \cite{crandall1992user}, $\sup_{\bar{\mathrm{U}}} (w^{\epsilon} - v) \leq \sup_{\partial U} (w^{\epsilon} - v)$.  We recover the result upon sending $\epsilon \to 0^{+}$.   
\end{proof}

Combining the elements above, we conclude that $\bar{u}_{n} \to \bar{u}$, where $\bar{u}$ is the unique solution of \eqref{eq:eikonal}.  Before showing this, for the sake of precision, let us first define what we mean by a viscosity solution of the boundary value problem \eqref{eq:eikonal}:

\begin{definition}  We say that a continuous function $u \in C(\bar{\mathrm{U}})$ is a viscosity solution of \eqref{eq:eikonal} if 
	\begin{itemize}
		\item[(i)] $u$ is a viscosity solution of $\bar{H}_{\varphi}(Du) = f$ in $\mathrm{U}$, and
		\item[(ii)] $u = 0$ on $\partial \mathrm{U}$.
	\end{itemize}
\end{definition}  

Notice that Lemma \ref{lemma:comparison-principle} already implies that if \eqref{eq:eikonal} has a viscosity solution, then it is unique.

\begin{proof}[Proof of Theorem \ref{theorem:pareto-convergence}]  A direct argument shows that $u^{*}$ is upper semi-continuous and $u_{*}$, lower semi-continuous (see, \eg, \cite[Chapter 5]{bardi_capuzzo-dolcetta}).  Lemma \ref{lemma:l-inf-estimate} shows they are, in fact, bounded, hence not identically infinity.  

By Lemma \ref{lemma:boundary-continuity}, $u_{*} = u^{*} = 0$ on $\partial \mathrm{U}$.  Together with Propositions \ref{prop:subsolution} and \ref{prop:supersolution}, this implies $u^{*}$ and $u_{*}$ satisfy the hypotheses of the comparison principle.  In particular, $u^{*} \leq u_{*}$ in $\mathrm{U}$.  At the same time, the definitions \eqref{eq:upper-half-relaxed} and \eqref{eq:lower-half-relaxed} directly give $u^{*} \geq u_{*}$.  Therefore, $u^{*} = u_{*}$ in $\bar{\mathrm{U}}$ and this function, call it $\bar{u}$, is the unique viscosity solution of \eqref{eq:eikonal}.  

Finally, observe that, in general, the identity $u^{*} = u_{*}$ holds if and only if the sequence $(\bar{u}_{n})_{n \in \N}$ converges uniformly in $\bar{U}$ as $n \to \infty$.  Hence $\bar{u}_{n} \to \bar{u}$.  \end{proof}

\subsection{Global upper bound}
As an immediate application of the comparison with nondominated sorting, we prove an $L^{\infty}$ estimate.

\begin{proof}[Proof of Lemma \ref{lemma:l-inf-estimate}]
By Proposition \ref{prop:compare_to_nds}, \eqref{eq:q-longest-chain}, and our assumption that $\mathcal{N}^*$ is nonempty, 
there is a $Q_p$ so that for all $x \in \R^2$, 
\[
u_n(x) \leq s^{Q_{p}}_{X_{nf}}(x) \leq \ell^{Q_p}(X_{n f} \cap \mathrm{U}).
\]
At the same time, we can fix a rhombus $B$ with sides parallel to $v_{p}$ and $w_{p}$, that is, a set of the form $B = L_{Q_{p}}^{-1}(B')$ for some coordinate rectangle $B'$, such that $B \supseteq \mathrm{U}$.  This gives $\bar{u}_n(x) \leq n^{-1/2} \ell^{Q_{p}}(X_{nf} \cap B)$.  We conclude by Proposition \ref{prop:longest-chain-convergence} and Lemma \ref{lemma:nds-change-of-variables}.
\end{proof}

\subsection{Compatibility and boundary H\"{o}lder estimate} \label{subsec:domain-consistency}
In this section, we prove a boundary H\"{o}lder estimate on open Pareto efficient domains that are compatible with the given norm. 
Recall that a bounded open set $\mathrm{U}$ is Pareto efficient if $\inte(\mathcal{P}(\bar{\mathrm{U}})) = \mathrm{U}$.  

%
%

Let $\mathrm{U}$ be an open, bounded, and Pareto efficient set. We say that a convex cone $Q$ {\it supports} 
$\mathrm{U}$ at $x_0$ if $x_0 \in (x_0 + Q) \cap \bar{\mathrm{U}} $ and  $(x_0 + Q) \cap \bar{\mathrm{U}} \subseteq \partial \mathrm{U}$. The set of all such convex cones supporting $\mathrm{U}$ at $x_0 \in \partial \mathrm{U}$ will be denoted by $\mathcal{C}(\mathrm{U},x_0)$. 
Note that since $\mathrm{U}$ is Pareto efficient and open this set is non-empty for each such $x_0$. 

\begin{definition} \label{def:compatible} We say a Pareto efficient set $\mathrm{U} \subseteq \R^2$ is {\it compatible} with $\varphi$ if for each $x_0 \in \partial U$, there is a convex cone $\tilde{Q} \in \mathcal{C}(\mathrm{U}, x_0)$ and a $p \in \mathcal{N}^{*}$ such that $Q_p \subseteq \tilde{Q}$. \end{definition}

Equivalently, $\mathrm{U}$ is compatible with $\varphi$ if and only if $\mathcal{C}(\mathrm{U},x_{0}) \cap \{\mathcal{Q}_{p} \, \mid \, p \in \mathcal{N}^{*}\} \neq \emptyset$ for each $x_{0} \in \partial \mathrm{U}$. 

Shortly we will show that the limiting height function equals zero on the boundary provided the compatibility condition holds.  In the next subsection, we will show that this need not be the case when compatibility fails.  

That still leaves the question under what conditions a Pareto efficient domain $\mathrm{U}$ will be compatible with a given norm $\varphi$.  In the next result, we show that the compatibility condition is superfluous whenever $\varphi$ is polyhedral.

\begin{prop} \label{prop:polyhedral-compatibility}If $\varphi$ is a polyhedral norm in $\mathbb{R}^{2}$ and $\mathrm{U} \subseteq \mathbb{R}^{2}$ is a bounded, open, Pareto efficient set, then $\mathrm{U}$ is compatible with $\varphi$.  \end{prop}  

%
%

\begin{proof}  Suppose that $x_{0} \in \partial \mathrm{U}$.  We need to show that $\mathcal{C}(\mathrm{U},x_{0}) \cap \{Q_{p} \, \mid \, p \in \mathcal{N}^{*}\} \neq \emptyset$.  Since $\mathrm{U}$ is Pareto efficient, we know that $x_{0} \notin \inte(\mathcal{P}(\bar{\mathrm{U}}))$.  Accordingly, Theorem \ref{thm:cone_efficient} implies that there is a $p \in \mathcal{N}^{*}$ such that $\bar{\mathrm{U}} \cap (x_{0} + \inte(Q_{p})) = \emptyset$.  This implies that $(x_{0} + Q_{p}) \cap \bar{\mathrm{U}} \subseteq \partial \bar{\mathrm{U}}$, hence $Q_{p} \in \mathcal{C}(\mathrm{U},x_{0})$.   \end{proof}

Finally, we show that compatible domains satisfy the H\"{o}lder estimate \eqref{eq:holder_estimate}.  The proof below is based on the approach in \cite[Lemma 1]{calder2016direct}.

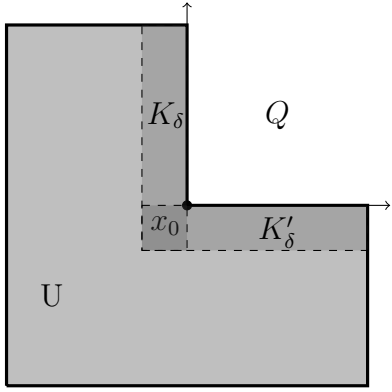
\begin{figure}
	\begin{tikzpicture}[scale=1.2]

\filldraw[fill = lightgray, very thick] (-2,-2) --(-2,2) -- (0,2) -- (0,0) -- (2,0) -- (2,-2) -- (-2,-2);
\draw (-1.5,-1) node {$\mathrm{U}$};

\draw[anchor=north east] (0.05,0) node {$x_0$};
\filldraw[fill=black] (0,0) circle (0.05);

\filldraw[style = dashed, fill = darkgray, fill opacity = 0.2, draw = black] (-0.5,-0.5) -- (0,-0.5) -- (0,2) -- (-0.5,2) -- (-0.5,-0.5);
\draw[anchor=east] (0.1,1) node {$K_{\delta}$};


\filldraw[style = dashed, fill = darkgray, fill opacity = 0.2, draw = black] (-0.5,0) -- (2,0) -- (2,-0.5) -- (-0.5,-0.5) -- (-0.5,0);
\draw[anchor=north] (1,0) node {$K_{\delta}'$};

\path[draw, ->] (0,0) -- (0,2.25);
\path[draw, ->] (0,0) -- (2.25,0);
\draw (1,1) node {$Q$};

\end{tikzpicture}
	\caption{Construction in the proof of Lemma \ref{lemma:boundary-continuity}. 
	} \label{fig:supporting-cone-cone}
\end{figure}

\begin{proof}[Proof of Lemma \ref{lemma:boundary-continuity}]
Since $\mathrm{U}$ is Pareto efficient, Proposition \ref{prop:basic_boundary_condition} implies that
\begin{equation}
u_n(x) = 0 \mbox{ for all $x \not \in \mathrm{U}$}.
\end{equation}
Let $x_0 \in \partial \mathrm{U}$ be given. Since $\mathrm{U}$ is compatible with $\varphi$, we can fix a $\tilde{Q} \in \mathcal{C}(\mathrm{U}, x_{0})$ and a $p \in \mathcal{N}^{*}$ such that $Q_{p} \subseteq \tilde{Q}$.  Notice that the inclusion $Q_{p} \subseteq \tilde{Q}$ implies that $Q_{p} \in \mathcal{C}(\mathrm{U},x_{0})$.  Therefore, without loss of generality, we can set $Q := Q_p = \tilde{Q}$.

\textbf{Step 1: Case $Q = [0,\infty)^{2}$}  

To simplify the argument, we first treat the case when $Q = Q_{p} = [0,\infty)^{2}$.  Up to an irrelevant translation, there is no loss of generality assuming that $x_{0} = 0$, which simplifies the arguments that follow.  

Define $D_{1},D_{2} \geq 0$ by 
	\begin{align*}
		D_{1} &= \inf \left\{ s > 0 \, \mid \, [s,\infty) \times [-\delta,0] \subseteq \mathbb{R}^{2} \setminus \mathrm{U} \right\}, \\
		D_{2} &= \inf \left\{ s > 0 \, \mid \, [-\delta,0] \times [s,\infty) \subseteq \mathbb{R}^{2} \setminus \mathrm{U} \right\},
	\end{align*} 
and let $D = \max(D_{1},D_{2})$.  
Consider the two rectangles, $K_{\delta}, K_{\delta}' \subseteq \mathbb{R}^{2}$ given by
	\begin{align*}
		K_{\delta} &= \{ x \in \mathbb{R}^{2} \, \mid \, - \delta \leq x_{1} \leq 0, \, \, -\delta \leq x_{2} \leq D\}, \\
		K_{\delta}' &= \{x \in \mathbb{R}^{2} \, \mid \, -\delta \leq x_{1} \leq D, \, \, -\delta \leq x_{2} \leq 0\}. 
	\end{align*}
Notice that $B(0,\delta) \subseteq (-\delta,-\delta) + \inte(Q)$ and, in particular, from the inclusion $Q \in \mathcal{C}(0,\mathrm{U})$,
	\begin{equation} \label{eq:key_inclusion_boundary_estimate}
		B(0,\delta) \cap \mathrm{U} \subseteq ((-\delta,-\delta) + \inte(Q)) \cap \mathrm{U} \subseteq K_{\delta} \cup K_{\delta}'.
	\end{equation}
See Figure \ref{fig:supporting-cone-cone}.

Write $P_{\delta} := K_{\delta} \cup K'_{\delta}$. 
We claim that 
\begin{equation}  \label{eq:sorting_boundary_bound}
\sup_{x \in P_{\delta}} u_n(x) \leq \ell^{-Q}(X_{nf} \cap P_{\delta}),
\end{equation}
where, for an arbitrary $A \subseteq \mathbb{R}^{2}$, $\ell^{-Q}(X_{nf} \cap A)$ denotes the $-Q$-longest chain in $X_{n f} \cap P_{\delta}$ as in Section \ref{sec:q-nds}.  Indeed, by Proposition \ref{prop:compare_to_nds} and \eqref{eq:q-longest-chain}, 
	\begin{equation*}
		\sup_{x \in P_{\delta}} u_{n}(x) \leq \sup_{x \in P_{\delta}} s^{-Q}_{X_{nf}}(x) = \sup_{x \in P_{\delta}} \ell^{-Q}((x + \inte(Q)) \cap X_{nf}).
	\end{equation*}
Notice that, by \eqref{eq:key_inclusion_boundary_estimate}, if $y \in (x + \inte(Q)) \cap \mathrm{U}$, then $y \in P_{\delta}$.  Thus, for each $x \in P_{\delta}$, we have
	\begin{equation*}
		\ell^{-Q}((x + \inte(Q)) \cap X_{nf}) = \ell^{-Q}((x + \inte(Q)) \cap X_{nf} \cap P_{\delta})
	\end{equation*} 
and \eqref{eq:sorting_boundary_bound} follows.

Finally, observe that if $\{x_{1},\dots,x_{m}\} \subseteq P_{\delta}$ is an increasing $-Q$-chain, that is, $x_{1} \leq_{-Q} \dots \leq_{-Q} x_{m}$, then either $\{x_{1},\dots,x_{m}\} \subseteq K_{\delta}$ or $\{x_{1},\dots,x_{m}\} \subseteq K_{\delta}'$ depending on whether $x_{1} \in K_{\delta}$ or $x_{1} \in K_{\delta}'$.  It follows that
	\begin{equation*}
		\sup_{x \in P_{\delta}} u_n(x) \leq \ell^{-Q}(X_{nf} \cap P_{\delta}) \leq \max\{\ell^{-Q}(X_{nf} \cap K_{\delta}), \ell^{-Q}(X_{nf} \cap K_{\delta}')\}.
	\end{equation*}
Therefore, by the asymptotics of the longest chain (Proposition \ref{prop:longest-chain-convergence}), with probability 1,
	\begin{equation} \label{eq:rectilinear_boundary_estimate}
		\limsup_{n \to \infty} \sup_{x \in B(0,\delta) \cap \mathrm{U}} \bar{u}_{n}(x) \leq \limsup_{n \to \infty} \sup_{x \in P_{\delta}} \bar{u}_{n}(x) \leq 2 \left( \sup_{P_{\delta}} f \right)^{\frac{1}{2}} \sqrt{(D + \delta) \delta}.
	\end{equation}
	
\textbf{Step 2: General case}  

Finally, we treat the general case.  Recall that $Q = Q_{p}$ for some $p \in \mathcal{N}^{*}$ and that $Q$ is a proper cone.  Let $\mathcal{B}(p)$ be the collection of all rhombuses in $\mathbb{R}^{2}$ with sides parallel to $v_{p}$ and $w_{p}$.  More precisely, in terms of the linear transformation $L_{Q}$ from Section \ref{sec:isomorphism},
	\begin{equation*}
		\mathcal{B}(p) = \{L_{Q}^{-1}(B') \, \mid \, B' = [a,b] \times [c,d] \, \, \text{for some} \, \, a < b \mbox{ and } c < d \in \mathbb{R}\}.
	\end{equation*}
Combining Proposition \ref{prop:longest-chain-convergence} and Lemma \ref{lemma:nds-change-of-variables}, we deduce that there is an event $\Omega(p)$ of probability 1 such that, on $\Omega(p)$, for any $B \in \mathcal{B}(p)$, we have
	\begin{equation*}
		\limsup_{n \to \infty} n^{-1/2} \ell^{-Q}(X_{nf} \cap B) \leq 2 |v_{p} \times w_{p}|^{\frac{1}{2}} \left(\sup_{B} f\right)^{\frac{1}{2}} |B|^{\frac{1}{2}}.
	\end{equation*}  
Note that the event $\Omega(p)$ does not depend on choice of boundary point $x_{0}$, hence in what follows there will be no risk of generating non-measurable sets through uncountable intersections.
	
Let $\mathrm{U}'$ be the transformed domain $U' = L_{Q}(-x_{0} + \mathrm{U})$, $X' = L_{Q}(-x_{0} + X_{nf})$ be the transformed process, and $f_{Q}(x') = f(L_{Q}^{-1}(x') + x_{0})$.  By Lemma \ref{lemma:affine-invariance}, the function $v_{n}(x') = u_{n}(L_{Q}^{-1}(x') + x_{0})$ is the height function associated with $X'$ in $\mathrm{U}'$.  Note that $X'$ is a Poisson process of intensity $n |v_{p} \times w_{p}| f_{Q}$ in $\mathrm{U}'$ and $0 \in \partial \mathrm{U}'$.  Thus, after changing variables in \eqref{eq:sorting_boundary_bound}, we obtain, on the event $\Omega(p)$, the following estimate:
	\begin{equation*}
		\limsup_{n \to \infty} \sup_{x \in (x_{0} + L_{Q}^{-1}(B(0,\delta))) \cap \mathrm{U}} \bar{u}_{n}(x) \leq 2 \sqrt{(D + \delta) |v_{p} \times w_{p}|} \left(\sup_{L_{Q}^{-1}(P_{\delta})} f \right)^{\frac{1}{2}} \sqrt{\delta},
	\end{equation*}
where, by construction, $D = \max(D_{1},D_{2})$ with $D_{1}$ and $D_{2}$ given by
	\begin{align*}
		D_{1} &= \inf \left\{s \geq 0 \, \mid \, x_{0} + L_{Q}^{-1}([s,\infty) \times [-\delta,0]) \subseteq \mathrm{U} \right\}, \\
		D_{2} &= \sup \left\{ s \geq 0 \, \mid \, x_{0} + L_{Q}^{-1}([-\delta,0] \times [s,\infty)) \subseteq \mathrm{U} \right\}.
	\end{align*}
At the same time, from the estimate \eqref{eq:operator-norm} on $\|L_{Q}\|$, if we set $\delta' = \delta \sqrt{1 - |\langle v_{p}, w_{p} \rangle|}$, then $B(0,\delta') \subseteq L_{Q}^{-1}(B(0,\delta))$.  Therefore, in $\Omega(p)$, we have, for any $\delta' > 0$,
	\begin{equation*}
		\limsup_{n \to \infty} \sup_{x \in B(x_{0},\delta')} u_{n}(x) \leq 2 \sqrt{\underline{A}[D + (1 - |\langle v_{p}, w_{p} \rangle|)^{-1/2} \delta']} \left(\sup_{\mathrm{U}} f \right)^{\frac{1}{2}} \sqrt{\delta'},
	\end{equation*}
where $\underline{A} > 0$ is determined by
	\begin{equation*}
		\underline{A} = \sup \left\{ \frac{|\sin(\theta)|}{\sqrt{1 - |\cos(\theta)|}} \, \mid \, \theta \in (0,\pi) \right\}.
	\end{equation*}
	
\textbf{Step 3: Boundary condition}  

Finally, we claim that $u_{*} = u^{*} = 0$ on $\partial \mathrm{U}$ almost surely.  Indeed, consider the the event $\cap_{p \in \mathcal{N}^{*}} \Omega(p)$, which has full probability since $\mathcal{N}^{*}$ is countable.  Invoking the definition \eqref{eq:upper-half-relaxed} of $u^{*}$, we find, for any $x_{0} \in \partial \mathrm{U}$,
	\begin{equation*}
		0 \leq u_{*}(x_{0}) \leq u^{*}(x_{0}) \leq \lim_{\delta \to 0^{+}} C \left( \sup_{\mathrm{U}} f \right)^{\frac{1}{2}} \sqrt{\delta} = 0.
	\end{equation*}	
\end{proof}

\subsection{Convex hull peeling}  We now show that when $\{\varphi \leq 1\}$ is strictly convex, or, equivalently, $\mathcal{N}^{*} = \emptyset$, scaling by $n^{-\frac{1}{2}}$ results in a trivial limit.

\begin{proof}[Proof of Corollary \ref{cor:convex-peeling}]  In this case, Proposition \ref{prop:supersolution} implies $\bar{H}_{\varphi}(Du_{*}) \geq f$ in $\mathrm{U}$, where $\bar{H}_{\varphi} \equiv 0$.  Hence, since $f$ is positive, no smooth function can touch $u_{*}$ from below.  We claim this implies $u_{*} \equiv \infty$ in $\mathrm{U}$.  

Indeed, suppose there were an $x_{0} \in \mathrm{U}$ for which $u_{*}(x_{0}) < \infty$.  Given $\delta > 0$, the function $x \mapsto u_{*}(x) + \frac{\|x - x_{0}\|^{2}}{2 \delta}$ achieves its minimum at some point $x_{\delta} \in \bar{\mathrm{U}}$ by lower semi-continuity.    Since $u_{*}(x_{0}) \geq u_{*}(x_{\delta}) + \frac{\|x_{\delta} - x_{0}\|^{2}}{2 \delta}$, we know that $x_{\delta} \to x_{0}$ and $u_{*}(x_{\delta}) < \infty$.  Hence, for small enough $\delta$, $u_{*}$ is touched from below by the smooth function $u_{*}(x_{\delta}) - \frac{\|x - x_{0}\|^{2}}{2 \delta}$ at $x_{\delta} \in \mathrm{U}$, contradicting our previous deduction.

From the identity $u_{*} \equiv \infty$, one readily deduces that $\bar{u}_{n} \to \infty$ locally uniformly in $\mathrm{U}$.  \end{proof}  

\subsection{Counterexample when compatibility fails} \label{sec:counterexample}
\begin{figure}[t]
	\includegraphics[width=0.3\textwidth]{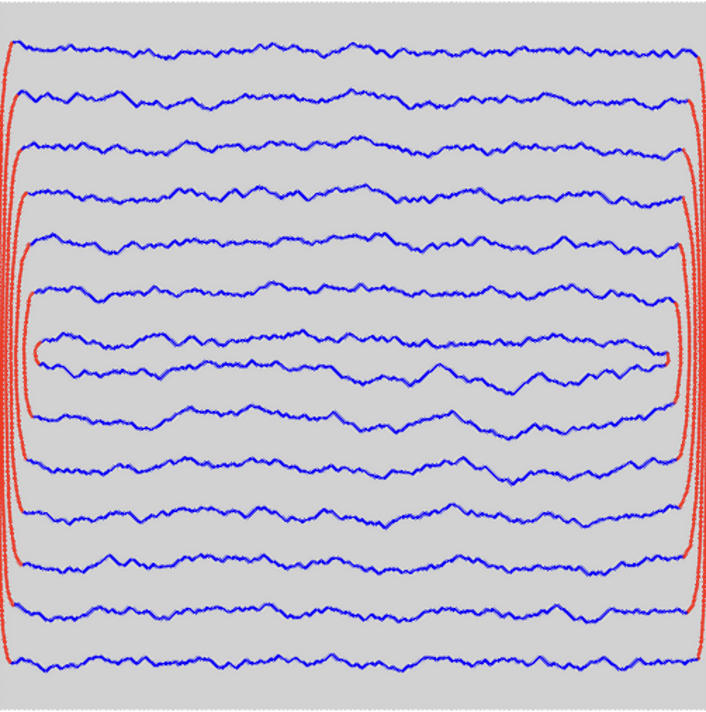} \qquad
	\includegraphics[width=0.3\textwidth]{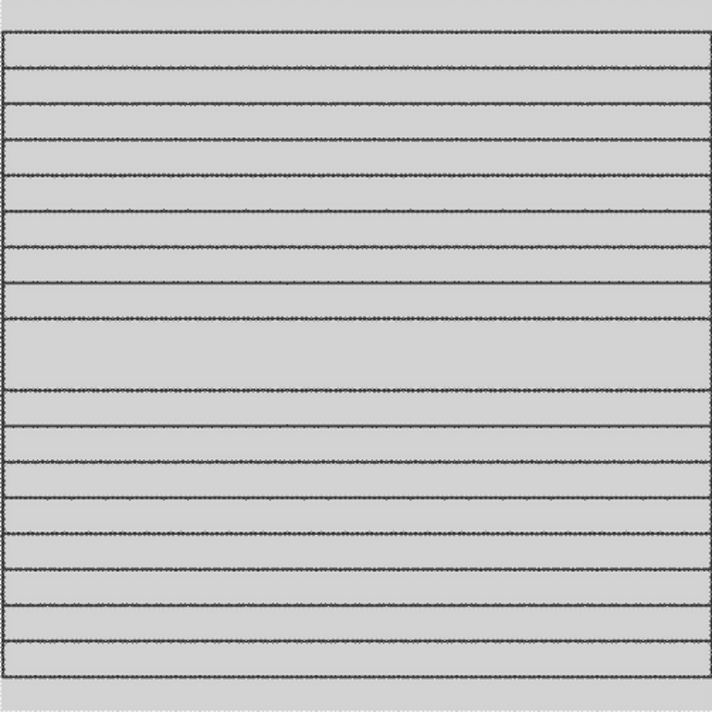} 
	\caption{Pareto peeling of a homogeneous Poisson cloud in a square domain $\mathrm{U} = (-1,1)^2$ with respect to the norm $\varphi(x) = \max(\sqrt{2} |x_1|, \|x\|_2)$. On the left is a simulation and on the right are level sets of the minimal supersolution given by $u(x_1, x_2) =  \sqrt{2}(1 - |x_2|)$.} \label{fig:simulation-counter-example}
\end{figure}

 In the next result, we prove that the Hamilton-Jacobi equation \eqref{eq:eikonal} in a Pareto efficient domain does not always have a classical viscosity solution.  As we will show, this implies that the height functions do not converge uniformly in $\bar{\mathrm{U}}$ in general without the compatibility assumption.  This does not rule out the possibility that they converge locally uniformly in $\mathrm{U}$ to the minimal supersolution; see Figure \ref{fig:simulation-counter-example}. 

\begin{prop} \label{prop:counterexample}  If $\mathrm{U} = (-1,1)^{2}$, $\varphi(x) = \max ( \sqrt{2}|x_{1}|, \|x\|_2)$, and $f \equiv 1$, then $\mathrm{U}$ is not compatible with $\varphi$ and \eqref{eq:eikonal} has no viscosity solution.  \end{prop}  

The unit ball of the norm $\varphi$ above is pictured in Figure \ref{fig:norm-balls}.

\begin{proof}  First of all, $\mathrm{U}$ is not compatible with $\varphi$ since $\mathcal{C}(U,(-1,0)) = H_q$
	for $ q=(-1,0)$. 
	
	With the given norm $\varphi$, we have 
	\begin{equation*}
\bar{H}_{\varphi}(p') = \max \left(0, \frac{-p_{1}'^2 + p_{2}'^2}{2}\right).
	\end{equation*}
Thus if we define $\{\tilde{H}_{N}\}_{N \in \N}$ by 
	\begin{equation*}
		\tilde{H}_{N}(p') = \max\left(\frac{2}{N}|p_{1}'|, 2|p_{2}'|\right),
	\end{equation*}
then 
	\begin{equation} \label{eq:subsolution}
		\bigcup_{N = 1}^{\infty} \{\tilde{H}_{N} \leq 1\} \subseteq \{\bar{H}_{\varphi} < 1\}.
	\end{equation}
See Figure \ref{fig:hyperbola-counterexample} for a depiction of this; the key point is that the rectangle can grow arbitrarily long.  We will use $\{\tilde{H}_{N}\}$ to build subsolutions that are nonzero near $(-1,0)$.

Let us argue by contradiction.  Suppose $u \in C(\bar{\mathrm{U}})$ is a viscosity solution of \eqref{eq:eikonal}.   By \eqref{eq:subsolution}, if we let $(u_{N})_{N \in \N} \subseteq C(\bar{\mathrm{U}})$ denote the solutions of the Eikonal equations
	\begin{equation*}
		\left\{ \begin{array}{r l}
				\tilde{H}_{N}(Du_{N}) = 1 & \text{in} \, \, (-1,1)^{2}, \\
				u_{N} = 0 & \text{on} \, \, \partial [-1,1]^{2},
			\end{array} \right.
	\end{equation*}
then $\bar{H}_{\varphi}(Du_{N}) \leq 1$ in $(-1,1)^{2}$ independent of $N$.  Thus, $\sup_{N} u_{N} \leq u$ by Lemma \ref{lemma:comparison-principle}.

At the same time, notice that if we rescale by setting $v_{N}(y) = u_{N}(N^{-1}y_{1},y_{2})$, then $(v_{N})_{N \in \N}$ are viscosity solutions of the problems
	\begin{equation*}
		\left\{\begin{array}{r l}
				\tilde{H}_{1}(Dv_{N}) = 1 & \text{in} \, \, (-N,N) \times (-1,1), \\
				v_{N} = 0 & \text{on} \, \, \partial [-N,N] \times [-1,1].
			\end{array} \right.
	\end{equation*}
These have explicit representations as the distance functions to the boundary with respect to the dual norm $\tilde{H}_{1}^{*}$, that is,
	\begin{equation*}
		v_{N}(y) = \inf \left\{ \tilde{H}_{1}^{*}(y - z) \, \mid \, z \in \partial [-N,N] \times [-1,1] \right\},
	\end{equation*}
see, \eg, \cite{hjbook} for the proof of this.
A well-known computation shows that $\tilde{H}_{1}^{*}(q') = \frac{1}{2}(|q_{1}'| + |q_{2}'|)$ and, thus,
	\begin{equation*}
		v_{N}(-(N-1),0) = \tilde{H}_{1}^{*}(1,0) = \frac{1}{2}.
	\end{equation*}

Scaling back, we have $u_{N}(-1 + N^{-1},0) = \frac{1}{2}$ and, thus,
	\begin{equation*}
		u(-1 + N^{-1},0) \geq u_{N}(-1 + N^{-1},0) = \frac{1}{2}.
	\end{equation*}
Since $u(-1,0) = 0$, this contradicts the continuity of $u$.  \end{proof}

\begin{figure} 
	\begin{center}
		\begin{tikzpicture}

\pgfmathsetmacro{\e}{1.44022}   
\pgfmathsetmacro{\a}{1}
\pgfmathsetmacro{\b}{(\a*sqrt((\e)^2-1)} 
\draw[color=red, very thick] plot[domain=-2:2] ({\b*sinh(\x)}, {\a*cosh(\x)-0.7});
\draw[color=red, very thick] plot[domain=-2:2] ({\b*sinh(\x)}, {-\a*cosh(\x)+0.7});

\filldraw[fill = blue, opacity = 0.25] (-3.5,-0.125) rectangle (3.5,0.125);

\end{tikzpicture}
	\end{center}
	\caption{The red curves comprise the set $\{\bar{H}_{\varphi} = 1\}$, while the blue rectangle is $\{\tilde{H}_{N} \leq 1\}$.} \label{fig:hyperbola-counterexample}
\end{figure}
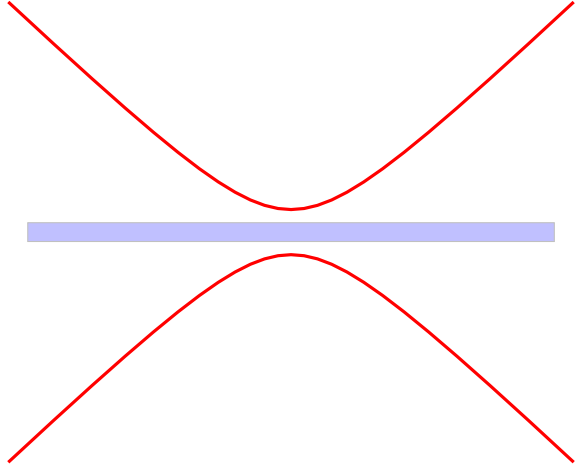

For the sake of completeness, we show how the nonexistence of viscosity solutions implies the failure of uniform convergence of the height functions.

\begin{proof}[Proof of Corollary \ref{cor:compatible-counterexample}]  Let $\mathrm{U}$, $\varphi$, and $f$ be as in the previous proposition.  If the rescaled height functions converged uniformly to some continuous function, then this would be a viscosity solution of \eqref{eq:eikonal} by Proposition \ref{prop:subsolution} and \ref{prop:supersolution}, which was just shown to be impossible.  \end{proof} 

\subsection{Level Set PDE}  We conclude the section by proving that the Pareto peels themselves converge to surfaces moving with normal velocity given by \eqref{eq:normal-velocity}.  

\begin{proof}[Proof of Corollary \ref{cor:level-set}]  To start with, we need to justify the claim that the sets $(E_{t})_{t \geq 0}$ given by \eqref{eq:geometric-flow} are, in fact, the unique generalized level set evolution associated with \eqref{eq:normal-velocity}.  That they are a generalized level set evolution in the sense of \cite{barles_souganidis_fronts} follows from the fact that the function $v$ defined by \eqref{eq:arrival-time} is a viscosity solution of \eqref{eq:level-set}.  By Proposition 2.2 in \cite{barles_souganidis_fronts}, the evolution is unique if and only if there is no fattening.  In other words, we need to prove that $\partial E_{t}$ has empty interior for all $t \geq 0$. 

This part follows readily from the equation solved by $u$.  Recall that $\partial E_{t} = \{u = t\}$.  If $B(x_{0},r) \subseteq \{u = t\}$ for some $x_{0} \in \mathrm{U}$ and $r > 0$, then the constant function $t$ touches $u$ below at $x_{0}$.  This implies $0 = \bar{H}_{\varphi}(0) \geq f(x_{0})$, which is absurd since $f$ is assumed to be positive in $\mathrm{U}$.  Hence $\partial E_{t}$ has empty interior as claimed.

Finally, the convergence $\bar{E}^{(n)}_{\lfloor n^{\frac{1}{2}} t \rfloor} \to \bar{E}_{t}$ in the Hausdorff distance for any fixed $t > 0$ follows immediately from the fact that $\partial E_{t}$ has empty interior and $\bar{u}_{n} \to u$ uniformly in $\bar{\mathrm{U}}$.    
\end{proof}

\section{Subsolution proof} \label{sec:subsolution}
	In this section we show, essentially citing a result of \cite{calder2016direct}, that the upper half-relaxed limit $u^{*}$ is a subsolution. This is more straightforward than proving $u_{*}$ is a supersolution. Indeed, the dynamic programming principle (Proposition \ref{prop:dpp}) implies that the height function $u_{n}$ is, in a sense that can be made precise, a subsolution of $Q_{p}$-nondominated sorting for any $p \in \mathcal{N}^{*}$.  Accordingly, as the argument here demonstrates, it suffices to invoke the scaling limit result for nondominated sorting \cite{calder2016direct}.
	
	\begin{proof}[Proof of Proposition \ref{prop:subsolution}]  Suppose that $\psi$ is a smooth function in $\mathrm{U}$, and $u^{*} - \psi$ has a strict local maximum at 0 and $u^*(0)=\psi(0)$.  Assume that $\langle D \psi(0), v_p \rangle \langle D \psi(0), w_p \rangle > 0$ for some $p \in \mathcal{N}^*$, otherwise the claim is immediate as $f \geq 0$.  By possibly reflecting (using symmetry of the norm), we may further suppose $\min(\langle D \psi(0), v_p \rangle, \langle D \psi(0), w_p \rangle) > 0$. 
		
		After this, the proof follows the nondominated sorting subsolution direction of Calder \cite{calder2016direct, calderminicourse} closely --- the only change is that 
		nondominated sorting is replaced by $Q_p$-nondominated sorting. 
		We reproduce the argument for the reader's convenience. 
		
		By varying $\psi$ away from 0, assume $\psi$ is strictly increasing
		with respect to the partial order induced by $Q_p$ and 
		$u^*(z) \leq \psi(z)$ for all $z \in -Q_p$. Let $\epsilon > 0$, $v \in Q_p$ 
		and set
		\begin{equation}
		A = \{ x \in -Q_p \, \mid \, \psi(x) \geq \psi(-\epsilon v) - \epsilon^2\}
		\end{equation}
		and
		\begin{equation}
		A_n = \{ x \in -Q_p \, \mid \, u_n(x) \geq n^{1/2} \psi(-\epsilon v)\}. 
		\end{equation}
		Note that since $\psi$ is smooth and $Q_p$-strictly increasing, $A \subseteq B(0, C \epsilon)$. 		Also, $A_n \subseteq A$ for all $n$ sufficiently large. 
		
		By the dynamic programming principle, Proposition \ref{prop:dpp}, 
		\begin{equation} \label{eq:nds-upper-bound}
		u_n(0) \leq n^{1/2} \psi(-\epsilon v) + \ell^{Q_p}(X_{n f} \cap A_n). 
		\end{equation}
		The construction of $\psi$ together with \eqref{eq:nds-upper-bound} then implies 
		\begin{equation} \label{eq:subsolution-compare}
		\langle \epsilon D \psi(0),   v \rangle - C \epsilon ^2  \leq \limsup_{n \to \infty} n^{-1/2} \ell^{Q_p}(X_{n f} \cap A).
		\end{equation}
		Also, by Taylor's formula in $A$, we have for any $x \in A$ and $y \in Q_p$
		with $|y| \leq \epsilon^2$ that 
		\[
		1 + \langle x-y , q  \rangle \geq 0
		\]
		where 
		\[
		q = \frac{ D \psi(0)}{ \langle \epsilon D \psi(0) , v \rangle + C \epsilon^2}.
		\]
		This shows $A \subseteq S_q := \{ x \in -Q_p \, \mid \, 1 + \langle x ,q \rangle \geq 0\}$, the simplex with sides
		parallel to $v_p$ and $w_p$. By \eqref{eq:subsolution-compare} and Proposition \ref{prop:nds-consistency}
		together with Lemma \ref{lemma:nds-change-of-variables}, 
		\begin{equation} \label{eq:simplex-change-of-variables}
		\langle \epsilon D \psi(0), v \rangle - C \epsilon^2 \leq |w_p \times v_p|^{1/2} \left(\frac{\sup_{y + S_q} f}{ \langle q , v_p \rangle \langle q , w_p \rangle} \right)^{1/2}
		\end{equation}
		which implies 
		\[
		\frac{\langle \epsilon D \psi(0), v \rangle - C \epsilon^2}{\langle \epsilon D \psi(0), v \rangle + C \epsilon^2} \leq |w_p \times v_p|^{1/2} \left(\frac{\sup_{y + S_q} f}{\langle D \psi(0), v_p \rangle \langle D \psi(0), w_p \rangle} \right)^{1/2}
		\]
		and by setting $v = D\psi(0)$ (so that $\langle D\psi(0), v \rangle \neq 0$), sending $\epsilon \to 0^+$, and invoking the continuity of $f$, we find
		\[
		 \frac{\langle D \psi(0), v_p \rangle \langle D \psi(0), w_p \rangle}{|w_p \times v_p|} \leq f(0),
		\]
		concluding the proof.			
	\end{proof}

\section{Conditional supersolution proof} \label{sec:supersolution}
In this section, we present a ``conditional" proof that the lower half-relaxed limit $u_{*}$ is a supersolution of \eqref{eq:eikonal}.  More precisely, we prove Proposition \ref{prop:supersolution} conditional on the assumption that the gradient points in a nondegenerate direction, \ie, $\bar{H}_{\varphi}(D\psi(x_{0})) > 0$.  In Section \ref{sec:degenerate-directions}, we prove that this is indeed the case.  The result is stated next; the remainder of the section is devoted to its proof.

\begin{prop}  \label{prop:conditional-supersolution}
	On an event of probability 1, if $x_{0} \in \mathrm{U}$, $\psi$ is a smooth function touching $u_{*}$ from below at $x_{0}$, and 
	\begin{equation} \label{eq:supersolution-gradient-assumption}
	\bar{H}_{\varphi}(D\psi(x_{0})) > 0,
	\end{equation}
	then
	\begin{equation*}
	 \bar{H}_{\varphi}(D\psi(x_{0})) \geq f(x_0).
	\end{equation*}
\end{prop}  

To simplify the proofs, we will frequently change variables so that $Q_{p} = [0,\infty)^{2}$.  This can be done using the change of coordinates $T(a w_{p} + bv_{p}) = (a,b)$.  By affine invariance (Section \ref{sec:affine-invariance}), this is no loss of generality so long as the proper accounting is carried out.

	\subsection{Legendre transform}  Once again, convex duality will play a role.  In particular, to connect the asymptotics of the longest chain to the Hamilton-Jacobi PDE, we will use the following result. 

\begin{lemma} \label{lemma:legendre-convex}  The function $h : \R^{2} \to (-\infty,0] \cup \{\infty\}$ given by 
	\begin{equation*}
	h(a,b) = \left\{ \begin{array}{r l}
	-\sqrt{ab}, & \text{if} \, \, a, b \in [0,\infty)^{2}, \\
	+ \infty, & \text{otherwise,}
	\end{array} \right.
	\end{equation*}
	is convex.  Its Legendre transform $h^{*} : \R^{2} \to \R \cup \{\infty\}$ is given by 
	\begin{equation*}
	h^{*}(c,d) = \left\{ \begin{array}{r l}
	0, & \text{if} \, \, c,d \in (-\infty,0]^{2}, \, \, |cd| \geq \frac{1}{4}, \\
	+\infty, & \text{otherwise.}
	\end{array} \right.
	\end{equation*}
\end{lemma}  

Recall that the Legendre transform is defined by
\begin{equation} \label{eq:legendre-rep}
\begin{aligned}
h^{*}(c,d) &= \sup \left\{ ca + db - p(a,b) \, \mid \, (a,b) \in \R^{2} \right\}  \\
&= \sup \left\{ ca + db + \sqrt{ab} \, \mid \, a,b \geq 0 \right\}.
\end{aligned}
\end{equation}
The proof of Lemma \ref{lemma:legendre-convex} is left to the interested reader.

\subsection{Forcing a peeling direction}
\begin{figure}
	\includegraphics[scale=0.25]{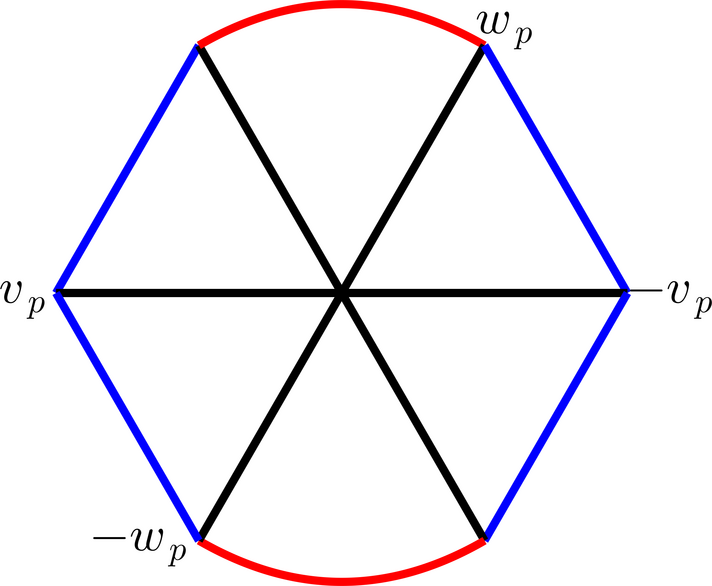} \qquad
		\includegraphics[scale=0.25]{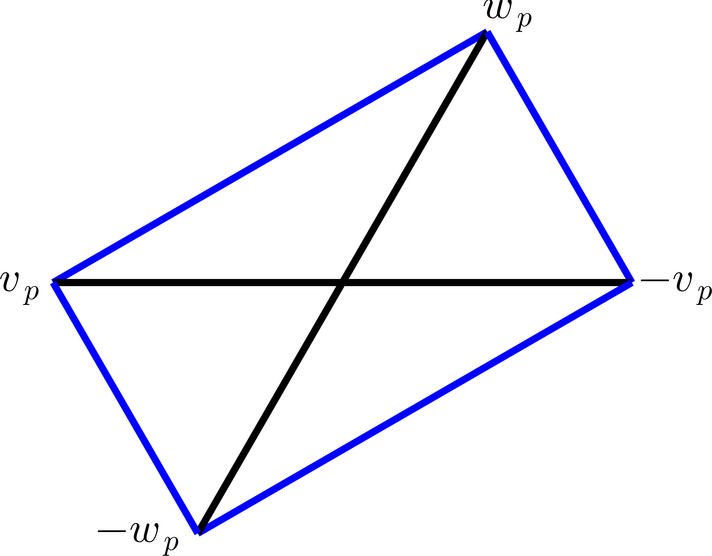}
	\caption{A norm ball and its $Q_p$-flattened norm ball.} \label{fig:flattened-ball}
\end{figure}

We introduce a certain `flattened' norm which we use to control the growth of general Pareto peeling from below.

For $p \in \mathcal{N}^*$, let the {\it $Q_p$-flattened norm} $\varphi_{Q_p}$ of $\varphi$ be the
norm with unit ball the parallelogram, $\conv(v_p, w_p, -w_p, -v_p)$. See Figure \ref{fig:flattened-ball} for an example. The $Q_p$-flattened Pareto hull is 
\begin{equation} \label{eq:flattened-hull}
\mathcal{P}_{Q_p}(A) = \{ x \in \R^2 \, \mid \, \forall y \not = x  \mbox{ there exists $a \in A$  with $\varphi_{Q_p}(a-x) < \varphi_{Q_p}(a-y)$} \}.
\end{equation}

We use the dynamic programming principle to compare Pareto hulls with their flattened hulls.
\begin{lemma}\label{lemma:cone-reduction}
	(i) Let $A$ be a finite set of points in $\R^2$, then for all $p \in \mathcal{N}^*$, 
	\begin{align*}
	x \in \inte(\mathcal{P}_{Q_p}(A)) \implies x \in \inte(\mathcal{P}(A)).
	\end{align*}
	
(ii)  Fix $x \in \R^{2}$ and $p \in \mathcal{N}^{*}$.  If there are points $y_{1}^{+}, y_{1}^{-}, y_{2}^{+}, y_{2}^{-}  \in \R^2$ such that
		\begin{equation} \label{eq:polyhedral_cone_condition}
	y_1^{\pm} \in (x \pm \inte(Q_p)) \cap A \quad \mbox{and} \quad 	 y_2^{\pm} \in (x \pm \inte(\mathcal{Q}_p)) \cap A,
		\end{equation}
	then $x \in \inte(\mathcal{P}(A))$. 
\end{lemma}

\begin{proof} \textbf{(i):} Let $p \in \mathcal{N}^*$ be given and let $x \in \inte(\mathcal{P}_{Q_p}(A))$. By Corollary \ref{cor:-constrained-convex-hull}, we know that the inclusion $x \in \inte(\conv(A))$ follows from $x \in \inte(\mathcal{P}_{Q_p}(A))$. It remains to check the cone condition. 
	
	By construction of $\varphi_{Q_p}$ and Theorem \ref{thm:cone_efficient}, there are points $\{y_1^{\pm}, y_{2}^{\pm}\}$ for which \eqref{eq:polyhedral_cone_condition} holds. Now, let $p' \in \mathcal{N}^*$. If $p' \in \{p,-p\}$, then $\{y_{1}^{+},y_{1}^{-}\} \cap (x + \inte(Q_{p'})) \neq \emptyset$ follows directly.  Otherwise, by definition,  $\inte(\mathcal{Q}_p)$, $-\inte(\mathcal{Q}_{p})$, and $\inte(\mathcal{Q}_{p'})$ are pairwise disjoint.
	Since $\mathbb{R}^{2} = Q_{p} \cup (-Q_{p}) \cup \mathcal{Q}_{p} \cup (-\mathcal{Q}_{p})$, it follows that either $\mathcal{Q}_{p'} \subseteq  Q_p$ or $\mathcal{Q}_{p'} \subseteq - Q_{p}$.  In the first case, this implies $ \mathcal{Q}_p \subseteq Q_{p'}$ so $y_2^{+} \in (x + \inte(Q_p'))$.  Similarly, in the second case, $\mathcal{Q}_{p} \subseteq - Q_{p'}$ holds and so $y_{2}^{-} \in (x - \inte(Q_{p'})$.   

\textbf{(ii):}  Combining \eqref{eq:polyhedral_cone_condition} with the second part of Theorem \ref{thm:cone_efficient}, we deduce that we must have $x \in \inte(\mathcal{P}_{Q_{p}}(A))$.  Therefore, by (i), $x \in \inte(\mathcal{P}(A))$.  \end{proof}

We next transfer this lower bound to the height function. 
To avoid introducing additional notation, we suppose without loss of generality the given cone is a quadrant.
\begin{lemma} \label{lemma:force-peel}
	Fix $A \subseteq \R^2$ finite, $x \in \R^2$ and $p \in \mathcal{N}^*$, and let $u_{A}$ be the height function of Pareto hull peeling (see \eqref{eq:height}).
	Suppose $Q_p = [0, \infty)^2$.
	Let $R_z = [z,x]$ for $z \leq x$. 
	If z is such that there are points 
	\[
	y_1 \in \bigcap_{z' \in R_z} (z' + \inte(Q_p))  \quad \mbox{and} \quad y_2^{\pm} \in \bigcap_{z' \in R_z} (z' \pm \inte(\mathcal{Q}_p))  
	\]	
	where
	\begin{equation} \label{eq:cone-peeling-upper-bound}
	\min(u_A(y_1), u_A(y_2^{\pm})) \geq u_A(x)+1,
	\end{equation}
	then 
	\[
	u_A(z) + \ell(R_z \cap A) \leq u_A(x).
	\]

\end{lemma}

\begin{proof}
	Let $x_1, \ldots, x_n$ denote a (possibly empty) longest chain in $A \cap R_z$, $x_1 \leq \cdots \leq x_n$
	and set $x_{n+1} = x$ and $x_0 = z$. 
	By \eqref{eq:cone-peeling-upper-bound}, Lemma \ref{lemma:cone-reduction}, and the dynamic programming principle, for all $0 \leq i \leq n $ 
	\[
	u_A(x_{i+1}) \geq \min(u_A(x_i) + 1, u_A(x)+1).
	\]
	Thus, by summing 
	\[
	u_A(x_{i+1}) \geq \min(u_A(z) + i, u_A(x)+1),
	\]
	and so $u_A(x) \geq \min(u_A(z) + n, u_A(x) + 1)$. Since the other inequality is impossible, $u_A(x) \geq u_A(z) + n = u_A(z) + \ell(R_z)$. 
\end{proof}

\subsection{Conditional supersolution proof}
We finally give the proof of Proposition \ref{prop:conditional-supersolution}.

Since the supersolution proof in \cite{calder2016direct} does not appear to apply directly in our setting, we resort to a different approach in the argument that follows.  Specifically, instead of studying the longest chain in a simplex, we use the convex function $h^{*}$ of Lemma \ref{lemma:legendre-convex} to relate the asymptotics of the $Q$-longest chain to the PDE.  

\begin{proof}[Proof of Proposition \ref{prop:conditional-supersolution}]  Suppose that $\psi$ is a smooth function in $\mathrm{U}$, $x_{0} \in \mathrm{U}$, and $u_{*} - \psi$ has a strict local minimum at $x_{0}$ and $u_*(x_0) = \psi(x_0)$. By Corollary \ref{cor:affine_invariance} and \eqref{eq:supersolution-gradient-assumption}, we may make a change of variables so that there is $p \in \mathcal{N}^*$ with $Q_p = [0,\infty)^2$ and $\min(\psi_{x_{1}}(x_{0}),\psi_{x_{2}}(x_{0})) > 0$.

Note that this transformation turns the Poisson process of intensity $n f$ to a Poisson process
of intensity $n f | v_p \times w_p|$. Also note that since $Q_p = [0,\infty)^2$, 
by our convention $w_p \times v_p > 0$,  $\pm \mathcal{Q}_p = \pm ((-\infty,0] \times [0,\infty)])$.

	Since $\psi$ touches $u_{*}$ from below at $x_{0}$, we know that there is a $r > 0$ such that $\{\psi > u_*(x_{0})\} \cap B(x_{0},r) \subseteq \{u_{*} > u_*(x_{0})\} \cap B(x_{0},r)$.  Accordingly, since $\min(\psi_{x_{1}}(x_{0}),\psi_{x_{2}}(x_{0})) > 0$, we can find $\delta > 0$ and 
	points $y_1, y_2^{\pm} \in B(x_0, r)$ so that
	\begin{equation} \label{eq:lower-bound-super-solution-proof1}
	\begin{aligned}
	\min(u_{*}(y_1), u_*(y_2^{\pm})) &\geq	\min( \psi(y_1), \psi(y_2^{\pm}) ) \geq u_*(x_{0}) + \delta, \\
	y_1 \in x_0 + \inte(Q_p) \quad &\mbox{and} \quad y_2^{\pm} \in x_0 \pm \inte(\mathcal{Q}_p).
	\end{aligned}
	\end{equation}
	In fact, \eqref{eq:lower-bound-super-solution-proof1} 
	can be quantified. For a set $\mathrm{C} \subseteq \R^2$ and $\delta' > 0$,  let
	\[
	\mathrm{C}^{\delta'} = \{ x \in \mathrm{C} \, \mid \, x + B(0, \delta') \subseteq \mathrm{C}\}
	\]
	be a strict subset of $\inte(\mathrm{C})$. By the definition of $u_{*}$, we can fix $(x_{0}^{(n)})_{n \in \N} \subseteq \mathrm{U}$ such that \begin{equation*}
	\lim_{n \to \infty} x_{0}^{(n)} = x_{0}, \quad \lim_{n \to \infty} n^{-\frac{1}{2}} u_n(x_{0}^{(n)}) = u_{*}(x_{0}).
	\end{equation*}
	Thus, there is $\delta' > 0$ and a (random) $N \in \N$ such that, for each $n \geq N$,
	\begin{equation} \label{eq:need-this-to-force-peels}
	\min( u_{n}(y_1), u_n(y_2^{\pm})) \geq u_n(x_{0}^{(n)}) + 1
	\end{equation}
	and
	\begin{equation}\label{eq:strict-inclusion}
	y_1 \in x_0 + (Q_p)^{\delta'} \quad \mbox{and} \quad y_2^{\pm} \in x_0 \pm (\mathcal{Q}_p)^{\delta'}.
	\end{equation}

	Let $z = -(a,b) \leq 0$. Observe that \eqref{eq:strict-inclusion} and \eqref{eq:need-this-to-force-peels}
	imply there is a $\zeta_{0}(z,\delta') > 0$ small such that  if $n \geq N$ and $\zeta \in (0,\zeta_{0}(z,\delta'))$, then the point $x_{*}^{(n)} = x_{0}^{(n)} - \zeta z$ satisfies the hypotheses of Lemma \ref{lemma:force-peel}.  Hence, by Lemma \ref{lemma:force-peel}, if we set $R^{(n)}_{\zeta} = x_0^{(n)} + [\zeta z, 0]$, then
	\begin{equation*}
	n^{-\frac{1}{2}} u^{(n)}(x_{0}^{(n)} + \zeta z) + n^{-\frac{1}{2}} \ell_{n}(R_{\zeta}^{(n)}) \leq n^{-\frac{1}{2}} u^{(n)}(x_{0}^{(n)}),
	\end{equation*}
	where $\ell_n(\mathrm{B})$ denotes the length of the longest chain in $\mathrm{B} \cap X_{|w_p \times v_p| n f}$. 
		Sending $n \to \infty$ (and recalling $z = -(a,b)$), we find
	\begin{align*}
	\lim_{n \to \infty} n^{-\frac{1}{2}} \ell_{n}(R^{(n)}_{\zeta}) &\geq 2 \zeta  \left( \inf_{R_{\zeta}^{(n)}} f \right)^{\frac{1}{2}} \sqrt{a b | v_p \times w_p|} \\
	&=: c_{f, p, \zeta}  \zeta  \sqrt{a b},
	\end{align*} 
	thus,
	\begin{equation} \label{eq:height-function-lower-bound}
	u_{*}(x_{0} - \zeta  z) + c_{f,p, \zeta}   \zeta   \sqrt{ab} \leq u_{*}(x_{0}). 		
	\end{equation}
	
	Since $\psi$ touches $u_{*}$ from below at $x_{0}$, we can transfer \eqref{eq:height-function-lower-bound} to $\psi$:
	\begin{equation}\label{eq:test-function-lower-bound}
	\psi(x_{0} - \zeta  z ) + c_{f,p, \zeta}  \zeta  \sqrt{ab} \leq \psi(x_{0}).
	\end{equation}
	Dividing by $\zeta$ and taking the limit $\zeta \to 0^{+}$ in \eqref{eq:test-function-lower-bound}, we conclude, by continuity of $f$, that
	\begin{equation} \label{eq:test-function-lower-bound-2}
	\psi_{x_{1}}(x_{0}) a + \psi_{x_{2}}(x_{0}) b \geq \bar{c}_{f,p}   \sqrt{ab}   \quad \text{if} \, \, (a,b) \in (0,\infty)^{2},
	\end{equation}
	where $\bar{c}_{f,p} = 2 \sqrt{ f(x_0) |v_p \times w_p| }$.

	Recalling the Legendre transformation from \eqref{eq:legendre-rep}, \eqref{eq:test-function-lower-bound-2} implies,
	\begin{equation*}
	h^{*}(-\psi_{x_{1}}(x_{0}),-\psi_{x_{2}}(x_{0})) = \sup \left\{ \left( -\frac{\psi_{x_{1}}(x_{0})}{\bar{c}_{f,p}} \right) a + \left(- \frac{\psi_{x_{2}}(x_{0})}{\bar{c}_{f,p}} \right) b + \sqrt{ab} \, \mid \, a, b \geq 0 \right\} \leq 0.
	\end{equation*}
	Therefore, by the explicit representation of $h^{*}$ in Lemma \ref{lemma:legendre-convex},
	\begin{equation*}
	\sqrt{\psi_{x_{1}}(x_{0}) \psi_{x_{2}}(x_{0})} \geq \frac{\bar{c}_{f,p}}{2} = \sqrt{ f(x_0) |v_p \times w_p| }.
	\end{equation*}
\end{proof}

\section{Gradient control}\label{sec:degenerate-directions}

To complete the proof of Theorem \ref{theorem:pareto-convergence}), it only remains to verify the nondegeneracy condition assumed in Proposition \ref{prop:conditional-supersolution}, that is, to check that $\bar{H}_{\varphi}(D\psi(x_{0})) > 0$ at the contact point $x_{0}$.   Once this has been checked, the proof of Proposition \ref{prop:supersolution} will be complete.  Specifically, in this section we prove the following: 
	\begin{prop} \label{prop:degenerate-directions}
 On an event of probability 1, if $x_{0} \in \mathrm{U}$ and $\psi$ is a smooth function touching $u_{*}$ from below at $x_{0}$, then 
	\begin{equation*}
	\bar{H}_{\varphi}(D\psi(x_{0})) > 0.
	\end{equation*}  \end{prop}

We break the proof of Proposition \ref{prop:degenerate-directions} into several pieces.  Note that to verify the inequality $\bar{H}_{\varphi}(D\psi(x_{0})) > 0$, it suffices to prove the following three statements:
	\begin{itemize}
		\item[(i)] $D\psi(x_{0}) \neq 0$,
		\item[(ii)] $\sup_{p \in \mathcal{N}^{*}} \langle D\psi(x_{0}), v_{p} \rangle \langle D\psi(x_{0}),w_{p} \rangle \geq 0$,
		\item[(iii)] If $\langle D\psi(x_{0}),v_{p} \rangle \langle D\psi(x_{0}), w_{p} \rangle \geq 0$ for some $p \in \mathcal{N}^{*}$, then 
			\begin{equation*}
				\langle D\psi(x_{0}), v_{p} \rangle \langle D\psi(x_{0}), w_{p} \rangle > 0.
			\end{equation*}
	\end{itemize}
This follows from the explicit formula \eqref{eq:effective-hamiltonian}.  We will show that (i)--(iii) hold in Lemmas \ref{lemma:degenerate-directions-dual-1}--\ref{lemma:degenerate-directions-nonzero} below.

We begin with (iii), that is, we show that if $D\psi(x_{0}) \neq 0$ is in a dual cone $Q_{p}^{*}$ for some $p$, then it is actually strictly inside, that is, $D\psi(x_{0}) \in \inte(Q_{p}^{*})$.

\begin{lemma} \label{lemma:degenerate-directions-dual-1} On an event of probability 1, if $x_{0} \in \mathrm{U}$, $p \in \mathcal{N}^{*}$, $\psi$ is a smooth function touching $u_{*}$ from below at $x_{0}$, and $D\psi(x_{0}) \neq 0$, then
	\begin{equation*}
	\langle D \psi(x_{0}),v_p \rangle \langle D\psi(x_{0}),w_p \rangle > 0 \quad \text{if} \quad \langle D \psi(x_{0}), v_p \rangle \langle D \psi(x_{0}), w_p \rangle \geq 0.
	\end{equation*}
\end{lemma}  

Next, we prove (ii): if $D\psi(x_{0})$ is non-zero, then it is certainly dual to one of the flat cones.

\begin{lemma} \label{lemma:degenerate-directions-dual-2} On an event of probability 1, if $x_{0} \in U$, $\psi$ is a smooth function touching $u_{*}$ from below at $x_{0}$, and $D \psi(x_{0}) \neq 0$, then
	\begin{equation*}
	\sup_{p \in \mathcal{N}^*} \langle D \psi(x_{0}), v_p \rangle   \langle D \psi(x_{0}), w_p \rangle \geq 0.
	\end{equation*}
\end{lemma}  
Finally, we prove (i): $D\psi(x_{0})$ is always non-zero.

\begin{lemma} \label{lemma:degenerate-directions-nonzero} On an event of probability 1, if $x_{0} \in U$ and $\psi$ is a smooth function touching $u_{*}$  from below at $x_{0}$, then $D\psi(x_{0}) \neq 0$.  \end{lemma}  

The lemmas above explain the discussion in the introduction concerning corners in the graph of $\bar{u}$.  This is fleshed out in the next remark.

\begin{remark} \label{remark:corners}As mentioned already in the introduction, the lemmas above imply that $\bar{u}$ is never smooth, but, instead, its graph necessarily has corners in degenerate directions.  More precisely, if $x_{0} \in \mathrm{U}$ is a point where the gradient $D\bar{u}(x_{0})$ exists, then Lemmas \ref{lemma:degenerate-directions-dual-1}- \ref{lemma:degenerate-directions-nonzero} imply that there is a $p \in \mathcal{N}^{*}$ such that 
	\begin{equation*}
		\frac{\langle D\bar{u}(x_{0}), v_{p} \rangle \langle D\bar{u}(x_{0}), w_{p} \rangle}{|v_{p} \times w_{p}|} = \bar{H}_{\varphi}(D\bar{u}(x_{0})) = 1.
	\end{equation*}
In particular, by Proposition \ref{prop:noncoercive}, this implies that the gradient, where it exists, is constrained to be everywhere nonzero $D\bar{u} \neq 0$ and 
	\begin{equation} \label{eq:nondegenerate}
		D\bar{u}^{\perp} \notin \mathcal{E} \cup \bigcup_{p \in \mathcal{N}^{*}} \partial \mathcal{Q}_{p}.
	\end{equation}
This means that $\bar{u}$ is not smooth since, for example, the gradient of a smooth function necessarily vanishes at its global maximum.   (Here, to apply the lemmas, we use the well-known fact that if $D\bar{u}(x_{0})$ exists, then $\bar{H}_{\varphi}(D\bar{u}(x_{0})) = 1$; see \cite[Chapter 2]{bardi_capuzzo-dolcetta} or \cite[Chapter 1, Section 2]{tran}.)

More generally, if $\xi$ points in a degenerate direction (\ie, if $\xi^{\perp}$ belongs to the set in the right-hand side of \eqref{eq:nondegenerate}), then the function $x \mapsto \bar{u}(x) - \langle \xi,x \rangle$ is not differentiable at its global maximum, or at any of its local maxima.  Thus, at such points, the graph of $\bar{u}$ has a corner.

The constraints on the gradient are consistent with the appearance of corners in the level sets of the simulated height functions; see, for instance, Figures \ref{fig:pareto-peeling}, \ref{fig:non-convex-example}, and \ref{fig:not-efficient} above.  Note that this is more noticeable when the set $\mathcal{E}$ is large, as, for instance, when $\varphi(x) = \max\{|x_{1} - x_{2}|,\|x\|\}$, since then the gradient constraints are more severe, see, \eg, the right-most image in Figure \ref{fig:non-convex-example}.  The analysis of corners in the level sets of $\bar{u}$ (not just the graph) is an interesting direction for future work. \end{remark}

The proofs of Lemmas \ref{lemma:degenerate-directions-dual-1}, \ref{lemma:degenerate-directions-dual-2},  and \ref{lemma:degenerate-directions-nonzero} will be based on certain growth lemmas, which are stated and proved in the sections that follow.  As in Section \ref{sec:affine-invariance}, we will change variables so that $Q_{p} = [0,\infty)^{2}$ where it simplifies the exposition.


\subsection{Heuristics} Here is a heuristic argument to motivate Lemma \ref{lemma:degenerate-directions-dual-1}.

For clarity, we restrict attention to the case when $\varphi$ equals the $\ell^{1}$ norm.  Recall that, in this case, $\mathcal{N}^{*} = \{(1,1),(1,-1),(-1,1),(-1,-1)\}$ and the cones $\{Q_{p}\}_{p \in \mathcal{N}^{*}}$ are the quadrants
	\begin{equation} \label{eq:l1-cones}
		Q_{p} = \{x = (x_{1},x_{2}) \in \mathbb{R}^{2} \, \mid \, x_{1} p_{1} \geq 0, \, \, x_{2} p_{2} \geq 0\}.
	\end{equation}
Furthermore, by the second statement in Proposition \ref{prop:dpp}, the dynamic programming principle for the height function $u_{n}$ takes the following form:
	\begin{equation} \label{eq:l1-dpp}
		u_{n}(x) = \min_{p \in \mathcal{N}^{*}} \sup_{y \in x + \inte(Q_{p})} u_{n}(y) + 1_{X_{n}}(y).
	\end{equation}

First, consider Lemma \ref{lemma:degenerate-directions-dual-1}.  Let us suppose that the convergence $\bar{u}_{n} \to \bar{u}$ is already known and argue formally that the lemma applies to $\bar{u}$.  Actually, rather than considering Lemma \ref{lemma:degenerate-directions-dual-1} in its full generality, let us treat a weaker statement:
	\begin{gather*}
		\text{if} \quad \bar{u}_{n} \to \bar{u} \, \, \text{uniformly in} \, \, \bar{\mathrm{U}} \quad \text{and} \quad x_{0} \, \, \text{is a local maximum of} \, \, \bar{u}, \\
		\text{then} \, \, \bar{u} \, \, \text{is not differentiable at} \, \, x_{0}.
	\end{gather*}

To see why this holds, we argue by contradiction.  If $\bar{u}$ is differentiable at $x_{0}$, then 
	\begin{equation*}
		\bar{u}(x) = \bar{u}(x_{0}) + o(\|x - x_{0}\|) \quad \text{as} \quad x \to x_{0}.
	\end{equation*}
Therefore, 
	\begin{equation*}
		\min\{\bar{u}(x) \, \mid \, x \in x_{0} + [-n^{-\frac{1}{2}},n^{-\frac{1}{2}}]^{2}\} \leq \bar{u}(x_{0}) + o(n^{-\frac{1}{2}}) \quad \text{as} \quad n \to \infty.
	\end{equation*}
Accordingly, at least formally, we can write
	\begin{equation} \label{eq:silly-upper-bound}
		\min\{\bar{u}_{n}(x) \, \mid \, x \in x_{0} + [-n^{-\frac{1}{2}},n^{-\frac{1}{2}}]^{2}\} \leq \bar{u}_{n}(x_{0}) + o(n^{-\frac{1}{2}}) \quad \text{as} \quad n \to \infty.
	\end{equation}
	
On the other hand, let us subdivide the box $x_{0} + [-n^{-\frac{1}{2}},n^{-\frac{1}{2}}]$ into four smaller boxes as follows:
	\begin{equation*}
		x_{0} + [-n^{-\frac{1}{2}},n^{-\frac{1}{2}}]^{2} = \bigcup_{p \in \mathcal{N}^{*}} (x_{0} + Q_{p} \cap [-n^{-\frac{1}{2}},n^{-\frac{1}{2}}]^{2}).
	\end{equation*}
Let $E_{n}$ be the event that $X_{n}$ contains a point in each of the smaller boxes, \ie,
	\begin{equation*}
		E_{n} = \bigcap_{p \in \mathcal{N}^{*}} \{X_{n} \cap (x_{0} + \inte(Q_{p}) \cap [-n^{-\frac{1}{2}},n^{-\frac{1}{2}}]^{2}) \neq \emptyset\}.
	\end{equation*} 
Since $X_{n}$ is a Poisson process of intensity $n f$ and each of the four boxes has area $n^{-1}$, $E_{n}$ has probability of order one:
	\begin{equation*}
		\mathbb{P}(E_{n}) \approx f(x_{0}).
	\end{equation*}
Furthermore, due to \eqref{eq:l1-cones} and the dynamic programming principle \eqref{eq:l1-dpp},
	\begin{equation*}
		\bar{u}_{n}(x_{0}) \geq \min\{\bar{u}_{n}(x) \, \mid \, x \in x_{0} + [-n^{-\frac{1}{2}},n^{-\frac{1}{2}}]^{2}\} + n^{-\frac{1}{2}} \quad \text{on the event} \, \, E_{n}.
	\end{equation*}
Therefore, due to averaging (this part needs justification), there is a universal constant $\rho(f(x_{0})) > 0$ such that, with probability one, for all $n$ large enough,
	\begin{equation*}
		\bar{u}_{n}(x_{0}) \geq \min\{\bar{u}_{n}(x) \, \mid \, x \in x_{0} + [-n^{-\frac{1}{2}},n^{-\frac{1}{2}}]^{2}\} + \rho (f(x_{0})) n^{-\frac{1}{2}}.
	\end{equation*}
This contradicts \eqref{eq:silly-upper-bound}, completing the formal proof.

\begin{remark} Note that the heuristic proof above boils down to a ``growth lemma:" we argue that, close to a local maximum, $\bar{u}$ grows in a manner that is inconsistent with differentiability.  (At the level of the PDE, this can be seen, say, for the $\ell^{1}$ norm, by observing that the function $u(x) = -\|x\|_{1}$ is a solution of $\bar{H}_{\varphi}(Du) = 1$ and applying a comparison argument.) \end{remark}

\begin{remark} Similar reasoning can be used to motivate Lemma \ref{lemma:degenerate-directions-nonzero}, a fact that may be worth keeping in mind while reading its proof. \end{remark}

\subsection{Box growth} \label{sec:box_growth}
We start with a fundamental growth estimate for the height function in a square.  In addition to proving that the height function of $n$ random points in a square is at least order $\sqrt{n}$ at the center, the estimate implies that the limiting height function is not differentiable at any of its local maxima, making rigorous the previous heuristic. 

\begin{lemma} \label{lemma:box-growth}
	Suppose $Q_p = [0,\infty)^2$ for some $p \in \mathcal{N}^*$.  There is a function $\rho:(0,\infty) \to (0,\infty)$ so that on an event of probability 1, if $x_{0} \in \Q^{2}$ and $a \in \Q \cap (0,\infty)$ are chosen so that
	\[
	x_{0} + [-a,a]^2  \subseteq \{f > f(x_{0})/2\},
	\]
	and $y^{+,+}, y^{+,-}, y^{-,+}, y^{-,-} \in \mathbb{R}^{2}$ satisfy
	\begin{gather*}
	y^{+,+} \in (x_0 + (a,a)) + \inte(Q_p), \quad y^{+,-} \in  (x_0 + (a,-a)) - \inte(\mathcal{Q}_p), \\
	y^{-,+} \in (x_{0} + (-a,a)) + \inte(\mathcal{Q}_{p}), \quad y^{-,-} \in (x_{0} + (-a,-a)) - \inte(Q_{p}),
	\end{gather*}
	then for all $n$ sufficiently large,
	\[
	u_n(x_{0}) \geq \min (u_n(y^{+,+}), u_n(y^{+,-}), u_{n}(y^{-,+}), u_{n}(y^{-,-})) + \rho(f(x_0))  a  \sqrt{n}.
	\]
\end{lemma}

In the proof of the estimate, we will use the following observation about Poisson processes.

	\begin{prop} \label{prop:LLN}   Fix $n \in \mathbb{N}$, $\gamma > 0$, and $f \in L^{\infty}_{\text{loc}}(\mathbb{R}^{2})$, and let $X_{nf}$ be a Poisson process in $\mathbb{R}^{2}$ of intensity $nf$.  If $R \subseteq \mathbb{R}^{2}$ is a cube of side length $n^{-\frac{1}{2}}$ such that $f(x) \geq \gamma$ for each $x \in R$, then there is a another cube $R' \subseteq R$ such that the random variable $1\{X_{nf} \cap R' \neq \emptyset\}$ is $\mbox{Bernoulli}(p')$ with $p' = 1 - \exp(-\gamma)$.  \end{prop}  

\begin{proof}  Since $R$ is a cube, we can write $R = x_{0} + [-2^{-1}n^{1/2},2^{-1}n^{1/2}]^{2}$ for some $x_{0} \in \mathbb{R}^{2}$.  By hypothesis, we have $\int_{R} f(x) \, dx \geq \gamma n^{-1}$.  Thus, by the intermediate value theorem, there is a $t \in [0,1]$ such that $\int_{R(t)} f(x) \, dx = \gamma n^{-1}$, where $R(t) = x_{0} + [-t2^{-1} n^{-1/2},t2^{-1} n^{-1/2}]^{2}$.  We conclude by letting $R' = R(t)$ and invoking elementary properties of Poisson processes.\end{proof}  

\begin{figure}
	\input{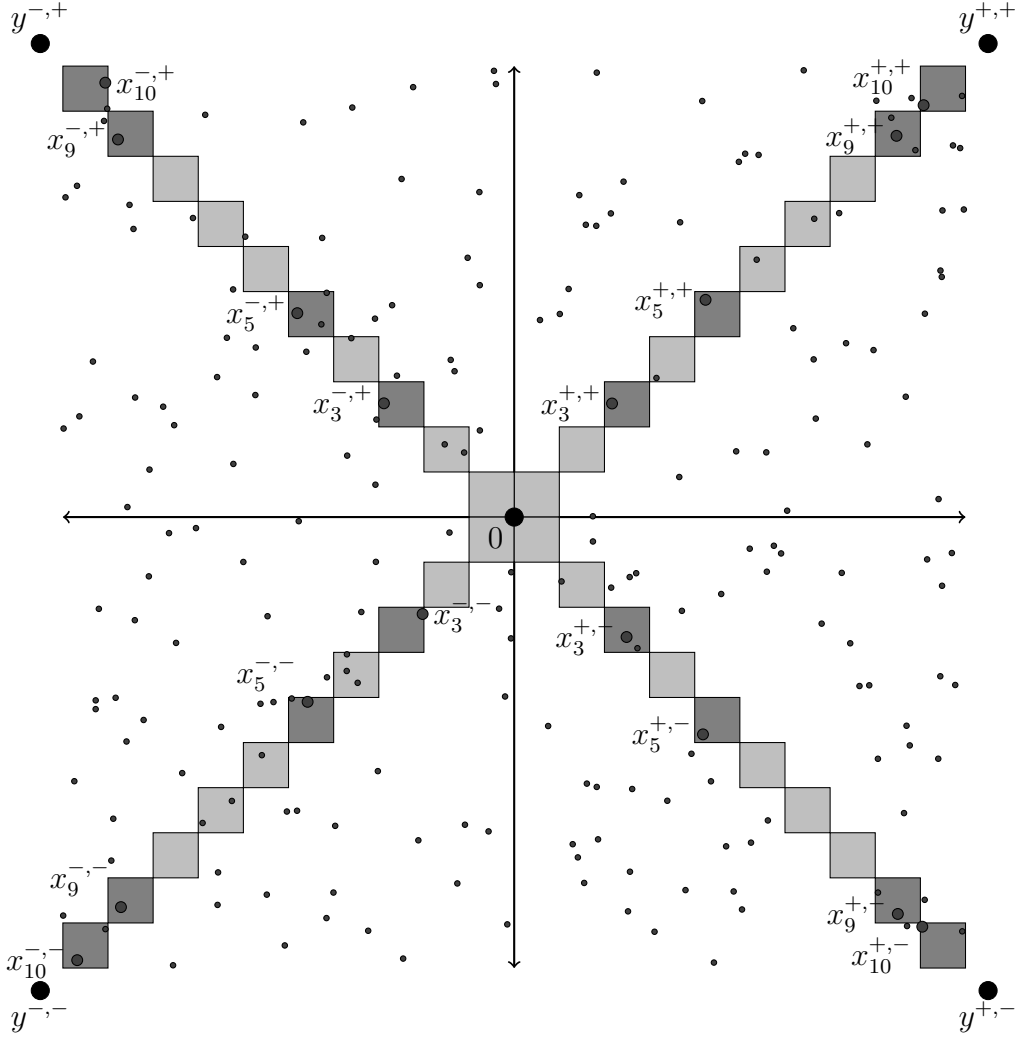}
	\caption{Construction in the proof of Lemma \ref{lemma:box-growth}.  To obtain a lower bound on $u_{n}(0) - \min(u_{n}(y^{+,+}),u_{n}(y,^{+,-}), u_{n}(y^{-,+}), u_{n}(y^{-,-}))$, we restrict attention to the points of $X_{nf}$ in the diagonal boxes (light gray), each of side length $n^{-\frac{1}{2}}$.  Dark gray boxes correspond to indices $\ell$ for which $A_{\ell} = 1$.  By the law of large numbers, the number of indices grows at rate $\sqrt{n}$, hence so does the desired lower bound.
	} \label{fig:box-growth}
\end{figure}

\begin{proof}[Proof of Lemma \ref{lemma:box-growth}]

		We split the proof into steps. We first identify an event of full probability
		which we then show leads to the desired lower bound. 
	
		Recall that since $Q_p = [0,\infty)^2$, by our convention that $w_p \times v_p > 0$, we have that $\pm \mathcal{Q}_p = \pm (-\infty,0] \times [0,\infty)$.

		{\it Step 1.} 
		Translate so that $x_0 = 0$ and let $\gamma = f(0)/2$. For $n \geq 1$, cover $[-a,a]^2$ by a disjoint grid of identical cubes of side length $n^{-1/2}$ where each such cube, $R_z$, is centered at a point $n^{-1/2} z$ for 
		$z \in \Z^2 \cap [-m,m]^2$ where $m = \lceil a \sqrt{n} \rceil$.
		
		For each $\ell \in \{1, \ldots, m\}$, let 
		\[
		A_{\ell} = \prod 1\{ X_{n f} \cap R_{(\pm \ell,\mp \ell)} \neq \emptyset \},
		\]
		denote the indicator of the event that four `corner' cubes contain a point from the Poisson process. By considering slightly smaller cubes $R'_{(\pm \ell, \mp \ell)} \subseteq R_{(\pm \ell, \mp \ell)}$ as in Proposition \ref{prop:LLN}, we observe that $A_{\ell} \geq A_{\ell}'$, where $\{A_{1}',\dots,A_{m}'\}$ are independent $\mbox{Bernoulli}(p)$ random variables where
		$p$ is independent of $n$: 
		\[
		p^{1/4} = P(\Poisson(\gamma) \geq 1) = 1 - \exp(- \gamma).
		\]
		Therefore, 
		\begin{equation} \label{eq:lower-bound}
		\Gamma_n = \sum_{\ell=1}^m A_{\ell}
		\end{equation}
		dominates a Binomial with mean $m p$. Thus, by, say, the strong law of large numbers, on an event, $\Omega_{x_0}$, of probability one, for all $n$ sufficiently large $\Gamma_n \geq m p/2$. 
		
		{\it Step 2.} 
		We next argue as in Lemma \ref{lemma:force-peel} to transfer the lower bound on $\Gamma_n$
		to the height functions. Let $\ell_1 \geq \cdots \geq \ell_k$ for $k = \Gamma_n$ be a sequence of indices with $A_{\ell_i} = 1$. Observe that by construction 
		\[
		R_{(\pm \ell_i, \pm \ell_i)} \subseteq x \mp \inte(Q_{p}) \quad \mbox{and} \quad R_{(\mp \ell_i, \pm \ell_i)} \subseteq x \mp \inte(\mathcal{Q}_{p})
		\]
		for all $x \in R_{(\pm \ell_i', \pm \ell_i')}$ with $\ell_i' > \ell_i$.
		Hence, since $A_{\ell_{i}} = 1$ for all $i$,
		there exists a list of quadruples of (random) points, $x^{\pm, \mp}_{\ell_i} \in R_{(\pm \ell_i, \mp \ell_i)} \cap X_{n}$
		which, in view of Lemma \ref{lemma:cone-reduction} and dynamic programming, satisfy
		\[
		\min_{\pm, \mp} (u_n(x^{\pm,\mp}_{\ell_{i}})) \geq \min_{\pm, \mp} (u_n(x^{\pm,\mp}_{\ell_{i + 1}})) + 1.
		\]
		 Therefore, by induction,
		 \[
		 u_n(0) \geq \min(u_n(y^{+,+}), u_{n}(y^{+,-}), u_{n}(y^{-,+}), u_{n}(y^{-,-})) + k,
		 \]
		 and
		 \[
		 k \geq mp/2 \geq  c a \sqrt{n} (1- \exp(-C f(0))^4 =: \rho(f(0)) a \sqrt{n}.
		 \]
		 
		{\it Step 3.} Conclude by observing $\Omega = \bigcap_{x_0 \in \Q^2} \Omega_{x_0}$, where $\Omega_{x_0}$ is as in the end of Step 1, has full probability. 		
\end{proof}

\subsection{Planar growth}  \label{sec:planar_growth} Next, we prove a planar growth lemma.  As will become clear shortly, this lower bound establishes that $Du^{\perp} \notin \partial \mathcal{Q}_{p}$ for any $p$, \ie, Lemma \ref{lemma:degenerate-directions-dual-1} holds.  Geometrically, this seems to explain why the level sets in the simulations appear to develop corners in certain directions (cf.\ Figures \ref{fig:pareto-peeling} and \ref{fig:non-convex-example}).

\begin{lemma} \label{lemma:unidirectional-growth}
	Suppose $Q_p = [0,\infty)^2$ for $p \in \mathcal{N}^*$. There is a function $\rho:(0,\infty) \to (0,\infty)$ so that on an event of probability 1, if  $x_{0} \in \Q^{2}$ and $a,b \in \Q \cap (0,\infty)$ are chosen so that 
	\[
	x_{0} + [-a,0] \times [-b,b] \subseteq \{f > f(x_{0})/2\},
	\]
	$y^{+,+}$ and $y^{+,-}$ are any two points in $\mathbb{R}^{2}$ satisfying
	\[
	y^{+,+} \in (x_0 + (0,b) + \inte(Q_p)) \quad \mbox{and} \quad y^{+,-} \in (x_0 + (0,-b) - \inte(\mathcal{Q}_p)),
	\]
	and
	\[
	s := \min \left\{ u_n(z) \, \mid \, z \in x_{0} + [-a,0] \times [-b,b] \right\},
	\]
	then for all $n$ sufficiently large, 
	\begin{equation} \label{eq:growth-estimate-result1}
	u_n(x_{0}) \geq \min \left\{u_n(y^{+,+}), u_n(y^{+,-}), s + \rho(f(x_0))  \sqrt{ a b n} \right\}.
	\end{equation}
\end{lemma}  

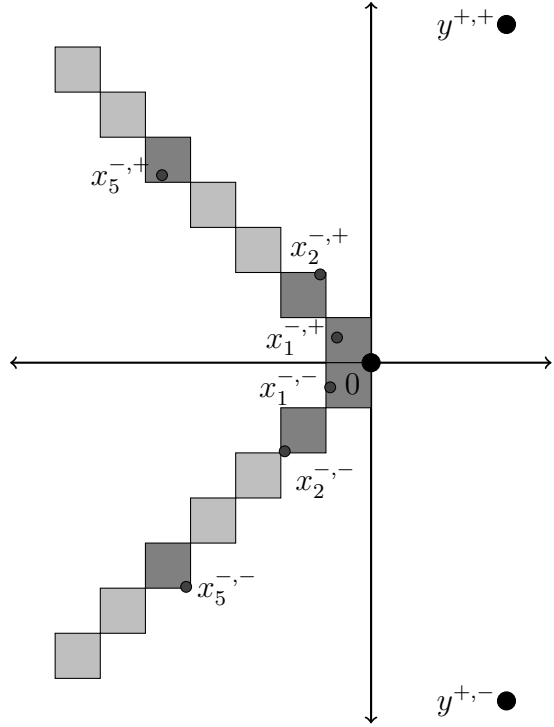
\begin{figure}
	\begin{tikzpicture}[scale=1.2]


\draw[thick,->] (0,0) -- (0,4);
\draw[thick,->] (0,0) -- (0,-4); 
\draw[thick,->] (0,0) -- (-4,0);
\draw[thick,->] (0,0) -- (2,0);

\filldraw[fill = gray] (-0.5,0) rectangle (0,0.5);
\filldraw[fill = gray] (-0.5,-0.5) rectangle (0,0);

\filldraw[fill = gray] (-1,0.5) rectangle (-0.5,1);
\filldraw[fill = gray] (-1,-1) rectangle (-0.5,-0.5); 

\filldraw[fill = lightgray] (-1.5,1) rectangle (-1,1.5);
\filldraw[fill = lightgray] (-1.5,-1.5) rectangle (-1,-1);

\filldraw[fill = lightgray] (-2,1.5) rectangle (-1.5,2);
\filldraw[fill = lightgray] (-2,-2) rectangle (-1.5,-1.5);

\filldraw[fill = gray] (-2.5,2) rectangle (-2,2.5); 
\filldraw[fill = gray] (-2.5,-2.5) rectangle (-2,-2);

\filldraw[fill = lightgray] (-3,2.5) rectangle (-2.5,3);
\filldraw[fill = lightgray] (-3,-3) rectangle (-2.5,-2.5);

\filldraw[fill = lightgray] (-3.5,3) rectangle (-3,3.5); 
\filldraw[fill = lightgray] (-3.5,-3.5) rectangle (-3,-3);


\draw[anchor=east] (-0.3774,0.2792) node {$x_{1}^{-,+}$};
\filldraw[fill = darkgray] (-0.3774,0.2792) circle (0.06);
\draw[anchor=east] (-0.4529,-0.2734) node {$x_{1}^{-,-}$};
\filldraw[fill = darkgray] (-0.4529,-0.2734) circle (0.06);

\draw[anchor=south] (-0.5635,0.9788) node {$x_{2}^{-,+}$};
\filldraw[fill = darkgray] (-0.5635,0.9788) circle (0.06);
\draw[anchor=north west] (-0.9567,-0.9824) node {$x_{2}^{-,-}$};
\filldraw[fill = darkgray] (-0.9567,-0.9824) circle (0.06);

\draw[anchor=east] (-2.3162,2.0788) node {$x_{5}^{-,+}$};
\filldraw[fill = darkgray] (-2.3162,2.0788) circle (0.06);

\draw[anchor= west] (-2.0488,-2.4853) node {$x_{5}^{-,-}$};
\filldraw[fill = darkgray] (-2.0488,-2.4853) circle (0.06);


\draw[anchor=north east] (0,0) node {$0$};
\filldraw[fill=black] (0,0) circle (0.1);

\draw[anchor=east] (1.5,3.75) node {$y^{+,+}$};
\filldraw[fill=black] (1.5,3.75) circle (0.1);

\draw[anchor=east] (1.5,-3.75) node {$y^{+,-}$};
\filldraw[fill=black] (1.5,-3.75) circle (0.1);

\end{tikzpicture}
	\caption{Construction in the proof of Lemma \ref{lemma:unidirectional-growth}.
	} \label{fig:unidirectional-growth}
\end{figure}

\begin{proof}
Define the transformation $T: \R^2 \to \R^2$ by 
\[
T(v) = \sqrt{\frac{b}{a}} v_{1}
+ \sqrt{\frac{a}{b}} v_2
\]	
and observe that $T$ maps $[-a,0] \times [b,b]$ to $[-\sqrt{a b}, 0] \times [-\sqrt{a b}, \sqrt{ a b}]$ and $\det(T) = 1$. After making this transformation, the rest of the proof is almost identical to that of Lemma \ref{lemma:box-growth}, the only change is that the growth bound is in one direction. Thus, we only sketch it. 

{\it Step 1.} Set $d = \sqrt{a b}$, translate so that $x_0 = 0$, and let $\gamma = f(0)/2$. For $n \geq 1$, cover $[-d,0] \times [-d,d]$ by a disjoint grid of identical cubes of side length $n^{-1/2}$ where each, $R_z$, is centered at a point $n^{-1/2} z$ for $z \in \Z^2 \cap ([-m,0] \times [-m, m])$ where $m = \lceil d \sqrt{n} \rceil$.

For each $\ell \in \{1, \ldots, m\}$, let 
\[
A_{\ell} = \prod 1\{ X_{n f} \cap R_{(-\ell,\pm \ell)} \neq \emptyset \},
\]
denote the indicator of the event that two corners contain a point from the Poisson process. By the domination argument in Lemma \ref{lemma:box-growth},
on an event, $\Omega_{x_0}$, of probability one, for all $n$ sufficiently large, $\Gamma_n \geq \rho(f(x_0)) \sqrt{ a b n}$. 

{\it Step 2.} 
Following the argument in Lemma \ref{lemma:box-growth}, there exists a sequence of indices 
$\ell_1 \geq \cdots \geq \ell_{\Gamma_n}$ 
and $x^{-, \pm}_{\ell_i} \in R_{(-\ell_i, \pm \ell_i)} \cap X_{n}$
which satisfy, by dynamic programming, 
\[
\min_{\pm} (u_n(x^{-,\pm}_{\ell_{i}})) \geq \min(u_n(y^{+,+}), u_n(y^{+,-}), \min_{\pm} (u_n(x^{-,\pm}_{\ell_{i + 1}})) + 1).
\]
We iterate to conclude. 
\end{proof}

\subsection{Proof of Lemma \ref{lemma:degenerate-directions-dual-1}}
By an affine transformation, using Corollary \ref{cor:affine_invariance}, we may assume $Q_p = [0, \infty)^2$ so that $\mathcal{Q}_p = (-\infty,0] \times [0,\infty)$ and 
$\psi_{x_{1}}(x_{0}) \psi_{x_{2}}(x_{0}) \geq 0$.  Up to another affine transformation, we can assume that $\min(\psi_{x_{1}}(x_{0}),\psi_{x_{2}}(x_{0})) \geq 0$.

In view of the changes of coordinates we just employed, it remains to show that 
	\begin{equation*}
		\psi_{x_{1}}(x_{0}) \psi_{x_{2}}(x_{0}) > 0
	\end{equation*} 
or, equivalently, that $\min(\psi_{x_{1}}(x_{0}),\psi_{x_{2}}(x_{0})) > 0$.  By symmetry, we only need to prove that $\psi_{x_{2}}(x_{0}) \neq 0$.   

Suppose, for sake of contradiction, that $\psi_{x_2}(x_0) = 0$.  Let $t = u_*(x_0) = \psi(x_0)$.

Since $D \psi(x_0) \not = 0$ by assumption, we must have $\psi_{x_1}(x_0) > 0$. This means, by Taylor approximation, there are positive constants $\epsilon_0 $ and $c$ such that 
\[
\psi(x) \geq t -  c \epsilon^2 \quad \mbox{for $x \in x_0 \pm 2 R_{\epsilon} $}
\]
and $\epsilon \in (0, \epsilon_0)$ where $R_\epsilon = [-\epsilon^2, 0] \times [-\epsilon,\epsilon]$.  Given such $\epsilon$, pick rational points $\{x_i^{\pm}\} \in (x_0 +  \{ R_{\epsilon} \cup -R_{\epsilon}\})$ such that 
\[
x_1^{\pm} \in \bigcap_{x_0' \in B(x_0, \delta)} (x_0' \pm \inte(Q_p)) \quad \mbox{and} \quad x_2^{\pm} \in \bigcap_{x_0' \in B(x_0, \delta)} (x_0' \pm \inte(\mathcal{Q}_p)),
\]
where $\delta =  \min(\epsilon^2/4, \epsilon/4)$.
Also, let 
\[
y_1^+ = x_0 + (2 \epsilon^2, 0) + (0, 4 \epsilon) \quad \mbox{and} \quad y_2^+ = x_0 + (2 \epsilon^2,0) - (0,4 \epsilon) 
\]
and observe that $y_1^+ \in (x_i^{\pm} + (0,\epsilon) + \inte(Q_p))$ and $y_2^+ \in (x_i^{\pm} + (0,-\epsilon) - \inte(\mathcal{Q}_p))$ for $i = 1, 2$.
Furthermore, making $\epsilon_{0}$ smaller if necessary, we have
\[
\min(\psi(y_1^+), \psi(y_2^+)) \geq t + c \epsilon^{2}.
\]
Since $\psi$ touches $u_{*}$ below, we deduce that for all $n$ sufficiently large
\[
\begin{aligned}
\min(u_{n}(y_1^+), u_n(y_2^+)) &\geq (t + c \epsilon^{2})n^{1/2}, \\
u_{n}(x) \geq (t - c \epsilon^{2})n^{1/2} &\quad \mbox{for $x \in \{x_0 \pm 2 R_{\epsilon} \}$}.
\end{aligned}
\]
Therefore, by Lemma \ref{lemma:unidirectional-growth} and Lemma \ref{lemma:cone-reduction}, again taking $\epsilon_0$ smaller if necessary so that $x_0 \pm 2 R_{\epsilon} \subseteq \{f > f(x_0)/2\}$,  
\[
\min_{x_0' \in B(x_0, \delta)} u_n(x_0') \geq \min \left[ (t - c \epsilon^{2}) n^{1/2} + (C  \epsilon^{3/2} \rho(f(x_0)) ) n^{1/2},  (t + c \epsilon^{2}) n^{1/2}\right].
\]
Sending $n \to \infty$, this implies,
\[
t = u_*(x_0) \geq \min(t - c \epsilon^{2} + C  \epsilon^{3/2} \rho(f(x_0)) ,  t + c \epsilon^{2}),
\]
a contradiction for $\epsilon$ sufficiently small as $\rho(f(x_0)) > 0$. 

\subsection{Proof of Lemma \ref{lemma:degenerate-directions-nonzero}}
The argument is similar to the proof of Lemma \ref{lemma:degenerate-directions-dual-1}
so we only sketch it. 

We argue by contradiction. Again, by an affine transformation, using Corollary \ref{cor:affine_invariance}, we may assume $Q_p = [0, \infty)^2$ so that $\mathcal{Q}_p = [0,\infty) \times (-\infty,0]$ and 
$\psi_{x_1}(x_0) = \psi_{x_2}(x_0)  = 0$.   Let $t = u_*(x_0) = \psi(x_0)$. By Taylor approximation, there are positive constants $\epsilon_0 $ and $c$ such that 
\[
\psi(x) \geq t -  c \epsilon^2 \quad \mbox{for $x \in \{x_0 + R_{\epsilon} \}$}
\]
and $\epsilon \in (0, \epsilon_0)$ where $R_\epsilon = [-\epsilon,\epsilon]^2$. Given such $\epsilon$, let
\[
y_1^{\pm} = x_0 \pm (2 \epsilon, 2 \epsilon)  \quad \mbox{and} \quad y_2^{\pm} = x_0 \pm (2 \epsilon, -2 \epsilon) 
\]
and observe that 
\[
y_1^{\pm} \in ( (x_0 \pm (\epsilon,\epsilon)) \pm \inte(Q_p)) \quad \mbox{and} \quad y_2^{\pm} \in ( (x_0 \pm (-\epsilon,\epsilon)) \pm \inte(\mathcal{Q}_p)).
\]

Making $\epsilon$ sufficiently small, exactly as in the proof of Lemma \ref{lemma:degenerate-directions-dual-1}, we deduce that for all $n$ large
\[
u_{n}(x) \geq (t - c \epsilon^{2})n^{1/2} \quad  \mbox{for $x \in x_0 + R_{\epsilon} $},
\]
and hence by Lemma \ref{lemma:box-growth},
\[
\min_{x_0' \in B(x_0, c\epsilon)} u_n(x_0') \geq  (t - c \epsilon^{2}  + C \epsilon \rho( f(x_0) )) n^{1/2},
\]
which leads to a contradiction as in the last proof.

\subsection{Proof of Lemma \ref{lemma:degenerate-directions-dual-2}} We now prove Lemma \ref{lemma:degenerate-directions-dual-2}.  The proof follows a similar strategy to the one employed in Lemma \ref{lemma:degenerate-directions-dual-1}, except the growth lemmas require some additional geometric reasoning.  The difference can be explained by the fact that whereas in Lemma \ref{lemma:degenerate-directions-dual-1}, the tangent vector $D\psi(x_{0})^{\perp}$ belongs to one of the flat cones appearing in the dynamic programming principle, in the present scenario, some work is needed to relate the tangent vector to those cones.

More precisely, to prove Lemma \ref{lemma:degenerate-directions-dual-2}, we argue by contradiction.  Hence we are interested in the case when the gradient of the test function $D\psi(x_{0})$ at the contact point $x_{0}$ is such that
	\begin{equation*}
		\sup_{p \in \mathcal{N}^{*}} \langle D\psi(x_{0}),v_{p} \rangle \langle D\psi(x_{0}), w_{p} \rangle < 0.
	\end{equation*}
In view of Propositions \ref{prop:key-duality} and \ref{prop:noncoercive}, this is equivalent to assuming that
	\begin{equation} \label{eq:tilde-E}
		D\psi(x_{0})^{\perp} \in \cone(\tilde{\mathcal{E}}), \quad \text{where} \quad \tilde{\mathcal{E}} := \mathcal{E} \setminus \bigcup_{p \in \mathcal{N}^{*}} \mathcal{Q}_{p}.
	\end{equation}

\begin{remark} It is worth pointing out at this stage that the set $\tilde{\mathcal{E}}$ is empty if $\varphi$ is polyhedral (see Definition \ref{def:norms}).  Thus, Lemma \ref{lemma:degenerate-directions-dual-2} is vacuously true for polyhedral norms $\varphi$, such as the $\ell^{1}$ norm and the $\ell^{\infty}$ norm, and, therefore, the proof of Theorem \ref{theorem:pareto-convergence} for such norms is already complete at this stage of the paper.  The arguments that remain are the most technical part of the paper, their difficulty stemming from the fact that, in general, the set $\tilde{\mathcal{E}}$ may be very rough (\eg, a Cantor set).   \end{remark}

	To begin, we fix a subset $\mathcal{D} \subseteq \tilde{\mathcal{E}}$ such that
		\begin{equation} \label{eq:dense-set-approximation}
			\mathcal{D} \, \, \text{is a countable, dense subset of} \, \, \tilde{\mathcal{E}}.
		\end{equation}
	Such a set necessarily exists since $\mathbb{R}^{2}$ is a separable metric space.  Since $\tilde{\mathcal{E}}$ can be uncountable in general, to avoid measurability issues we use $\mathcal{D}$ as a countable approximation of $\tilde{\mathcal{E}}$.
			
Like Lemma \ref{lemma:degenerate-directions-dual-1}, the proof of Lemma \ref{lemma:degenerate-directions-dual-2} follows from a growth lemma.  In order to streamline the exposition, we state the main consequence of the growth lemma as a separate result, which is stated next.  In what follows, given a $v \in \mathbb{R}^{2}$, we define cones
\begin{equation} \label{eq:plus-minus-cones}
	\begin{aligned}
		Q_{+,+}(v) &= \{v' \in \mathbb{R}^{2} \, \mid \, \langle v', v \rangle \geq 0, \, \, \langle v', v^{\perp} \rangle \geq 0\}, \\
		Q_{-,+}(v) &= \{v' \in \mathbb{R}^{2} \, \mid \, \langle v', v \rangle \leq 0, \, \, \langle v', v^{\perp} \rangle \geq 0\}, \\
		Q_{+,-}(v) &= \{v' \in \mathbb{R}^{2} \, \mid \, \langle v', v \rangle \geq 0, \, \, \langle v', v^{\perp} \rangle \leq 0\}, \\
		Q_{-,-}(v) &= \{v' \in \mathbb{R}^{2} \, \mid \, \langle v', v \rangle \leq 0, \, \, \langle v', v^{\perp} \rangle \leq 0\}, \\
		A_{\epsilon}(v) &= \left\{v' \in \mathbb{R}^{2} \, \mid \, |\langle v', v^{\perp} \rangle| < \epsilon |\langle v', v \rangle| \right\},
	\end{aligned}
	\end{equation}
and, for $a,b > 0$, we define the rectangle
	\begin{equation*}
		D^{v}_{-}[a,b] = \{v' \in \mathbb{R}^{2} \, \mid \, -a \leq \|v^{\perp}\|^{-1} \langle v', v^{\perp} \rangle \leq 0, \, \, \|v\|^{-1} |\langle v', v \rangle| \leq b\}.
	\end{equation*}

	\begin{lemma} \label{lemma:reduced-planar-growth} There are functions $\zeta : \tilde{\mathcal{E}} \to (0,\infty)$ and $\rho : (0,\infty) \to (0,\infty)$ such that, on an event of probability one, if for some $q \in \tilde{\mathcal{E}}$, $x^{-,+},x_{0},x^{+,+} \in U$, and $a,b > 0$, we have
		\begin{gather*}
			x_{0} + D^{q}_{-}[a,b] \subseteq \{f > f(x_{0})/2\}, \\
			 x_{0} + D^{q}_{-}[a,b] \subseteq x^{+,+} - A_{\zeta(q)}(q) \cap \inte(Q_{+,+}(q)), \\
			 x_{0} + D^{q}_{-}[a,b] \subseteq x^{-,+} - A_{\zeta(q)}(q) \cap \inte(Q_{-,+}(q)),
		\end{gather*}
	and 
		\begin{equation*}
			s := \min \left\{ u_{*}(x) \, \mid \, x \in x_{0} + D^{q}_{-}[a,b]\right\},
		\end{equation*}
	then 
		\begin{equation*}
			u_{*}(x_{0}) \geq \min \left\{ u_{*}(x^{-,+}), u_{*}(x^{+,+}), s + \rho(f(x_{0})) \sqrt{2ab} \right\}.
		\end{equation*}\end{lemma}  
		
For the reader's convenience, here is how to deduce Lemma \ref{lemma:degenerate-directions-dual-2} from Lemma \ref{lemma:reduced-planar-growth}:

\begin{proof}[Proof of Lemma \ref{lemma:degenerate-directions-dual-2}]    We argue by contradiction, \ie, we assume that 
	\begin{equation*}
		\sup_{p \in \mathcal{N}^{*}} \langle D\psi(x_{0}), v_{p} \rangle \langle D\psi(x_{0}), w_{p} \rangle < 0.
	\end{equation*}
Thus, by Propositions \ref{prop:key-duality} and \ref{prop:noncoercive}, 
	\begin{equation} \label{eq:gradient-part}
		D\psi(x_{0})^{\perp} \in \cone(\tilde{\mathcal{E}}),
	\end{equation}
where $\tilde{\mathcal{E}}$ is the set defined by \eqref{eq:tilde-E}.  

Let $q = - \varphi(D\psi(x_{0})^{\perp})^{-1} D\psi(x_{0})^{\perp}$.  By \eqref{eq:gradient-part}, $q \in \tilde{\mathcal{E}}$ holds.  Observe that, by continuity, there is a small $\mu \in (0,1)$ such that
	\begin{equation*}
		D^{q}[\mu,\mu] \subseteq 2 \|q\|^{-1}\left( q + \frac{\zeta(q)}{2} q^{\perp} \right) - A_{\zeta(q)}(q) \cap \inte(Q_{+,+}(q)).
	\end{equation*}
In particular, this implies that, for any $\epsilon > 0$, we have
	\begin{equation} \label{E: tedious geometric thing 1}
		D^{q}[\mu \epsilon, \mu \epsilon] \subseteq 2\epsilon \|q\|^{-1} \left( q + \frac{\zeta(q)}{2} q^{\perp}\right) - A_{\zeta(q)}(q) \cap \inte(Q_{+,+}(q))
	\end{equation}
and then, by symmetry,
	\begin{equation} \label{E: tedious geometric thing 2}
		D^{q}[\mu\epsilon,\mu\epsilon] \subseteq 2 \epsilon \|q\|^{-1} \left( -q + \frac{\zeta(q)}{2} q^{\perp} \right) - A_{\zeta(q)}(q) \cap \inte(Q_{-,+}(q)).
	\end{equation}
	
Let $\epsilon > 0$.  If we define $x^{+,+}$ and $x^{-,+}$ by
	\begin{align*}
		x^{+,+} = x_{0} + 2\epsilon \|q\|^{-1} \left( q + \frac{\zeta(q)}{2} q^{\perp} \right), \quad x^{-,+} = x_{0} + 2\epsilon \|q\|^{-1} \left( -q + \frac{\zeta(q)}{2} q^{\perp} \right),
	\end{align*}
then we have
	\begin{align*}
		u_{*}(x^{\pm,+}) &\geq \psi(x^{\pm,+}) \geq u_{*}(x_{0}) + \epsilon \zeta(q) \|D\psi(x_{0})\| - C \epsilon^{2} 
	\end{align*}
and
	\begin{align*}
		\min\left\{ u_{*}(x) \, \mid \, x \in x_{0} + D^{q}_{-}[\epsilon^{2},\mu\epsilon] \right\} &\geq \min \left\{ \psi(x) \, \mid \, x \in x_{0} + D^{q}_{-}[\epsilon^{2},\mu \epsilon] \right\} \\
			&\geq u_{*}(x_{0}) - \epsilon^{2} \|D\psi(x_{0})\| - C \epsilon^{2}
	\end{align*}
for some $C > 0$ depending only on $\varphi$ and $\mu$.  By \eqref{E: tedious geometric thing 1} and \eqref{E: tedious geometric thing 2}, we can apply Lemma \ref{lemma:reduced-planar-growth} (with $a = \epsilon^{2}$ and $b = \mu \epsilon$) provided that $\epsilon < \mu$.  Applying the lemma leads us to deduce that
	\begin{align*}
		u_{*}(x_{0}) &\geq \min \biggl\{ u_{*}(x_{0}) + \epsilon \zeta(q) \|D\psi(x_{0})\| - C \epsilon^{2}, \\
		&\qquad \qquad u_{*}(x_{0}) - \epsilon^{2} \|D\psi(x_{0})\| - C \epsilon^{2} + \sqrt{2\mu} \rho(f(x_{0})) \epsilon^{\frac{3}{2}} \biggr\}.
	\end{align*}
As before, this inequality leads to an absurd conclusion in the limit $\epsilon \to 0^{+}$. \end{proof}

The question now is simply how to prove Lemma \ref{lemma:reduced-planar-growth}.  Much of the challenge results from the fact that $\tilde{\mathcal{E}}$ can be uncountable, hence, to avoid measurability issues, we need to be careful to develop constructions that treat multiple directions simultaneously.  Toward that end, we will use an intermediate approximation result.  Before stating it, we need to state a geometric fact used in the proof.

	\begin{prop} \label{prop:special-pareto} There are functions $\zeta,\bar{R} : \tilde{\mathcal{E}} \to (0,\infty)$ with the following property: for any $q \in \tilde{\mathcal{E}}$, if $q' \in B(q,\bar{R}(q)) \cap \tilde{\mathcal{E}}$ and $x^{+,+}, x^{-,+}, x^{+,-}, x^{-,-} \in \mathbb{R}^{2}$ are such that
		\begin{gather}
			x^{+,+} \in A_{\zeta(q)}(q) \cap \inte(Q_{+,+}(q')), \quad x^{-,+} \in A_{\zeta(q)}(q) \cap \inte(Q_{-,+}(q')), \label{E: special-pareto-1}\\
			x^{+,-} \in \inte(Q_{+,-}(q')), \quad x^{-,-} \in \inte(Q_{-,-}(q')), \label{E: special-pareto-2}
		\end{gather}
	then
		\begin{equation*}
			0 \in \inte(\mathcal{P}(\{x^{+,+},x^{-,+},x^{+,-},x^{-,-}\})).
		\end{equation*}
	\end{prop}  	
	
\begin{figure}
	\begin{tikzpicture}[scale=1.9]

\draw[color=blue!80, fill=blue!10] (0,0) -- (1.9696,-0.3473) arc [start angle = -10, end angle = 80, radius = 2] -- (0,0);

\draw[color=blue!80, fill=blue!10] (0,0) -- (0.3472,1.9696) arc[start angle = 80, end angle = 90, radius = 2] -- (0,0);

\draw[color=blue!80, fill=blue!10] (0,0) -- (-1.7321,-1) arc[start angle = 210, end angle = 260, radius = 2] -- (0,0);

\draw[color=blue!80, fill=blue!10] (0,0) -- (-0.34721,-1.9696) arc[start angle = 260, end angle = 350, radius = 2] -- (0,0);

\draw[thick,dotted,->] (0,0) -- (0,2);
\draw[thick,dotted,->] (0,0) -- (0,-2); 
\draw[thick,->] (0,0) -- (-2,0);
\draw[thick,->] (0,0) -- (2,0);

\draw[anchor = west] (2,0) node {$w_{p}$}; 
\draw[anchor = east] (0,2) node {$v_{p}$};
\draw[anchor = east] (-2,0) node {$-w_{p}$}; 
\draw[anchor = east] (0,-2.1) node {$-v_{p}$};

\draw[dashed] (-1,0) -- (0,1);
\draw[dashed] (0,-1) -- (1,0);

\draw[dashed] (0,1) arc[start angle=90, end angle=0, radius=1];
\draw[dashed] (0,-1) arc[start angle=-90, end angle=-180, radius=1];

\draw[anchor = east] (1.1,1.1) node {$Q_{p}$};
\draw[anchor = east] (-1.1, -1.1) node {$-Q_{p}$};
\draw[anchor = east] (1.1,-1.1) node {$-\mathcal{Q}_{p}$};
\draw[anchor = east] (-1.1,1.1) node {$\mathcal{Q}_{p}$};

\draw[color=blue, anchor = east] (1.9,0.3) node {$Q_{+,-}(q')$}; 
\draw[color=blue, anchor = east] (-0.05, 1.6) node {$Q_{+,+}(q') \cap A_{\zeta}(q)$};
\draw[color=blue, anchor = east] (1.3, -0.8) node {$Q_{-,+}(q')$};
\draw[color=blue, anchor = east] (-0.3,-1.4) node {$Q_{-,+}(q') \cap A_{\zeta}(q)$};

\filldraw[black] (0.5,0.8660) circle (0.04);
\draw[anchor = east] (0.5,0.9) node {$q$};
\draw[anchor = north] (0.1,0) node {$0$};
\filldraw[black] (0,0) circle (0.04);

\filldraw[black] (0.1736,0.9848) circle (0.04);
\filldraw[anchor = south] (0.1200,1) node {$q'$};

\draw[thick,dotted,->] (0,0) -- (1.7321,1);
\draw[thick,dotted,->] (0,0) -- (-1.7321,-1);

\draw[thick,blue,->] (0,0) -- (0.3473,1.9696);
\draw[thick,blue,->] (0,0) -- (-0.3473,-1.9696);
\draw[thick,blue,->] (0,0) -- (1.9696,-0.3473);
\draw[thick,blue,->] (0,0) -- (-1.9696,0.3473);


\filldraw[red] (0.1, 1.8) circle (0.04);
\filldraw[red] (1.55, 0.6) circle (0.04);
\filldraw[red] (1.3,-0.5) circle (0.04);
\filldraw[red] (-0.6, -0.7) circle (0.04);

\end{tikzpicture}
	\caption{A depiction of the situation in Proposition \ref{prop:special-pareto} when $\varphi(x) = \max\{|x_{1} - x_{2}|,\|x\|_{2}\}$.  The dotted lines delineate the cone $A_{\zeta}(q)$.  Blue regions are $Q_{+,+}(q') \cap A_{\zeta}(q)$, $Q_{-,+}(q') \cap A_{\zeta}(q)$, $Q_{+,-}(q')$, and $Q_{-,+}(q')$.  The Pareto hull of the red points contains $0$ in its interior.}
	\label{fig:special-pareto}
\end{figure}
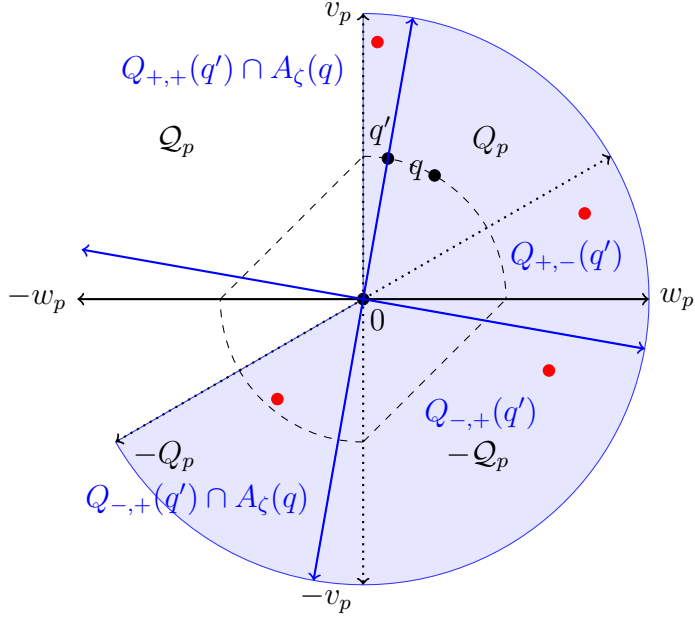
	
Figure \ref{fig:special-pareto} depicts the situation described in Proposition \ref{prop:special-pareto}.  In the figure, the point $q$ is $q = (\cos(\pi/3),\sin(\pi/3))$ and $\zeta = \zeta(q) = \tan(\pi/6)$.  In this particular example, for any $q' \in \inte(A_{\zeta}(q)) \cap Q_{p}$, any set of four points $\{x^{\pm,\mp}\}$ satisfying \eqref{E: special-pareto-1} and \eqref{E: special-pareto-2} will have zero in the interior of its Pareto hull.  In particular, any choice of $\bar{R}(q) > 0$ such that $B(q,\bar{R}(q)) \subseteq \inte(A_{\zeta}(q)) \cap Q_{p}$ suffices. 

Proposition \ref{prop:special-pareto} allows us to prove the next approximation result.

\begin{lemma} \label{lemma:special-planar-growth}
	Let $\zeta, \bar{R} : \tilde{\mathcal{E}} \to (0,\infty)$ be the functions from Proposition \ref{prop:special-pareto}.  There is a function $\rho:(0,\infty) \to (0,\infty)$ such that, on an event of probability 1, if $q \in \tilde{\mathcal{E}}$, $q' \in \mathcal{D}$ (given in \eqref{eq:dense-set-approximation}), $x_{0} \in \mathbb{Q}^{2}$, and $a,b \in \mathbb{Q} \cap (0,\infty)$ are chosen so that
\[
	x_{0} + D^{q'}_{-}[a,b] \subseteq \{f > f(x_{0})/2\}, \quad q' \in B(q,\bar{R}(q));
\]
$x^{+,+}, x^{-,+} \in \mathbb{R}^{2}$ are points such that
	\begin{gather*}
	x_{0} + D^{q'}_{-}[a,b] \subseteq x^{-,+} - A_{\zeta(q)}(q)  \cap \inte(Q_{-,+}(q')), \\
	 x_{0} + D^{q'}_{-}[a,b] \subseteq x^{+,+} - A_{\zeta(q)}(q) \cap \inte(Q_{+,+}(q'));
	\end{gather*}
	and
\[
	s_{n} := \min \left\{ u_n(z) \, \mid \, z \in x_{0} + D^{q'}_{-}[a,b] \right\},
\]
		then for all $n$ sufficiently large, 
	\begin{equation} \label{eq:growth-estimate-result2}
	u_n(x_{0}) \geq \min \left\{u_n(x^{+,+}), u_n(x^{-,+}), s_{n} + \rho(f(x_0)) \sqrt{ 2a b n} \right\}.
	\end{equation}
\end{lemma}  

We will now show how to prove Lemma \ref{lemma:reduced-planar-growth} using Lemma \ref{lemma:special-planar-growth} and Proposition \ref{prop:special-pareto}.  Lemma \ref{lemma:special-planar-growth} will be proved immediately afterward, while the next two sections are devoted to the proof of Proposition \ref{prop:special-pareto}.

	\begin{proof}[Proof of Lemma \ref{lemma:reduced-planar-growth}]  By definition of $u_{*}$, we can fix a (random) sequence $(x^{(n)})_{n \in \mathbb{N}} \subseteq U$ such that 
		\begin{align*}
			\lim_{n \to \infty} x^{(n)} = x_{0}, \quad \lim_{n \to \infty} n^{-\frac{1}{2}} u_{n}(x^{(n)}) &= u_{*}(x_{0}).
		\end{align*}
	
	Given $\nu > 0$ small, choose points $x^{+,-}_{\nu}, x^{-,-}_{\nu} \in \mathbb{Q}^{2} \cap U$ such that
		\begin{equation*}
			x^{\pm,-}_{\nu} \in x_{0} + \inte(Q_{\pm,-}(q)) \quad \text{and} \quad x^{\pm,-}_{\nu} + D^{q}_{-}[a-\nu,b-\nu] \subseteq x_{0} + \inte(D_{-}^{q}[a,b]).
		\end{equation*}
	Since $\mathcal{D}$ is dense in $\tilde{\mathcal{E}}$, we can choose $q_{\nu} \in \mathcal{D} \cap B(q,\bar{R}(q))$ such that
		\begin{align*}
			x^{\pm}_{\nu} + D^{q_{\nu}}_{-}[a-\nu,b-\nu] &\subseteq x_{0} + \inte(D^{q}_{-}[a,b]), \\
			x_{0} + D^{q}_{-}[a,b] &\subseteq x^{\pm,+} - \inte(Q_{\pm,+}(q_{\nu})).
		\end{align*}
	
	We are now in a position to apply Lemma \ref{lemma:special-planar-growth}.  In particular, for all $n$ sufficiently large, we have
		\begin{equation*}
			u_{n}(x^{\pm,-}_{\nu}) \geq \min \left\{ u_{n}(x^{+,+}), u_{n}(x^{-,+}), s_{n} + \rho(f(x_{0})) \sqrt{2 (a - \nu)(b - \nu)n} \right\},
		\end{equation*}
	where $s_{n}$ is given by 
		\begin{equation*}
			s_{n} := \min\left\{ u_{n}(z) \, \mid \, z \in x_{0} + D^{q}_{-}[a,b]\right\}.
		\end{equation*}

	At the same time, if $n$ is large enough, then
		\begin{gather*}
			x^{+,-}_{\nu} \in x_{n} + \inte(Q_{+,-}(q_{\nu})), \quad x^{-,-}_{\nu} \in x_{n} + \inte(Q_{-,-}(q_{\nu})), \\
			x^{+,+} \in x_{n} + \inte(Q_{+,+}(q_{\nu})) \cap A_{\zeta(q)}(q), \quad x^{-,+} \in x_{n} + \inte(Q_{-,+}(q_{\nu})) \cap A_{\zeta(q)}(q),
		\end{gather*}
	and, therefore, by Proposition \ref{prop:special-pareto} and the dynamic programming principle,
		\begin{align*}
			u_{n}(x_{n}) &\geq \min \left\{  u_{n}(x^{+,+}), u_{n}(x^{-,+}), u_{n}(x^{+,-}_{\nu}), u_{n}(x^{-,-}_{\nu}) \right\} \\
					&\geq \min\left\{ u_{n}(x^{+,+}), u_{n}(x^{-,+}), s_{n} + \rho(f(x_{0})) \sqrt{2 (a - \nu)(b - \nu)n} \right\}.
		\end{align*}
		
	Finally, renormalizing by $n^{\frac{1}{2}}$ and sending first $n \to \infty$ and then $\nu \to 0^{+}$, we obtain
		\begin{equation*}
			u_{*}(x_{0}) \geq \min \left\{ u_{*}(x^{+,+}), u_{*}(x^{-,+}), s + \rho(f(x_{0})) \sqrt{2a b} \right\}.
		\end{equation*}
	\end{proof}  
	
\begin{proof}[Proof of Lemma \ref{lemma:special-planar-growth}]  The argument is similar in spirit to that of Lemma \ref{lemma:unidirectional-growth} so we only sketch the proof.  As in that lemma,
	on an event of probability one, for all $n$ large enough, there is a random integer $L(n) \geq \rho(f(x_{0})) \sqrt{2abn}$ and a random subset $\{x_{1},\dots,x_{2L}\} \subseteq X_{n} \cap (x_{0} + D^{q'}_{-}[a,b])$ such that 
		\begin{gather*}
			x_{0} \in (x_{1} - \inte(Q_{-,-}(q'))) \cap (x_{2} - \inte(Q_{+,-}(q')))
		\end{gather*}
and, for each $i \in \{1,2,\dots,L(n)-1\}$,
		\begin{gather*}
			\{x_{2i-1},x_{2i}\} \subseteq (x_{2i + 1} - \inte(Q_{-,-}(q'))) \cap (x_{2(i + 1)} - \inte(Q_{+,-}(q'))).
		\end{gather*}
	(Compared to Lemma \ref{lemma:unidirectional-growth}, at this stage, all that is different is $[0,\infty)^{2}$ is replaced by $Q_{-,+}(q')$; the geometrical picture differs only by a rotation.)
	
	In order to invoke the dynamic programming principle, we utilize Proposition \ref{prop:special-pareto}.  Indeed, by assumption, for each $i \in \{1,2,\dots,L(n)\}$, we have
		\begin{align*}
			x^{-,+} &\in (x_{2i-1} + A_{\zeta(q)}(q) \cap \inte(Q_{-,+}(q'))) \cap (x_{2i} + A_{\zeta(q)}(q) \cap \inte(Q_{-,+}(q'))), \\
			x^{+,+} &\in (x_{2i - 1} + A_{\zeta(q)}(q) \cap \inte(Q_{+,+}(q'))) \cap (x_{2i} + A_{\zeta(q)}(q) \cap \inte(Q_{+,+}(q'))).
		\end{align*}
	Thus, the proposition implies that
		\begin{equation*}
			x_{2i-1},x_{2i} \in \inte(\mathcal{P}(\{x^{+,+},x^{-,+},x_{2i+1},x_{2(i+1)}\})),
		\end{equation*}
	which, in terms of the height function, reads
		\begin{equation} \label{eq:pre-iterate}
			\min(u_{n}(x_{2i-1}), u_{n}(x_{2i})) \geq \min (u_{n}(x^{+,+}),u_{n}(x^{-,+}), u_{n}(x_{2i+1}) + 1, u_{n}(x_{2(i+1)}) + 1).
		\end{equation}
	
	We conclude by iterating the bounds in \eqref{eq:pre-iterate}. \end{proof}

\subsection{Manipulations of Cones}  So far, we have proved Lemma \ref{lemma:degenerate-directions-dual-2} conditional on Proposition \ref{prop:special-pareto}.  This last proposition follows from a series of geometrical observations.  The main observations concern properties of the cones $\{Q_{p}\}_{p \in \mathcal{N}^{*}}$, which are detailed in this section.  The next section explains the remainder of the proof of Proposition \ref{prop:special-pareto}.

In what follows, given a $q \in \{\varphi = 1\}$, we denote by $N(q)$ the set
	\begin{equation*}
		N(q) = \{p \in \mathcal{N}^{*} \, \mid \, q^{\perp} \in \inte(\mathcal{Q}_{p})\}.
	\end{equation*}

The following basic observations concerning directions in $\tilde{\mathcal{E}}$ (see \eqref{eq:tilde-E}) will be fundamental in what follows.  In fact, the second observation provides half of the proof of Proposition \ref{prop:special-pareto}.
	
	\begin{prop} \label{P: basic generality} For each $q \in \{\varphi = 1\}$, we have $\#N(q) \leq 1$.  Further, for $q \in \tilde{\mathcal{E}}$, if $p \in \mathcal{N}^{*}$ is such that $\{p,-p\} \cap N(q) = \emptyset$, then there is a $\bullet \in \{(+,+),(-,+),(+,-),(-,-)\}$ such that 
		\begin{equation*}
			\mathcal{Q}_{p} \subseteq Q_{\bullet}(q),
		\end{equation*}
		where $Q_{\bullet}$ is as in \eqref{eq:plus-minus-cones}.
	\end{prop}  
	
		\begin{proof}  To start with, recall that $\{\inte(\mathcal{Q}_{p})\}_{p \in \mathcal{N}^{*}}$ is pairwise disjoint.  Hence there is at most one element in $N(q)$.  
		
		Next, suppose that $p \in \mathcal{N}^{*}$ and $\{-p,p\} \cap N(q) = \emptyset$.  Since $q \in \tilde{\mathcal{E}}$, the results of Section \ref{sec:duality} imply that $\langle v_{p}, q^{\perp} \rangle$ and $\langle (-w_{p}), q^{\perp} \rangle$ have the same sign, that is,
			\begin{equation*}
				\langle v_{p},q^{\perp} \rangle \langle (-w_{p}), q^{\perp} \rangle \geq 0.
			\end{equation*}
		At the same time, from the fact that $\{q^{\perp},-q^{\perp}\} \cap \inte(\mathcal{Q}_{p}) = \emptyset$ and $\mathcal{Q}_{p} = \cone(\{v_{p},-w_{p}\})$, we similarly deduce that $\langle v_{p}, q \rangle$ and $\langle (-w_{p}), q \rangle$ have the same sign.  It follows that $\mathcal{Q}_{p}$ is contained in one of the quadrants determined by the basis $\{q,q^{\perp}\}$.  These quadrants are exactly $Q_{+,+}(q)$, $Q_{-,+}(q)$, $Q_{+,-}(q)$, and $Q_{-,-}(q)$.    \end{proof}
		
In the next result, we observe that for $q', q \in \tilde{\mathcal{E}}$ sufficiently close together, the cone determined by $N(q')$ necessarily misses a conic neighborhood of the line through the origin determined by $q$, or, more precisely, $A_{\epsilon}(q)$ for some small enough $\epsilon$.  We will see in the next section that this observation accounts for one half of the proof of Proposition \ref{prop:special-pareto}.

	\begin{prop} \label{prop:epsilon-choice} Given $q \in \tilde{\mathcal{E}}$, there are constants $\bar{R}_{1}(q), \zeta(q) > 0$ such that if $q' \in B(q,\bar{R}_{1}(q)) \cap \{\varphi = 1\}$ and $p' \in N(q')$, then 
		\begin{equation*}
			\mathcal{Q}_{p'} \subseteq \{v \in \mathbb{R}^{2} \, \mid \, \langle v, q^{\perp} \rangle \geq 0\} \quad \text{and} \quad \mathcal{Q}_{p'} \cap \bar{A}_{\zeta(q)}(q) = \{0\}.
		\end{equation*}
	\end{prop}
	
		\begin{proof}  In view of Proposition \ref{P: basic generality}, there are two cases to consider: (i) $N(q) = \{p\}$ for some $p \in \mathcal{N}^{*}$ and (ii) $N(q) = \emptyset$. 
		
		\textbf{Case (i): $N(q) = \{p\}$ for some $p \in \mathcal{N}^{*}$.}
		
		In this case, we first claim that there is an $\bar{R}_{1}(q) > 0$ such that $N(q') = \{p\}$ for each $q' \in B(q,\bar{R}_{1}(q)) \cap \{\varphi = 1\}$.  Put slightly differently, there is a choice of $\bar{R}_{1}(q)$ for which the identity $p' = p$ necessarily holds.
		
		To see this, recall that $q^{\perp} \in \inte(\mathcal{Q}_{p})$ by definition of $N(q)$.  Hence we can fix $\bar{R}_{1}(q) > 0$ such that $B(q^{\perp},\bar{R}_{1}(q)) \subseteq \inte(\mathcal{Q}_{p})$.  Since $N(q')$ contains at most one element, we conclude that $N(q') = N(q)$ for each $q' \in B(q,\bar{R}_{1}(q)) \cap \{\varphi = 1\}$.
		
	In view of what was just proved, we only need to find a $\zeta(q) > 0$ such that the desired inclusions hold with $\mathcal{Q}_{p'} = \mathcal{Q}_{p}$.  Toward that end, we know that $q^{\perp} \in \inte(\mathcal{Q}_{p})$, but $\{q,-q\} \cap \mathcal{Q}_{p} = \emptyset$ since $q \in \tilde{\mathcal{E}}$.   This readily implies that 
		\begin{equation} \label{eq:use-it-a-little-later}
			\mathcal{Q}_{p} \setminus \{0\} \subseteq \inte(Q_{+,+}(q) \cup Q_{-,+}(q)).
		\end{equation}  
From this, we deduce that if we define $\zeta(q)$ by
		\begin{equation*}
			\zeta(q) = \frac{1}{2} \min \left\{ \frac{\langle v_{p}, q^{\perp} \rangle}{|\langle v_{p},q\rangle|}, \frac{\langle (-w_{p}),q^{\perp} \rangle}{|\langle w_{p}, q \rangle|} \right\} > 0,
		\end{equation*}
	then $\mathcal{Q}_{p} \cap \bar{A}_{\zeta(q)} = \{0\}$.  
	
	Finally, observe that $\{v \in \mathbb{R}^{2} \, \mid \, \langle v, q^{\perp} \rangle \geq 0\} = Q_{+,+}(q) \cup Q_{-,+}(q)$.  Combining this with \eqref{eq:use-it-a-little-later}, we conclude $\mathcal{Q}_{p} \subseteq \{v \in \mathbb{R}^{2} \, \mid \, \langle v, q^{\perp} \rangle \geq 0\}$.
	
	\textbf{Case (ii): $N(q) = \emptyset$.} 
	
	We start by proving that $\mathcal{Q}_{p'} \cap \bar{A}_{\zeta(q)}(q) = \{0\}$ provided $\zeta(q)$ and $\bar{R}_{1}(q)$ are small enough.  Here it is simplest to argue by contradiction.  Suppose that we can find a sequence $(q_{n})_{n \in \mathbb{N}} \subseteq \{\varphi = 1\}$ converging to $q$ and a sequence $(v_{n})_{n \in \mathbb{N}} \subseteq S^{1}$ such that 
		\begin{equation*}
			v_{n} \in \mathcal{Q}_{p_{n}} \cap A_{n^{-1}}(q), \, \, N(q_{n}) = \{p_{n}\} \quad \text{for each} \, \, n \in \mathbb{N}.
		\end{equation*}
	Since $v_{n} \in S^{1} \cap A_{n^{-1}}(q)$ for each $n$, the accumulation points of $(v_{n})_{n \in \mathbb{N}}$ are contained in $\{-\|q\|^{-1}q,\|q\|^{-1}q\}$.  Let us assume, passing to a subsequence if necessary, that $\lim_{n \to \infty} v_{n} = \|q\|^{-1} q$; the case where the limit equals $-\|q\|^{-1} q$ can be treated similarly.

	Recall, by definition of $N(q_{n})$, that $q_{n}^{\perp} \in \mathcal{Q}_{p_{n}}$ for each $n$.  At the same time, since $q_{n}^{\perp} \to q^{\perp}$ and $v_{n} \to \|q\|^{-1} q$, we know that
		\begin{equation*}
			\lim_{n \to \infty} |v_{n} \times q_{n}^{\perp}| = \|q^{\perp}\| > 0.
		\end{equation*}
	From this, if we define $\mathcal{N}^{*}(\|q^{\perp}\|/2)$ as in \eqref{eq:finite-set}, then $(p_{n})_{n \geq N} \subseteq \mathcal{N}^{*}(\|q^{\perp}\|/2)$ for some $N \in \mathbb{N}$.  Thus, since $\mathcal{N}^{*}(\|q^{\perp}\|/2)$ is finite (see the proof of Proposition \ref{prop:hj-finite-set}), we deduce that the set $\{p_{n} \, \mid \, n \in \mathbb{N}\}$ is finite.  In particular, up to passing to a subsequence, we can assume without loss of generality that $p_{n} = p_{N}$ for all $n \geq N$.  
	
	This gives the desired contradiction.  Indeed, we know that $v_{n} \in \mathcal{Q}_{p_{N}}$ for all $n \in \mathbb{N}$ and $\mathcal{Q}_{p_{N}}$ is closed.  Hence $\|q\|^{-1} q = \lim_{n \to \infty} v_{n} \in \mathcal{Q}_{p_{N}}$, but then this contradicts the fact that $q \in \tilde{\mathcal{E}}$.  
	
	It remains to show that $\mathcal{Q}_{p'} \subseteq \{v \in \mathbb{R}^{2} \, \mid \, \langle v, q^{\perp} \rangle \geq 0\}$ provided $\bar{R}_{1}(q)$ is small enough.  Once again, this follows readily from contradiction: if it were not true, we could find a sequence $(q_{n})_{n \in \mathbb{N}} \subseteq \{\varphi = 1\}$ converging to $q$ and vectors $(v_{n})_{n \in \mathbb{N}} \subseteq S^{1}$ such that 
		\begin{equation*}
			\langle v_{n},q^{\perp} \rangle \leq 0, \quad v_{n} \in \mathcal{Q}_{p_{n}}, \quad N(q_{n}) = \{p_{n}\}.
		\end{equation*}
	Restricting to large $n$ if necessary, we can assume that $\langle q_{n}^{\perp}, q^{\perp} \rangle > 0$ for all $n \in \mathbb{N}$.  Hence, by continuity, there is a $t_{n} \in [0,1]$ such that $\langle (1 - t_{n}) v_{n} + t_{n} q_{n}^{\perp}, q^{\perp} \rangle = 0$.  Yet $\{v_{n},q_{n}^{\perp} \} \subseteq \mathcal{Q}_{p_{n}}$ so, by convexity, $(1 - t_{n})v_{n} + t_{n} q_{n}^{\perp} \in \mathcal{Q}_{p_{n}}$.  Renormalizing by the length, this implies that either $\|q\|^{-1} q \in \mathcal{Q}_{p_{n}}$ or $-\|q\|^{-1} q \in \mathcal{Q}_{p_{n}}$, which contradicts our assumption that $q \in \tilde{\mathcal{E}}$ in any case. \end{proof}  
	
The previous proposition showed that if $q'$ is close enough to $q$ and $N(q')$ is non-empty, then we can conclude that the cone $\mathcal{Q}_{p}$ avoids $A_{\epsilon}(q)$ for some small $\epsilon > 0$.  The next result is a more-or-less straightforward observation about such cones.
	
\begin{prop} \label{P: convexity thing} Fix $q \in \{\varphi = 1\}$, $p \in \mathcal{N}^{*}$, and $\zeta > 0$.  If $\mathcal{Q}_{p}$ satisfies
	\begin{equation*}
		\mathcal{Q}_{p} \subseteq \{v \in \mathbb{R}^{2} \, \mid \, \langle v, q^{\perp} \rangle \geq 0\} \quad \text{and} \quad \mathcal{Q}_{p} \cap \bar{A}_{\zeta}(q) = \{0\},
	\end{equation*}
then 
	\begin{equation*}
		A_{\zeta}(q) \cap \{v \in \mathbb{R}^{2} \, \mid \, \langle v, q \rangle \geq 0\} \subseteq Q_{p} \quad \text{and} \quad A_{\zeta}(q) \cap \{v \in \mathbb{R}^{2} \, \mid \, \langle v, q \rangle \leq 0\} \subseteq -Q_{p}.
	\end{equation*}
\end{prop}

\begin{figure}
	\begin{tikzpicture}[scale=1.2]

\draw[thick,->] (0,0) -- (0,3);
\draw[thick,->] (0,0) -- (0,-3); 
\draw[thick,->] (0,0) -- (-4,0);
\draw[thick,->] (0,0) -- (4,0);

\draw[anchor=east] (0,3) node {$q^{\perp}$};
\draw[anchor=south] (4,0) node {$q$};

\draw[dashed] (0,0) -- (4,0.75);
\draw[dashed] (0,0) -- (4,-0.75);
\draw[dashed] (0,0) -- (-4,0.75);
\draw[dashed] (0,0) -- (-4,-0.75);

\draw[anchor=south] (3,0) node {$A_{\zeta}(q)$};

\draw[thin,->] (0,0) -- (1,3);
\draw[thin,->] (0,0) -- (-1,-3);
\draw[thin,->] (0,0) -- (-1.5,3);
\draw[thin,->] (0,0) -- (1.5,-3);

\draw[anchor = east] (0,2.25) node {$\mathcal{Q}_{p}$};
\draw[anchor = west] (1.5,-3) node {$w_{p}$};
\draw[anchor = west] (1,3) node {$v_{p}$};
\draw[anchor = west] (2, 2) node {$Q_{p}$};
\draw[anchor = north] (-2,-2) node {$-Q_{p}$};

\end{tikzpicture}
	\caption{Proof of Proposition \ref{P: convexity thing}.
	} \label{fig:cone-proof}
\end{figure}
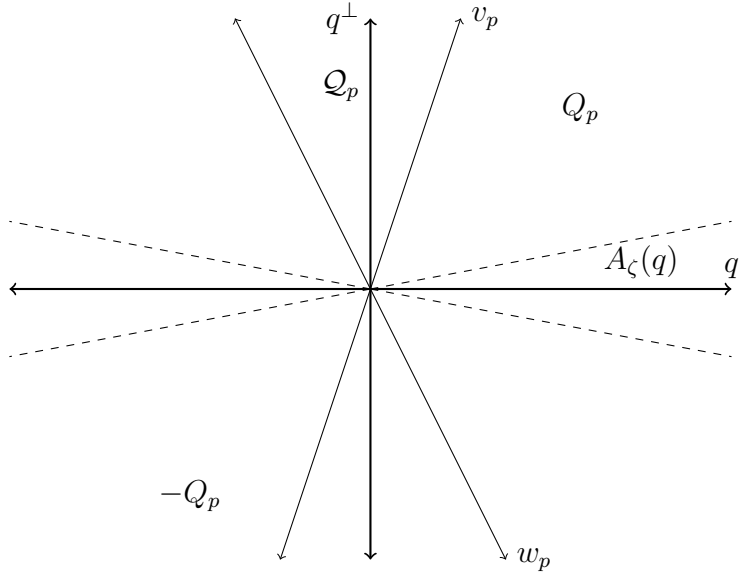

	See Figure \ref{fig:cone-proof} for a ``proof by picture."  The interested reader is invited to work out the details of a rigorous proof on their own.
	
\subsection{Proof of Proposition \ref{prop:special-pareto}}  In this section, we prove Proposition \ref{prop:special-pareto} using the observations made in the previous section and a few additional lemmas.  

	The first lemma allows us to easily relate the hypothesis \eqref{E: special-pareto-1} to the cones appearing in Proposition \ref{P: convexity thing}.
	
		\begin{lemma} \label{lemma:limiting-direction-part}  Fix $q \in \tilde{\mathcal{E}}$ and let $\zeta(q) > 0$ be the constant from Proposition \ref{prop:epsilon-choice}.  There is a $\bar{R}_{2}(q) > 0$ such that if $q' \in B(q,\bar{R}_{2}(q)) \cap \{\varphi = 1\}$, then 
			\begin{align*}
				Q_{+,+}(q') \cap A_{\zeta(q)}(q) \subseteq \{v \in \mathbb{R}^{2} \, \mid \, \langle v, q \rangle \geq 0\}, \\
				Q_{-,+}(q') \cap A_{\zeta(q)}(q) \subseteq \{v \in \mathbb{R}^{2} \, \mid \, \langle v, q \rangle \leq 0\}.
			\end{align*}
		\end{lemma}  
		
			\begin{proof}  We only prove the inclusion involving $Q_{+,+}(q')$ since the other one follows by similar arguments.  
			
			We argue by contradiction.  If the claim were false, we could fix a sequence $(q_{n})_{n \in \mathbb{N}} \subseteq \{\varphi = 1\}$ converging to $q$ and a sequence $(v_{n})_{n \in \mathbb{N}} \subseteq S^{1}$ such that, for each $n \in \mathbb{N}$,
				\begin{equation*}
					v_{n} \in Q_{+,+}(q_{n}) \cap A_{\zeta(q)}(q), \quad \langle v_{n}, q \rangle \leq 0.
				\end{equation*}
			Since $S^{1}$ is compact, we lose no generality assuming that the limit $v_{*} = \lim_{n \to \infty} v_{n}$ exists.  Now the vector $v_{*}$ is an element of $Q_{+,+}(q) \cap \bar{A}_{\zeta(q)}(q)$ and, thus,
				\begin{equation*}
					0 \leq \langle v_{*}, q^{\perp} \rangle \leq \zeta(q) \langle v_{*}, q \rangle = \zeta(q) \cdot \lim_{n \to \infty} \langle v_{n}, q \rangle \leq 0.
				\end{equation*}
	We deduce that $0 = \langle v_{*}, q^{\perp} \rangle = \langle v_{*}, q \rangle$, hence $v_{*} = 0$, contradicting the fact that $\|v_{*}\| = 1$.
				\end{proof}  
		
	The final lemma is a fundamental observation that explains the role of the cones $Q_{+,+}(q')$, $Q_{-,+}(q')$, $Q_{+,-}(q')$, and $Q_{-,-}(q')$ in Proposition \ref{prop:special-pareto}.
				
				\begin{lemma} \label{lemma:convex-thing} If $q \in \mathbb{R}^{2} \setminus \{0\}$ and there are points $y^{+,+},y^{-,+},y^{+,-},y^{-,-} \in \R^{2}$ such that
	\begin{equation*}
		y^{\bullet} \in \inte(Q_{\bullet}(q)) \quad \text{for each} \, \, \bullet \in \{(+,+),(-,+),(+,-),(-,-)\},
	\end{equation*}
then
	\begin{equation*}
		0 \in \inte(\conv(y^{+,+},y^{-,+},y^{+,-},y^{-,-})).
	\end{equation*}
\end{lemma}   

	\begin{proof}  If we define $\varphi_{q}$ by
		\[
		\varphi_{q}(q') = |\langle q',q \rangle| + |\langle q',q^{\perp}\rangle|,
		\]
		then the second statement in Theorem \ref{thm:cone_efficient} implies that $0 \in \inte(\mathcal{P}_{\varphi_{q}}(y^{+,+},y^{-,+},y^{+,-},y^{-,-}))$.  Therefore, by Corollary \ref{cor:-constrained-convex-hull}, $0$ is necessarily in the interior of the convex hull of those points.  \end{proof}

		\begin{proof}[Proof of Proposition \ref{prop:special-pareto}]  Let $\bar{R}_{1}(q), \zeta(q) > 0$ and $\bar{R}_{2}(q) > 0$ be the constants from Proposition \ref{prop:epsilon-choice} and Lemma \ref{lemma:limiting-direction-part}, respectively.  Define $\bar{R}(q) = \min(\bar{R}_{1}(q),\bar{R}_{2}(q))$.
		
		With this choice of $\zeta(q)$ and $\bar{R}(q)$, suppose that $q' \in B(q,\bar{R}(q)) \cap \{\varphi = 1\}$ and the points $\{x^{+,+},x^{-,+},x^{+,-},x^{-,-}\}$ satisfy the hypotheses \eqref{E: special-pareto-1} and \eqref{E: special-pareto-2}.  We need to prove that $0 \in \inte(\mathcal{P}(\{x^{+,+},x^{-,+},x^{+,-},x^{-,-}\}))$.  Recall from Corollary \ref{cor:-constrained-convex-hull} that it suffices to establish that
			\begin{gather}
				\inte(Q_{p}) \cap \{x^{+,+},x^{-,+},x^{+,-},x^{-,-}\} \neq \emptyset \quad \text{for each} \, \, p \in \mathcal{N}^{*}, \label{E: cone inclusion} \\
				0 \in \inte(\conv(\{x^{+,+},x^{-,+},x^{+,-},x^{-,-}\})). \label{E: convex part}
			\end{gather}
		Notice that \eqref{E: convex part} follows from \eqref{E: special-pareto-1} and \eqref{E: special-pareto-2} after an immediate application of Lemma \ref{lemma:convex-thing}.  It only remains to verify \eqref{E: cone inclusion}.
		
		Suppose that $p \in \mathcal{N}^{*}$.  We consider cases.
		
		\textbf{Case 1: $p \in N(q')$ or $-p \in N(q')$}
		
		We want to prove that \eqref{E: cone inclusion} holds.  We will assume that $p \in N(q')$ and go on to show that 
			\begin{equation*}
				\inte(Q_{p}) \cap \{x^{+,+},x^{-,+},x^{+,-},x^{-,-}\} \neq \emptyset \quad \text{and} \quad \inte(Q_{-p}) \cap \{x^{+,+},x^{-,+},x^{+,-},x^{-,-}\} \neq \emptyset.
			\end{equation*}
		If instead $-p \in N(q')$, the desired conclusion follows from consideration of $-p$.
		
		By assumption, the hypotheses of Proposition \ref{prop:epsilon-choice} hold.  Hence we can invoke Proposition \ref{P: convexity thing} to find that
			\begin{equation*}
				A_{\zeta(q)}(q) \cap \{v \in \mathbb{R}^{2} \, \mid \, \langle v, q \rangle \geq 0\} \subseteq Q_{p}, \quad A_{\zeta(q)}(q) \cap \{v \in \mathbb{R}^{2} \, \mid \, \langle v, q \rangle \leq 0 \} \subseteq Q_{-p}.
			\end{equation*}
		At the same time, Lemma \ref{lemma:limiting-direction-part} implies that
			\begin{align*}
				A_{\zeta(q)}(q) \cap Q_{+,+}(q') &\subseteq \{v \in \mathbb{R}^{2} \, \mid \, \langle v, q \rangle \geq 0\}, \\
				A_{\zeta(q)}(q) \cap Q_{-,+}(q') &\subseteq \{v \in \mathbb{R}^{2} \, \mid \, \langle v, q \rangle \leq 0\}.
			\end{align*}
		Therefore, invoking \eqref{E: special-pareto-1}, we conclude that
			\begin{equation*}
				x^{+,+} \in \inte(Q_{p}), \quad x^{-,+} \in \inte(-Q_{p}) = \inte(Q_{-p}).
			\end{equation*}
		
		\textbf{Case 2: $\{p,-p\} \cap N(q') = \emptyset$}
		
		By Proposition \ref{P: basic generality}, we know that there is a $\bullet \in \{(+,+),(-,+),(+,-),(-,-)\}$ such that 
			\begin{equation*}
				\mathcal{Q}_{p} \subseteq Q_{\bullet}(q').
			\end{equation*}
		Recall that $-\mathcal{Q}_{p} = \mathcal{Q}_{-p}$ (see \eqref{eq:negation-thing}).  Thus, as in the previous step, there is no loss of generality in assuming $\bullet \in \{(+,+),(-,+)\}$.
		
		If $\mathcal{Q}_{p} \subseteq Q_{+,+}(q')$, then $\{-w_{p},v_{p}\} \subseteq Q_{+,+}(q')$.  Hence, recalling \eqref{eq:planar_thing}, we deduce that
			\begin{equation*}
				\langle (q')^{\perp}, v_{p}^{*} \rangle \leq 0, \quad \langle (q')^{\perp}, w_{p}^{*} \rangle \leq 0, \quad
				\langle q', v_{p}^{*} \rangle \geq 0, \quad \text{and} \quad \langle q', w_{p}^{*} \rangle \geq 0.
			\end{equation*} 
		In view of formula \eqref{eq:coordinates}, these inequalities immediately imply that $Q_{+,-}(q') \subseteq Q_{p}$ and $Q_{-,+}(q') \subseteq -Q_{p} = Q_{-p}$.  Therefore, by assumption,
			\begin{equation*}
				x^{+,-} \in \inte(Q_{+,-}(q')) \subseteq \inte(Q_{p}), \quad x^{-,+} \in \inte(Q_{-,+}(q')) \subseteq \inte(Q_{-p}).
			\end{equation*}
		
		If instead $\mathcal{Q}_{p} \subseteq Q_{-,+}(q')$, then we argue as in the previous paragraph to find $Q_{+,+}(q') \subseteq Q_{p}$ and $Q_{-,-}(q') \subseteq Q_{-p}$.  Therefore,
			\begin{equation*}
				x^{-,-} \in \inte(Q_{-,-}(q')) \subseteq \inte(Q_{-p}), \quad x^{+,+} \in \inte(Q_{+,+}(q')) \subseteq \inte(Q_{p}).
			\end{equation*}
		\end{proof}

\section{Further remarks and some open problems}\label{sec:conclusion}

\subsection{Higher dimensions}
\begin{figure}
\includegraphics[width=0.25\textwidth]{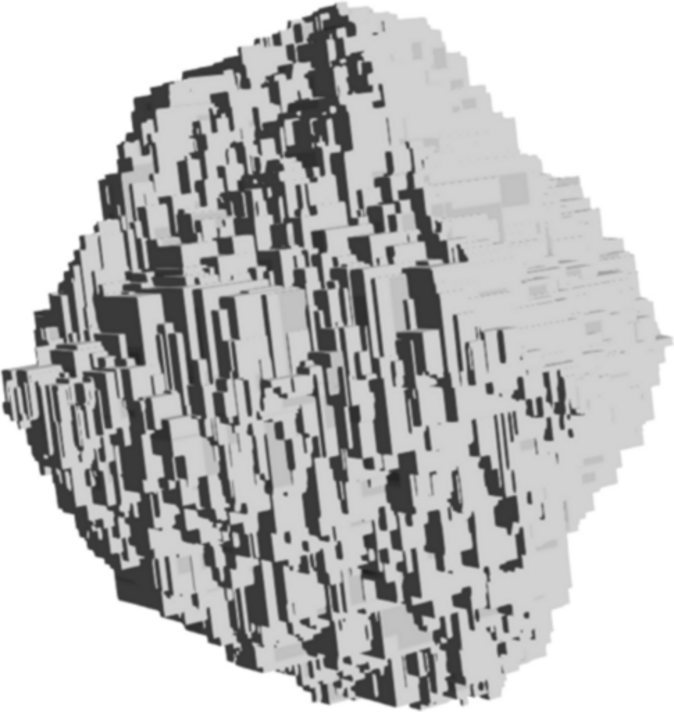} \qquad
\includegraphics[width=0.25\textwidth]{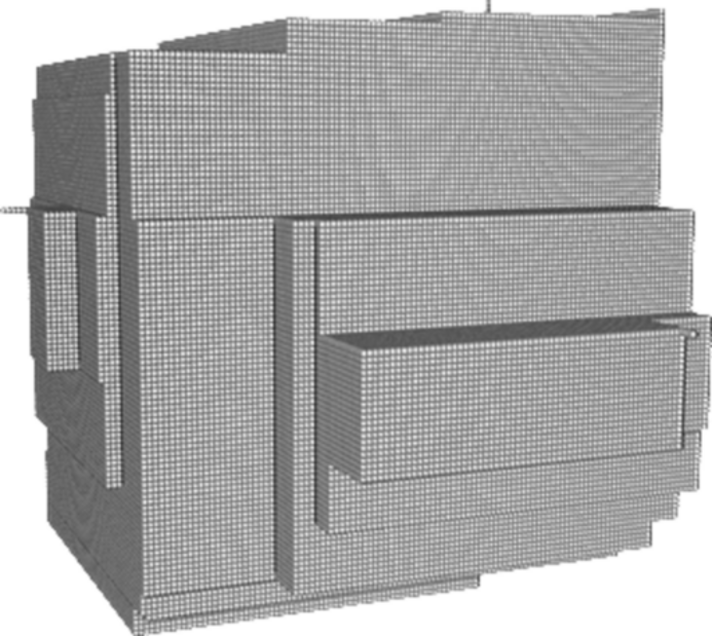}
\caption{Snapshots of $\ell^{\infty}$ (left) and $\ell^{1}$ (right) Pareto peeling of Poisson points in a cube in 3D.}
\end{figure}

A natural followup is to analyze the behavior of Pareto peeling in $\R^d$ for $d > 2$. 
In higher dimensions, only much weaker versions of Theorem \ref{thm:cone_efficient} are available \cite{durier1986sets} and we expect this reflects new phenomena that occur in higher dimensions. 
On the one hand, when $\varphi(\cdot) = \|\cdot\|_{\infty}$, the situation is similar to the two dimensional case.  
The family of cones that describe the Pareto hull 
are rotated quadrants, 
\begin{equation}
Q_{k}^{\pm} = \{ x \in \R^d \, \mid \, \pm x_k = \|x\|_{\infty}\},
\end{equation}
for $k = 1, \ldots, d$. Equivalently, $Q_k$ are cones generated by 0 and facets of the cube $[-1,1]^{d}$.  In this case, it is straightforward to extend the above arguments
to prove the following.

\begin{theorem} \label{theorem:l-inf-convergence}
	If $X_n$ are Poisson point processes in $\mathrm{U}$, Pareto efficient, bounded, and open, with intensities $n$ and $\varphi(x) = \|x\|_{\infty}$ then, on an event of probability 1, the sequence of rescaled height functions
	$n^{-1/d} u_n := \bar{u}_n \to \bar{u}$, where $\bar{u}$ is the unique viscosity solution 
	to the PDE 
	\[
	\begin{cases}
	\max_{k} \left( \prod_j \langle D \bar{u} ,  v_{k,j} \rangle \right) = c_d  &\mbox{ in $\mathrm{U}$} \\
	\bar{u} = 0 &\mbox{ on $\partial \mathrm{U}$}, 
	\end{cases}
	\]
	$\{v_{k,j}\}$ range over the extremal directions of $Q_k^{\pm}$, and $c_d > 0$ is a finite constant. 
\end{theorem}

On the other hand, cones with more complex geometries are also possible in higher dimensions. For example, when $\varphi(\cdot) = \|\cdot\|_1$ in three dimensions, the dynamic programming principle becomes
\begin{equation}
u_A(x) = \inf_{C_{\delta}} \sup_{z \in x - \inte(C_{\delta})} (u_{A}(z) + 1_{A}(z))
\end{equation}
where $\delta$ ranges over $\{ (\pm 1, \mp 1, 0), (\pm 1, 0, \mp 1), (0, \pm 1, \mp 1)\}$
and if, say, $\delta = (1, 1, 0)$	then $C_{\delta} = \R^+ \times \R^+ \times \R$. Importantly, $C_{\delta}$  are convex, but not pointed. However, they are pointed in one dimension lower which leads to the inequality
\[
u_{A}(x_1,x_2,x_3) \leq  \min(h_{A}^{1}(x_1,x_2), h_{A}^{2}(x_1,x_3), h_{A}^{3}(x_2,x_3))
\]
where $h_{A}^{i}$ is $\ell^1$-Pareto peeling in two dimensions and $A$ is projected from $\R^3$ to $\R^2$ in the indicated way. These considerations suggest the following. 
\begin{conj}
	Under the same assumptions as Theorem \ref{theorem:l-inf-convergence}, in dimension $d = 3$, when $\varphi(\cdot) = \|\cdot\|_1$, $n^{-1/2} u_{n}$ almost surely converges to $\bar{h}$ locally uniformly where $\bar{h}$ is the unique viscosity solution to
	\[
	\begin{cases}
	\max( |\bar{h}_{x_1} \bar{h}_{x_2}|, |\bar{h}_{x_1} \bar{h}_{x_3}|, |\bar{h}_{x_2} \bar{h}_{x_3}|) = 1   & \mbox{ in $\mathrm{U}$,} \\
	\bar{h} = 0 &\mbox{ on $\partial \mathrm{U}$.}
	\end{cases}
	\]
\end{conj}	
Of note, this conjecture suggests different scalings for $\ell^1$ and $\ell^{\infty}$-Pareto peeling in dimensions higher than 2. 

\subsection{Other versions of peeling}
In this article, we have only considered one particularly convenient notion of Pareto efficiency.  Our definition of Pareto hull corresponds to what is known in location analysis as a {\it strictly efficient set}
but there are also efficient and weakly efficient sets, which we now discuss. 

Consider $\R^d$ equipped with a norm $\varphi(\cdot)$, let $A \subseteq \R^d$ and denote by $B_{a_i}(x)$ the closed ball of radius $\varphi(x-a_i)$ centered at $a_i$. 
 \begin{enumerate}
\item  The set of {\it efficient} points with respect to $A$ is
\begin{align*}
E(A) = &\{ x \in X \, \mid \, \forall y \not = x, (\exists a \in A, \varphi(a - x) < \varphi(a-y)) \mbox{ or } \\
&(\forall a \in A, \varphi(a-x) \leq \varphi(a-y) ).
\end{align*}
		\item The set of {\it strictly efficient} points with respect to $A$ is
\begin{align*}
c(A) &= \{ x \in \R^d \, \mid \, \forall y \not = x  \\
&\mbox{ there exists $a \in A$  with $\varphi(a-x) < \varphi(a-y)$} \}
\end{align*}
or equivalently $x \in c(A)$ if and only if $\bigcap_{i=1}^n B_{a_i}(x) = \{x\}$.
	\item The set of {\it weakly efficient} points with respect to $A$ is
\begin{align*}
C(A) &= \{ x \in \R^d \, \mid \, \forall y \not = x, \\
&\mbox{ there exists $a \in A$  with $\varphi(a-x) \leq \varphi(a-y)$}\}
\end{align*}
or equivalently $x \in C(A)$ if and only if $\bigcap_{i=1}^n \inte(B_{a_i}(x)) = \emptyset$.

\end{enumerate}

The definitions imply $A \subseteq c(A) \subseteq E(A) \subseteq C(A)$. 
Moreover, one may check that if $A \subseteq B$ then, $c(A) \subseteq c(B)$ and $C(A) \subseteq C(B)$. 
However, counterexamples demonstrate $A \subseteq B$ with $E(A) \not\subset E(B)$ --- see \cite{durier1986sets}. 
The monotonicity of strictly and weakly efficient sets suggest both enjoy scaling limits in general; however, it is not clear how to use these to tightly bound
$E(A)$. Indeed, weakly and strictly efficient peeling may have different scalings as indicated in Section \ref{subsec:weakly-efficient}. 

On the other hand, if $\varphi(\cdot)$ is induced by an inner product or its unit ball is strictly convex and $d = 2$, then $c(A) = C(A) = E(A) = \conv(A)$ \cite{durier1986sets}. Interestingly, Durier-Michelot have an example of a strictly convex norm ball in $d=3$ where $c(A) \not \subset \conv(A)$ --- see Section 4.3 of \cite{durier1986sets}.

\subsection{Weakly efficient peeling in two dimensions and higher} \label{subsec:weakly-efficient}
\begin{figure}
	\includegraphics[width=0.25\textwidth]{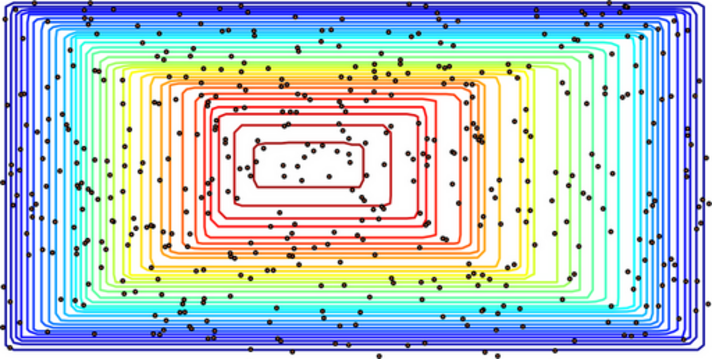} \qquad 
	\includegraphics[width=0.25\textwidth]{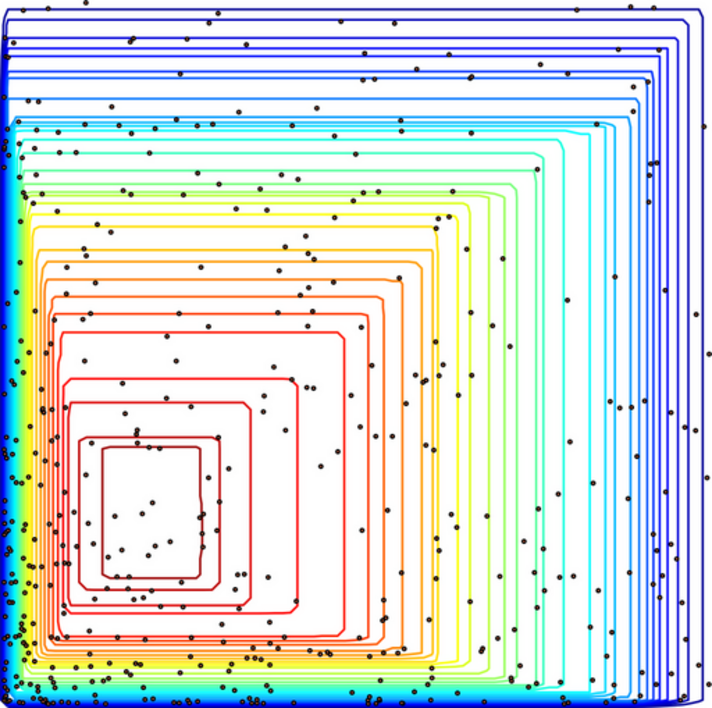}
	\caption{Level sets of weakly-efficient $\ell^1$-peeling for Poisson points in rectangular domains}
	\label{fig:weak-l1-peeling}
\end{figure}
In two dimensions, weakly efficient Pareto hulls are simpler to analyze than strictly efficient sets but the analogous height functions appear to have a different scaling. 
Specifically, for a finite set of points $X \subseteq \R^2$ and a norm $\varphi(\cdot)$, denote the {\it weak Pareto hull} by 
\begin{equation}
W_1(A) = C(A) \quad \mbox{and} \quad W_{n+1}(A) = C(A \cap \inte(W_n(A)))
\end{equation}
and the height function by 
\[
h_{X} = \sum_{n \geq 1} 1_{\inte(W_n(A))}.
\]
Weak Pareto hulls have a simpler inclusion constraint than strictly efficient sets --- see Theorem 4.3 of \cite{durier1986sets} and Theorem 3 of \cite{pelegrin1989determination}. 
For example, when $\varphi(\cdot) = \|\cdot\|_1$ the height function satisfies the dynamic programming principle,  
\[
h_{X}(x) = \inf_{C_i} \sup_{z \in x - \inte(C_i)} (h_{X}(z) + 1_{X}(z))
\]
where 
\[
C_i = \{ x \in \R^2 \, \mid \, \langle x , \xi_i \rangle \geq 0 \}
\]
for 
\[
\xi_1 = (1,0) \quad \xi_2 = (-1,0) \quad \xi_3 = (0,1) \quad \xi_4 = (0,-1).
\]
Equivalently $C(X)$ is the bounding rectangle of $X$, 
\begin{equation}
\br(X) := \{ z \in \R^2 \, \mid \,\mbox{ $m \leq z \leq M$} \},
\end{equation}
where $m_i := \min_{z \in X} z_i$ and $M_i := \max_{z \in X} z_i$ for $i = 1, 2$ and, as in Section \ref{sec:q-nds}, the vector inequalities are pointwise. A straightforward analysis, essentially counting Poisson points, 
leads to the following. 
\begin{example} \label{prop:weakly-l1-efficient}
Let $\varphi(\cdot) = \|\cdot\|_1$. If $X_n$ are Poisson point processes in $[-1,1]^2$ with intensities $n$, then, almost surely, the sequence of rescaled height functions $n^{-1} h_{n}$ converges to 
$\bar{h}$ where 
\begin{equation}
\bar{h}(x) = \min(1-x_1^2, 1- x_2^2).
\end{equation}
\end{example}

One can also check that if $f_i:\R^2 \to \R$ are continuous and invertible, the map 
\begin{equation}
F(x_1,x_2) = (f_1(x_1), f_2(x_2)) 
\end{equation}
`preserves bounding rectangles'. That is if $(x,y) \in \br(A)$ then $F(x,y) \in \br(F(A))$. 
This can be used to extend Example \ref{prop:weakly-l1-efficient} to rectangular domains --- see Figure \ref{fig:weak-l1-peeling}. However, it is not clear if there is a simple description of the limit when the Poisson intensity is not strictly positive in $\br(A)$. 

The situation in higher dimensions again appears to be even more difficult --- the cones describing weakly efficient sets may not always be convex --- see Example 2 in Section 4.2 of \cite{durier1986sets}.

\subsection{Other asymptotic regimes} Comparing our results to those of \cite{calder2020limit}, one observes that Pareto hull peeling has a different asymptotic behavior depending on the character of the norm $\varphi$: if the norm is strictly convex, then the height function scales like $n^{\frac{2}{3}}$, while the presence of even a single facet results in a scaling like $n^{\frac{1}{2}}$.  Further, while both the facets and the halfspaces contribute to the DPP (Proposition \ref{prop:dpp}), the facets dominate in the $n^{\frac{1}{2}}$ scaling regime, so much so that the geometry of the round parts $\mathcal{E}$ do not contribute in any way to the Hamiltonian $\bar{H}_{\varphi}$. 

This leads to a natural question: are there scaling regimes in which both the flat cones $\{Q_{p}\}_{p \in \mathcal{N}^{*}}$ and the round parts $\mathcal{E}$ contribute meaningfully to the limit?  One approach would be to vary the norm along with the intensity $n$ of the Poisson points.

For a specific example, in the discussion that follows, let $\varphi_k$ denote the norm with unit ball given by a regular $k$-gon.  Suppose that we vary both the parameter $k$ as well as $n$, the intensity of the point cloud, having in mind that both $k$ and $n$ are large.

In the extreme case when we first send $k \to \infty$ and then send $n \to \infty$, we recover the same asymptotic behavior as in convex hull peeling.   This is due to an observation that goes back to \cite{thisse1984some}.  In the statement, we write $\mathcal{P}_k$ for the Pareto hull 
with respect to $\varphi_k$. 
\begin{theorem}[Theorem 5 in \cite{thisse1984some}] \label{theorem:convergence-to-convex-hull}
	For each finite point set $A \subseteq \R^2$, it holds that for each sufficiently small $\delta > 0$, 
	there exists $k_0(A, \delta)$ such that,
	\[
	\inte(B_{\delta}^-(\conv(A)) \subset \inte(\mathcal{P}_k(A)) \subset \inte(\conv(A)) \quad \mbox{for all $k \geq k_0$}, 
	\]
	where the notation $B_{\delta}^-$ denotes the inner $\delta$ neighborhood of a set, $B_{\delta}^-(X) = \{ y \in X : B(y, \delta) \subseteq X\}$.
	In particular, for $k$ sufficiently large, $\partial \mathcal{P}_k(A) \cap A = \partial \conv(A) \cap A$. 
\end{theorem}
\begin{proof}
	A version of this result was proved in \cite{thisse1984some} and the result as stated may be deduced directly from Theorem \ref{thm:cone_efficient}.	
\end{proof}
This implies, for a fixed set of points $A$, that the sequence of height functions associated with $\varphi_k$ converges to the height function for convex hull peeling --- see Figure \ref{fig:kgon-converge-to-convex-hull-peeling}.  Thus, if we first send $k \to \infty$ and then send $n \to \infty$, the height functions behave as in convex hull peeling, and, by \cite{calder2020limit}, the relevant PDE is
	\begin{equation} \label{eq:curvature}
		\langle Du, \text{cof}(-D^{2}u) Du \rangle = f(x)^{2}.
	\end{equation}

\begin{figure}
	\includegraphics[width=0.22\textwidth]{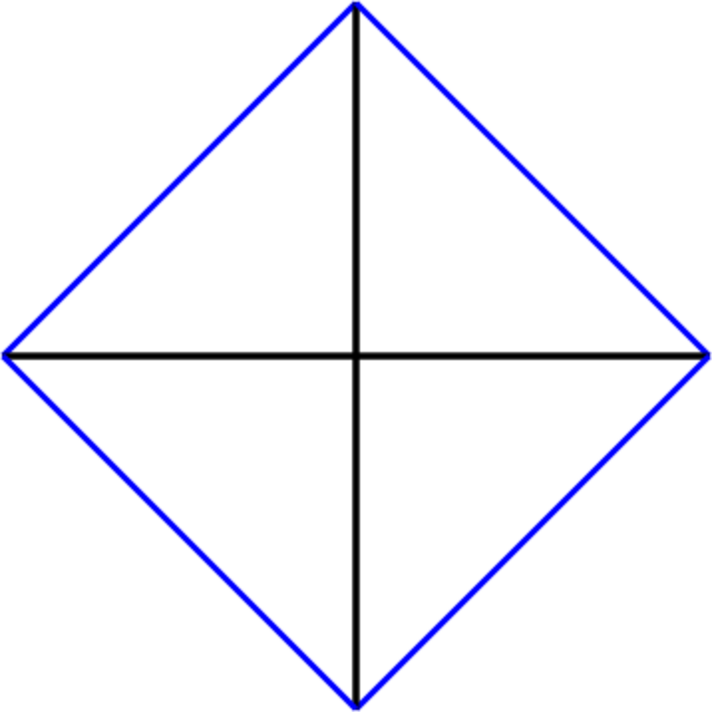}
	\includegraphics[width=0.22\textwidth]{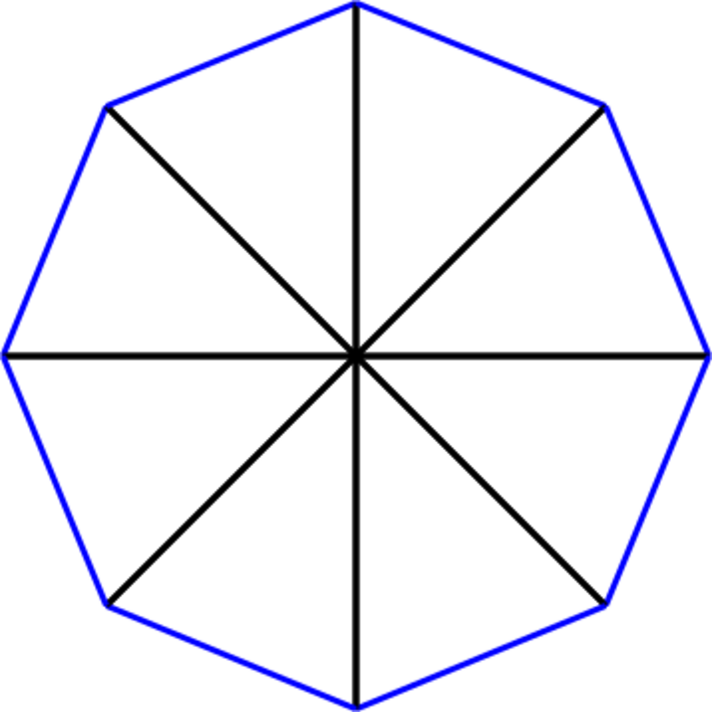}
	\includegraphics[width=0.22\textwidth]{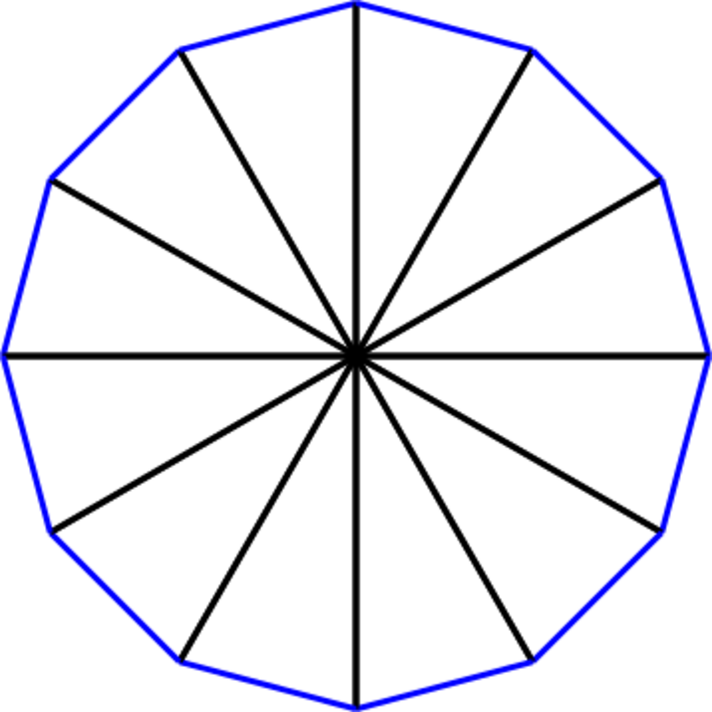} 
	\includegraphics[width=0.22\textwidth]{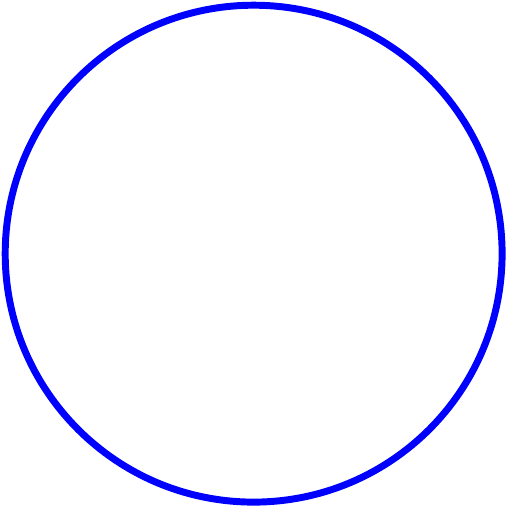} \\

	\includegraphics[width=0.22\textwidth]{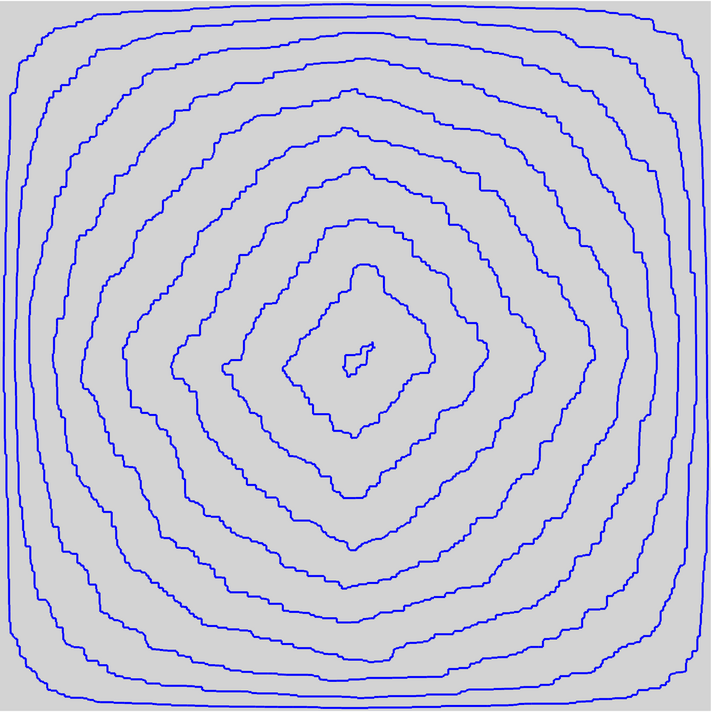}
	\includegraphics[width=0.22\textwidth]{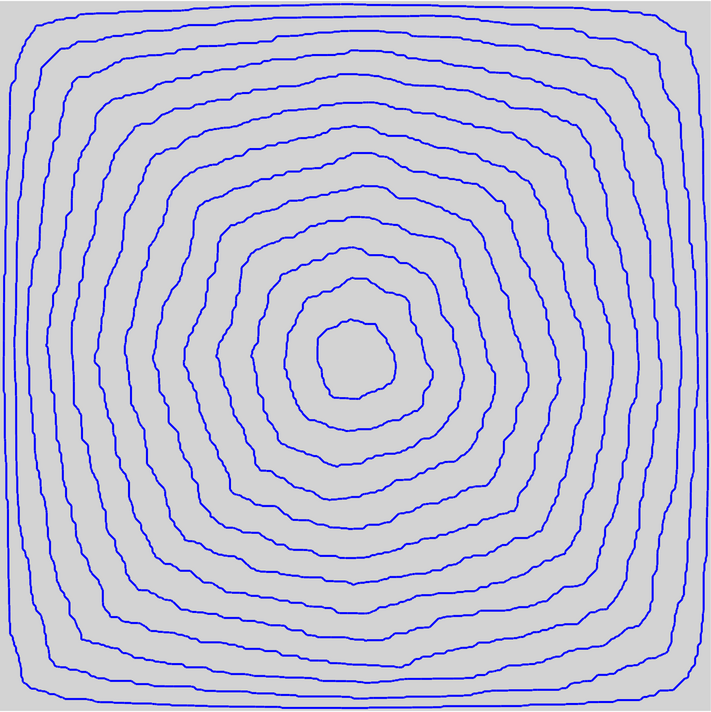}
	\includegraphics[width=0.22\textwidth]{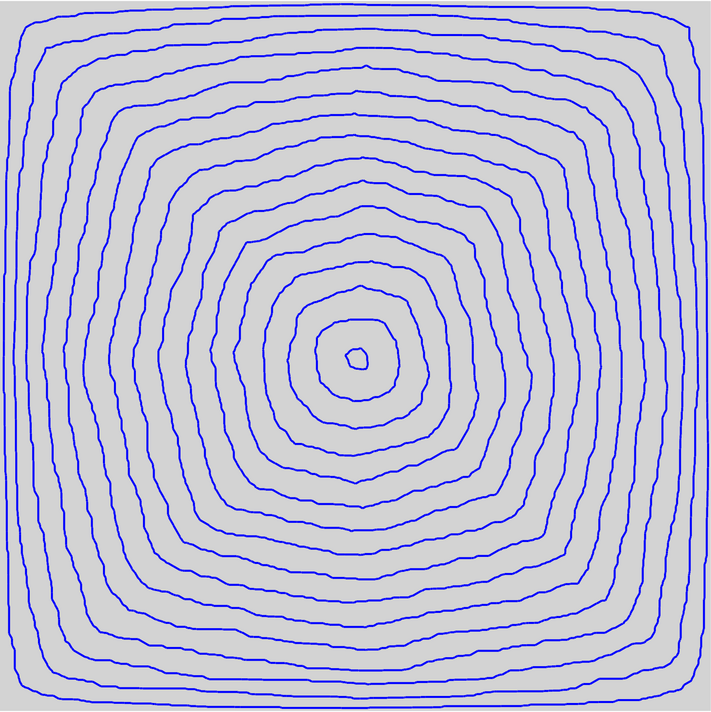}
	\includegraphics[width=0.22\textwidth]{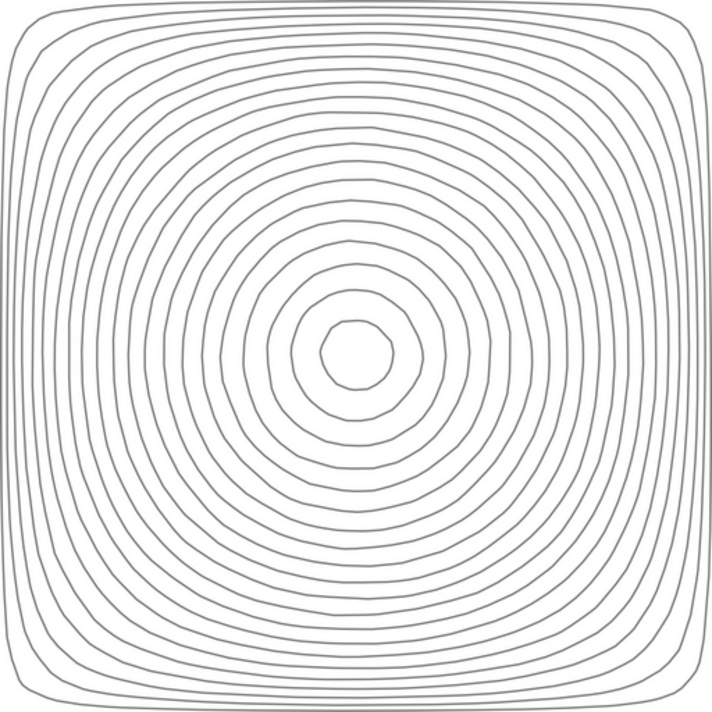}
	
	\caption{The first three columns are Pareto peeling of $10^5$ Poisson random points in a square with respect to the indicated ``k-gon'' norms, $\varphi_k$ for $k = 4,8,12$. The last column is convex hull peeling.}
	\label{fig:kgon-converge-to-convex-hull-peeling}
\end{figure}

On the other hand, at the opposite extreme, the limit that emerges when we first send $n \to \infty$ and then send $k \to \infty$ is less clear.  Toward this end, first, note that the Hamiltonian $\bar{H}_{\varphi_{k}}$ is given by 
\begin{equation} \label{eq:kgon-hamiltonian}
\bar{H}_{\varphi_k}(\xi) =  \max_{j = 0, \ldots, k - 1}  \frac{\langle \xi,v_{j,k} \rangle \langle \xi, w_{j,k} \rangle}{\sin(\frac{2\pi}{k})}
\end{equation}
where 
\[
\begin{aligned}
w_{j,k} = - \left( \cos(\frac{2\pi (j + 1)}{k}), \sin(\frac{2\pi (j + 1)}{k}) \right), \quad v_{j,k} = \left( \cos (\frac{2\pi j}{k}), \sin(\frac{2\pi j}{k}) \right).
\end{aligned}
\]
Since the form of $\bar{H}_{\varphi_{k}}$ is explicit, it should be possible to use it to characterize the behavior of the solution $\bar{u}^{(k)}$ of \eqref{eq:eikonal} in the limit $k \to \infty$.
	
The discussion above suggests that, in general, very different limiting equations might arise in the limit $\min\{k,n\} \to +\infty$ depending on the rate at which $k$ increases relative to $n$.  It would be interesting to study this question in more detail, particularly to determine whether or not there is a choice of $k = k(n)$ so that both first- and second-order terms appear in the limiting PDE.

A related alternative approach would be to add small, $n$-dependent facets to the Euclidean norm $\|\cdot\|_{2}$.  For instance, given a small angle $\theta \in (0,\pi)$, consider the norm $\varphi_{\theta}$ given by 
	\begin{equation*}
		\varphi_{\theta}(v) = \max\{\|v\|_{2}, \cos(2^{-1}\theta)^{-1} |v_{2}|\}.
	\end{equation*}
In this case, the boundary curve $\{\varphi_{\theta} = 1\}$ is almost identical to the unit circle, except for two line segments of length $2 \sin(\frac{\theta}{2})$ that cross the coordinate axes perpendicularly.    If we choose a sequence $(\theta_{n})_{n \in \mathbb{N}}$ such that $\theta_{n} \to 0$ as $n \to \infty$, then the facets get smaller and smaller so, as in the last example, it is not obvious a priori what the limiting behavior is or how it depends on the choice of sequence.  Again, with an eye toward deriving PDE involving both first- and second-order terms, such examples may be of interest.  See Figures \ref{fig:mixed-kgon-manyfacets} and \ref{fig:mixed-peeling-one-facet} for some simulations in this direction; as elsewhere in the paper, the red color indicates the influence of the round parts $\mathcal{E}$.

\begin{figure}
	\includegraphics[width=0.22\textwidth]{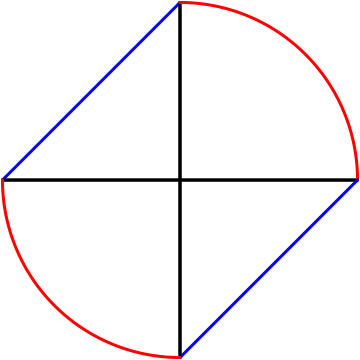}
	\includegraphics[width=0.22\textwidth]{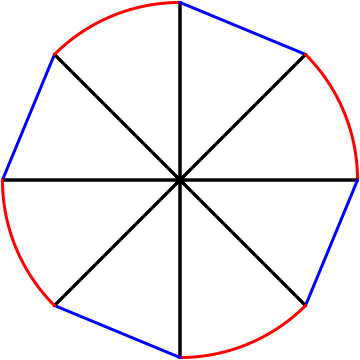}
	\includegraphics[width=0.22\textwidth]{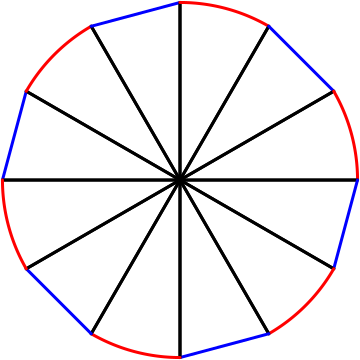}
	
	\includegraphics[width=0.22\textwidth]{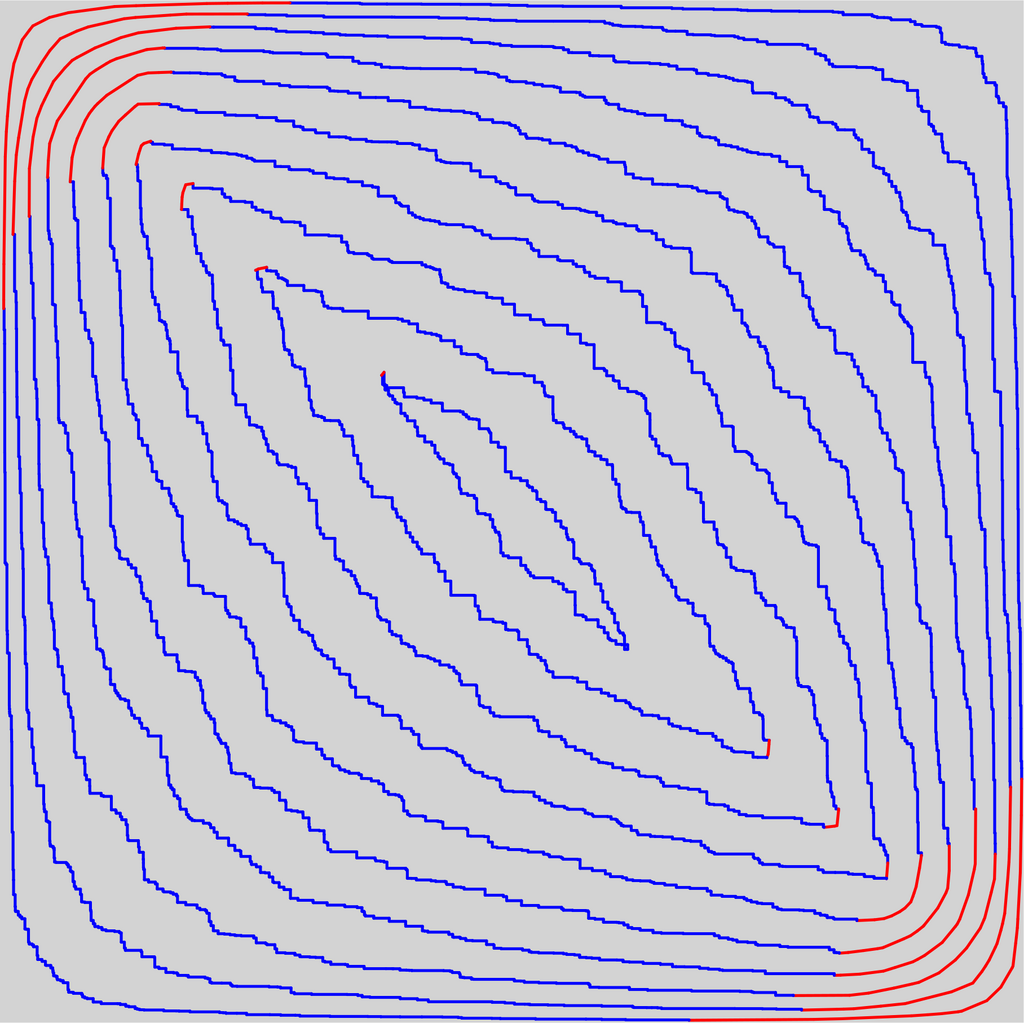}
	\includegraphics[width=0.22\textwidth]{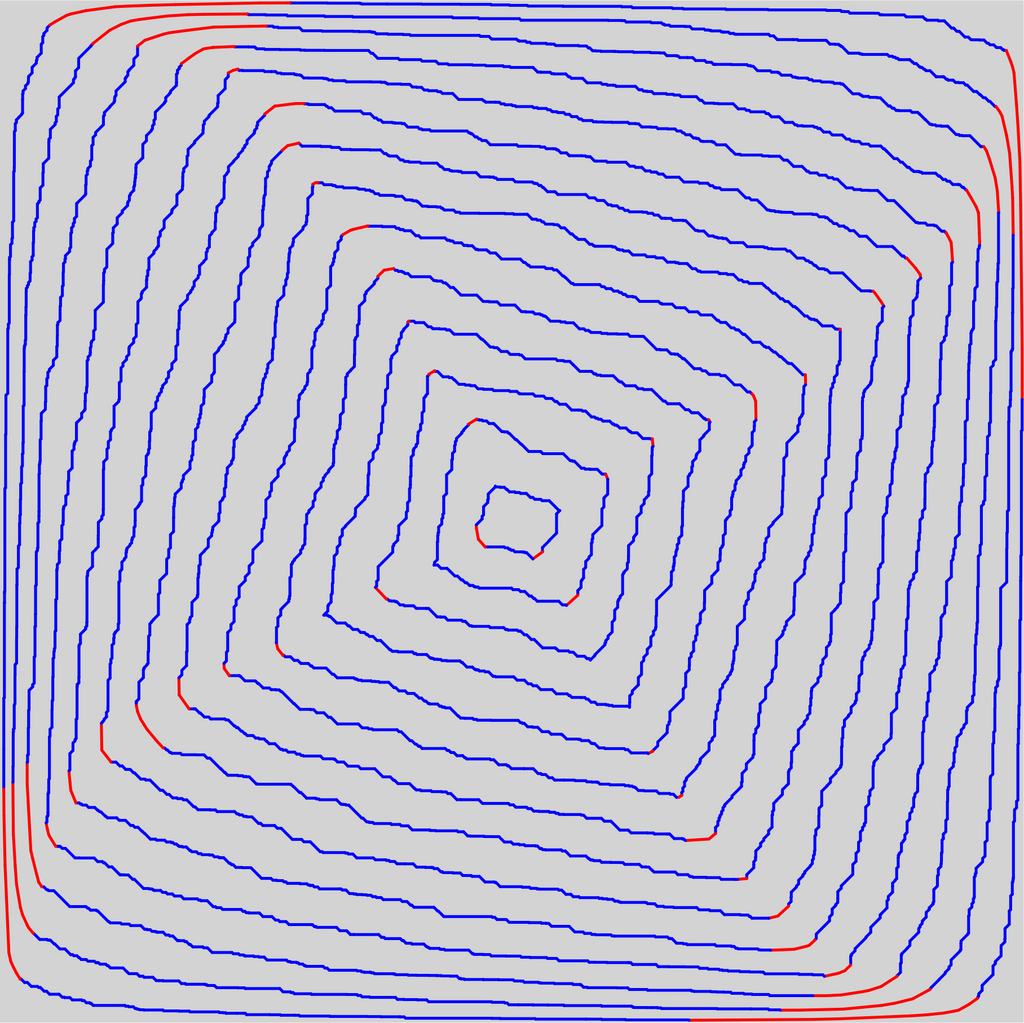}
	\includegraphics[width=0.22\textwidth]{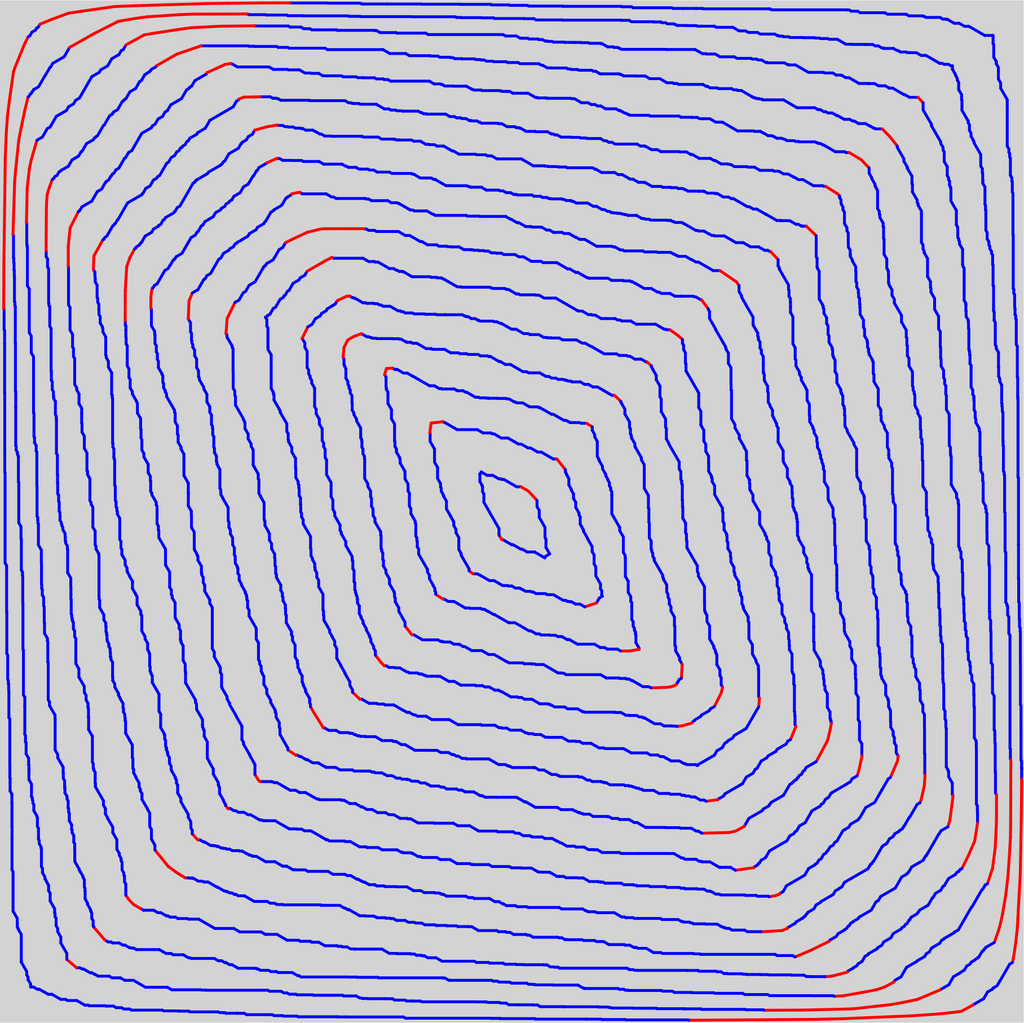}
	
	\caption{Pareto peeling of $10^5$ Poisson random points in a square with respect to the indicated ``mixed k-gon'' norms with $\frac{k}{2}$ facets, $\varphi^{(mixed)}_k$ for $k = 4,8,12$. }
	\label{fig:mixed-kgon-manyfacets}
\end{figure}

\begin{figure}
	\includegraphics[width=0.22\textwidth]{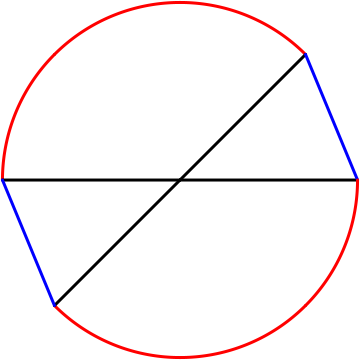}
	\includegraphics[width=0.22\textwidth]{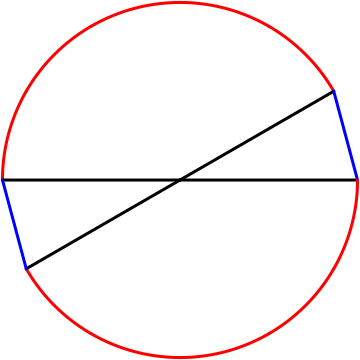}
	\includegraphics[width=0.22\textwidth]{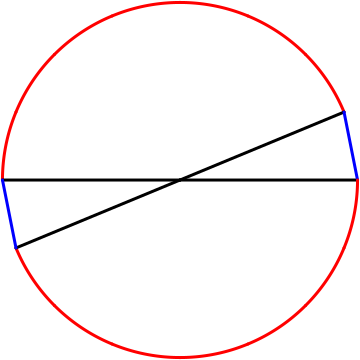}
	
	\includegraphics[width=0.22\textwidth]{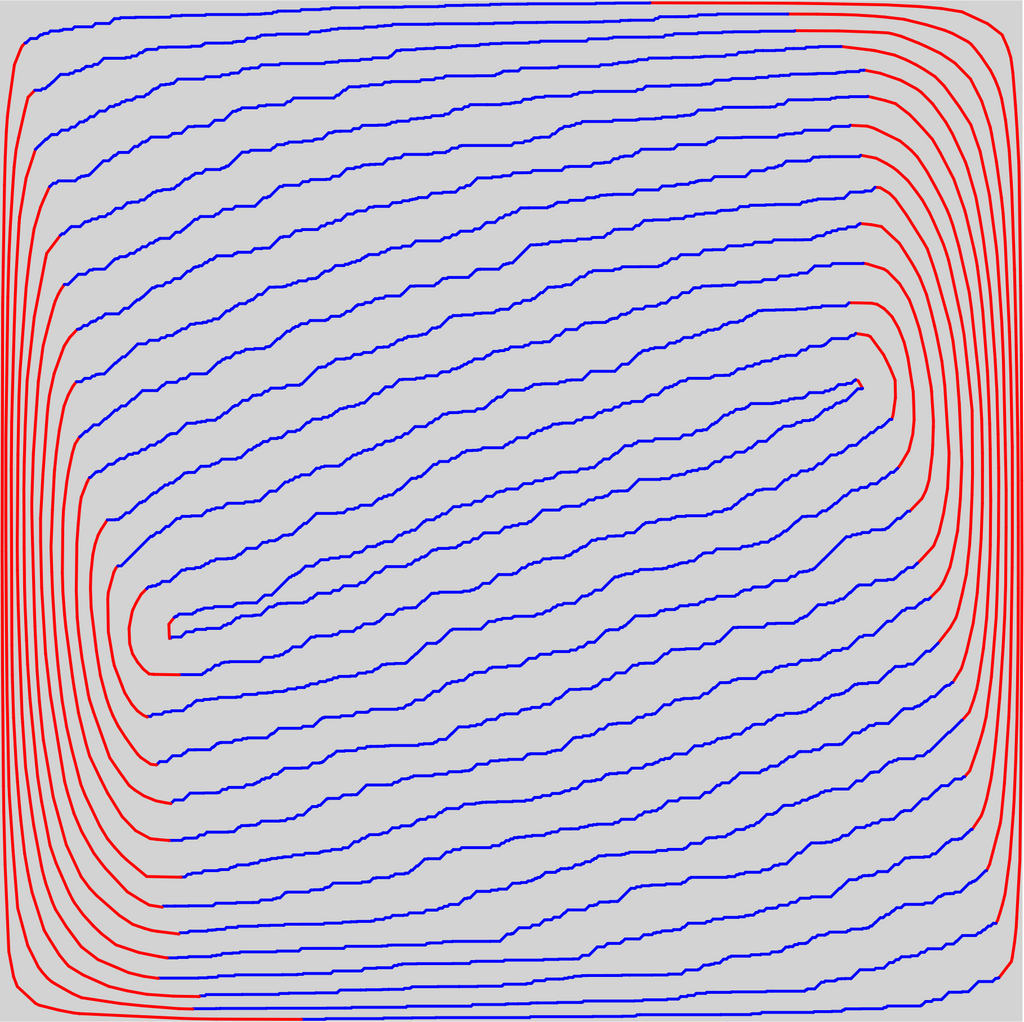}
	\includegraphics[width=0.22\textwidth]{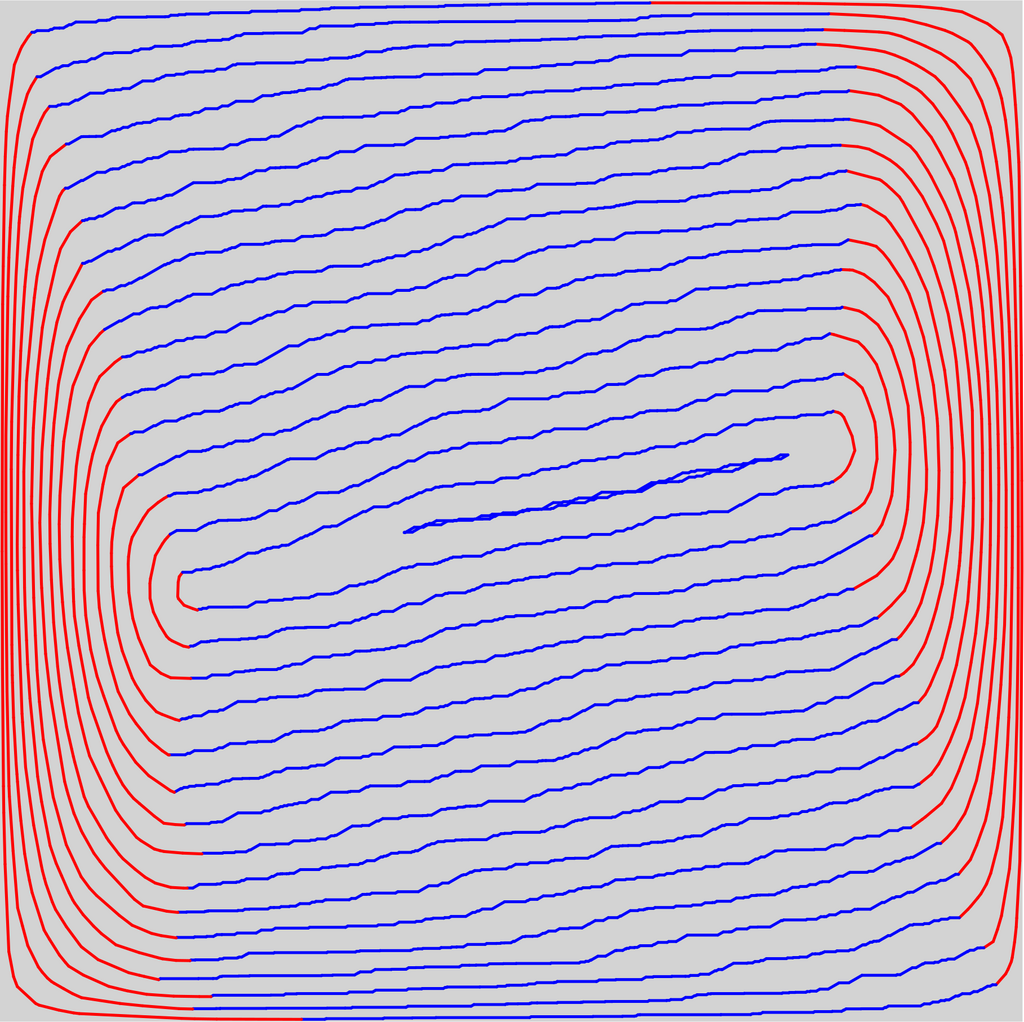}
	\includegraphics[width=0.22\textwidth]{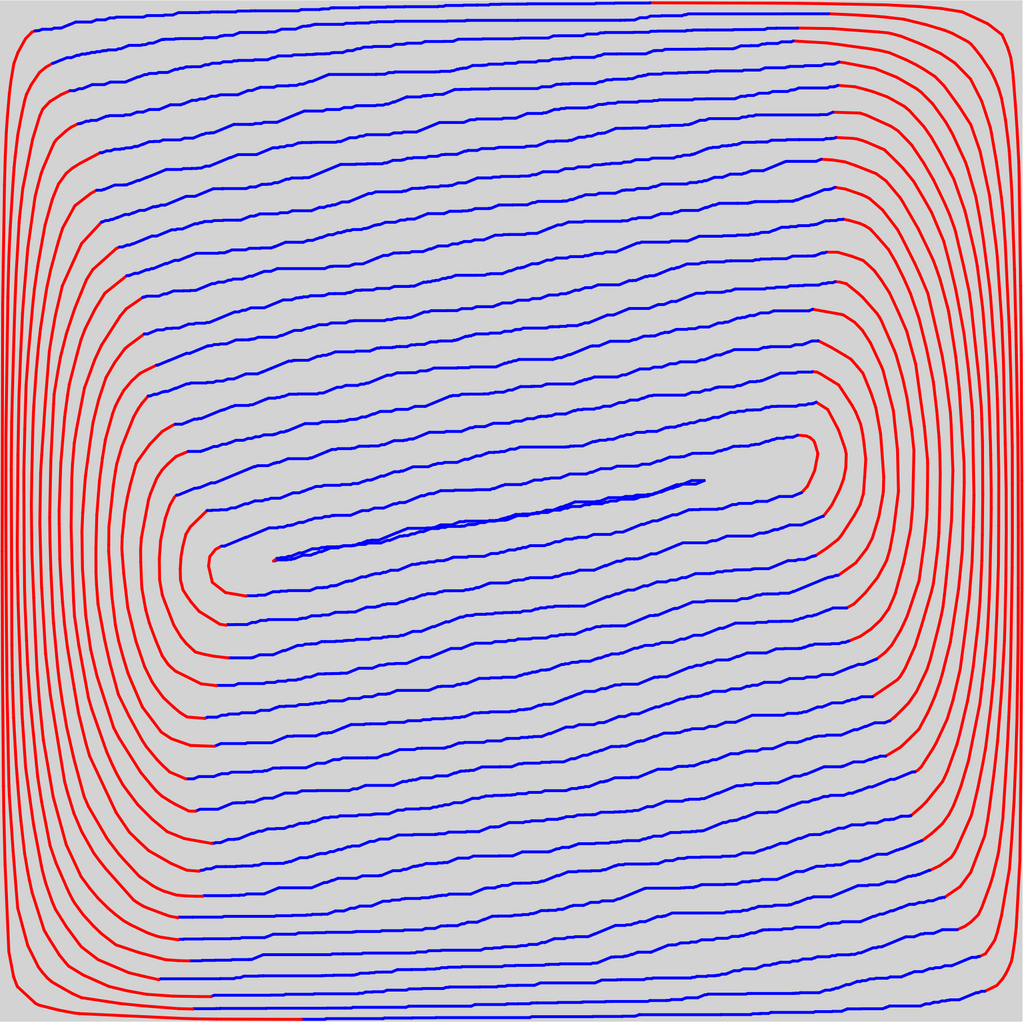}
	
	\caption{Pareto peeling of $10^5$ Poisson random points in a square with respect to the indicated ``mixed'' norms with just two small facets, $\varphi_{\theta}$ with opening angles $\theta = \frac{2\pi}{k}$, where $k = 8,12,16$. }
	\label{fig:mixed-peeling-one-facet}
\end{figure}

\bibliography{peeling_2d.bib}{}
\bibliographystyle{amsalpha}

\end{document}